\documentclass[reqno, 12pt]{amsart}

\usepackage{color}
\usepackage{mathrsfs}
\usepackage{amscd}
\usepackage{amsmath}
\usepackage{latexsym}
\usepackage{amsfonts}
\usepackage{amssymb}
\usepackage{amsthm}
\usepackage{graphicx}
\usepackage{hyperref}
\usepackage{makecell}
\usepackage{array,color}
\usepackage{booktabs}
\usepackage{multirow}

\newtheorem{theorem}{Theorem}[section]
\newtheorem{proposition}[theorem]{Proposition}

\newtheorem{lemma}[theorem]{Lemma}
\newtheorem{definition}[theorem]{Definition}

\newtheorem{remark}[theorem]{Remark}

\newcommand\R{\mathbb{R}}

\newcommand{\cqd}{\hfill$\Box$}

\numberwithin{equation}{section}

\title[Decay estimates for higher order Beam equations]
{Decay estimates for Beam equations with potential in dimension three}

\author{Miao Chen, Ping Li, Avy Soffer and Xiaohua Yao\textsuperscript{\dag }  }

\address{Miao Chen, Department of Mathematics,  Central China Normal University, Wuhan, 430079, P.R. China}
\email{miaochen@mails.ccnu.edu.cn}

\address{Ping Li, School of Informations  and  Mathematics, Yangtze University, Jingzhou 434000, China}
\email{liping@whu.edu.cn}

\address{Avy Soffer, Mathematics Department, Rutgers University, New Brunswick, NJ, 08903, USA}
\email{soffer@math.rutgers.edu}

\address{Xiaohua Yao, Department of Mathematics and Key Laboratory of Nonlinear Analysis and Applications(Ministry of Education), Central China Normal University, Wuhan, 430079, P.R. China }
\email{yaoxiaohua@ccnu.edu.cn}

\thanks{\textsuperscript{\dag} Corresponding author}


\begin{document}

\begin{abstract}\baselineskip=13pt
This paper is devoted to studying time decay estimates of the solution for Beam equation (higher order type wave equation) with a potential
 $$u_{t t}+\big(\Delta^2+V\big)u=0,  \,\ u(0, x)=f(x),\  u_{t}(0, x)=g(x)$$
in dimension three, where  $V$ is a real-valued and decaying potential.  Assume that zero is a regular point of $H= \Delta^2+V $,  we first prove the following optimal time decay estimates of the solution operators
\begin{equation*}
	\big\|\cos (t\sqrt{H})P_{ac}(H)\big\|_{L^{1} \rightarrow L^{\infty}} \lesssim|t|^{-\frac{3}{2}}\   \ \hbox{and} \  \
	\Big\|\frac{\sin(t\sqrt{H})}{\sqrt{H}} P_{a c}(H)\Big\|_{L^{1} \rightarrow L^{\infty}} \lesssim|t|^{-\frac{1}{2}}.
\end{equation*}
Moreover, if zero is a resonance  of $H$,  then time decay of  the  solution operators also is considered.  It is noted that a first-kind resonance does not affect the decay rates of  the propagator operators $\cos(t\sqrt{H})$  and $\frac{\sin(t\sqrt{H})}{\sqrt{H}}$,
 but their decay  will be significantly  changed for the second and third-kind resonances.
\end{abstract}
		\maketitle

	\baselineskip=15pt
\section{Introduction and main results}
\subsection{Introduction}\label{sbusection1}
In this paper we  consider the following Beam equation (higher order type wave equation)  with  a real-valued decaying potential in dimension three:
\begin{equation}\label{cauchyequation1}
\begin{cases}
u_{t t}+\big(\Delta^2+V(x)\big)u=0, \  \  \ (t, x) \in \mathbb{R} \times \mathbb{R}^{3}, \\
 u(0, x)=f(x),\ \
  u_{t}(0, x)=g(x).
    \end{cases}
\end{equation}
Let $H:=\Delta^2+V$ and  $|V(x)|\lesssim \langle x \rangle^{-\beta}$ for some $\beta >0$, $\langle x\rangle= \sqrt{1+|x|^2}$. Then $H$ is self-adjoint on $L^2$, and  the solution to the  equation \eqref{cauchyequation1} is given by the formula
\begin{equation}\label{equation-solution}
 u(t,x)=\cos (t \sqrt{H})f(x) + \frac{\sin (t \sqrt{H})}{\sqrt{H}}g(x).
\end{equation}
The expression \eqref{equation-solution} above depends on the branch chosen of $\sqrt{z}$ with $\Im \sqrt{z}\ge 0$, so  the solution $u(t,x)$  is well-defined even if $H$ is not positive.
We will study the decay estimates of the propagator operators $\cos (t \sqrt{H})$ and  $\frac{\sin (t \sqrt{H})}{\sqrt{H}}$.

In the free case, i.e.  $V=0$ and $\sqrt{H}=-\Delta $,  if the initial data $(f, g) \in L^1(\mathbb{R}^3)\times L^1(\mathbb{R}^3)$, then the solution operators of
 the equation \eqref{cauchyequation1} satisfy the following optimal time decay  estimates:
 \begin{equation}\label{decayestimate-freecase1}
\big \|\cos (t \Delta)f \big\|_{L^\infty(\mathbb{R}^3)}\lesssim |t|^{-\frac{3}{2}}\|f\|_{L^1(\mathbb{R}^3)},
\end{equation}
and
 \begin{equation}\label{decayestimate-freecase2}
\Big\|\frac{\sin (t \Delta)}{\Delta}g \Big\|_{L^\infty(\mathbb{R}^3)} \lesssim |t|^{-\frac{1}{2}}\|g\|_{L^1(\mathbb{R}^3)},
\end{equation}
see Theorem \ref{thm-main-results-resonance} below, and  also  see \cite[Theorem 7] {CMY}.

When $V\neq0$,  the decay estimates  of the solution operators of  the equation (\ref{cauchyequation1}) are affected by the spectrum of $H$,
which in turn depend on the conditions of potential $V$.
It is well-known that the spectrum of $H$ has negative eigenvalues  $\{\lambda_1\leq\lambda_2\leq \cdots <0 \}$ and continuous spectrum $[0, \infty)$ provides that a potential $V$ has certain decay rate, see e.g. \cite{Agmon, RS}.  If $H$ has no positive eigenvalues embedded into the continuous spectrum, then $(0,\infty)$ is the pure absolutely continuous spectrum.  In this case,  we assume that $\lambda_j(j\geq1)$ are  negative eigenvalues of the spectrum of $H$ (the counting multiplicity of $\lambda_j$) and $H \phi_j=\lambda_j\phi_j(j\geq 1) $ for $\phi_j \in L^2(\mathbb{R}^3)$. Let   $P_{ac}(H)$ denote the projection onto the absolutely continuous spectrum space of $H$. Then the solution of  the equation \eqref{cauchyequation1} can be written as
\begin{equation*}\label{wave-equation-solution}
   \begin{split}
u(t,x):=u_d(t,x)+u_c(t,x),
   \end{split}
\end{equation*}
where
\begin{equation}\label{solution-ud}
 u_d(t,x)=\sum_{j} \cosh( t\sqrt{-\lambda_j})(f,\phi_j)\phi_j(x) +\frac{\sinh (t\sqrt{-\lambda_j})}{\sqrt{-\lambda_j}}(g,\phi_j)\phi_j(x),
\end{equation}
\begin{equation}\label{solution-uc}
 u_c(t,x)=\cos (t \sqrt{H})P_{ac}(H)f(x) + \frac{\sin (t \sqrt{H})}{\sqrt{H}}P_{ac}(H)g(x).
\end{equation}
Obviously, the negative eigenvalues of the spectrum of $H$ results in the exponential growth of $u_d(t,x)$ defined in (\ref{solution-ud}) as $t$ becomes large. Hence we need to  remove  the discrete part   due to the existences of  negative eigenvalues, and then focus on decay estimates for the continuous spectrum part $u_c(t,x)$. In particular,
we notice that the absence of positive eigenvalue of $H$ has been an indispensable assumption in deriving of dispersive estimates,  see Subsection \ref{more-backgrpund}  for more comments on  the positive  eigenvalues of $H$.

In this paper, we are devoted to establishing the time decay bounds of  the solution (\ref{solution-uc}) with certain decay conditions  of potential $V$ in  dimension three.
For this purpose, we first study the asymptotic expansions
of the resolvent $R_V(\lambda^4)$  for $\lambda$ near zero in the presence of resonances or eigenvalue, and then establish the desired time decay  bounds by using  Stone's formula, Littlewood-Paley method and oscillation integral theory.

\subsection{Main results.}
We use the notation $a\pm:=a\pm\epsilon$ for some small but fixed $\epsilon>0$. For $A, B\in \mathbb{R}^+$, $A \lesssim B$  means that there exists some constant $C>0$ such that $A\leq C B$.

We first consider the case that zero is a regular point of  $H$ ( i.e.  $H$  has neither zero eigenvalue nor zero resonance).
\begin{theorem}\label{thm-main-results-regular}
Assume that $|V(x)|\lesssim (1+|x|)^{-\beta}$ for  $x\in \mathbb{R}^3$ with $\beta >7$, and that $H=\Delta^2+V(x)$ has no positive embedded eigenvalues. Let $P_{ac}(H)$ denotes the projection onto the absolutely continuous spectrum space of $H$. If zero is a regular point of $H$, then
\begin{equation}\label{main-results-regular1}
\big\|\cos (t\sqrt{H})P_{ac}(H)\big\|_{L^{1} \rightarrow L^{\infty}} \lesssim\ |t|^{-\frac{3}{2}},
\end{equation}
and
\begin{equation}\label{main-results-regular2}
\Big\|\frac{\sin(t\sqrt{H})}{\sqrt{H}} P_{a c}(H)\Big\|_{L^{1} \rightarrow L^{\infty}} \lesssim\ |t|^{-\frac{1}{2}}.
\end{equation}
\end{theorem}
\begin{remark}\label{remark-Thm1}
	Some remarks on Theorem \ref{thm-main-results-regular} are given as follows:
\begin{itemize}
\item[(i)] In view of the free decay estimates (\ref{decayestimate-freecase1}) and (\ref{decayestimate-freecase2}) (also see Proposition \ref{prop-free estimates} below),  the estimates (\ref{main-results-regular1}) and (\ref{main-results-regular2}) with $V\neq 0$ are actually optimal in the case that zero is a regular point of $H$.
 \vskip0.2cm

\item [(ii)]By the spectrum theorem of self-adjoint operator,  it immediately follows that
$$\big\|\cos (t\sqrt{H})P_{ac}(H)f\big\|_{L^2} +\Big \|\frac{\sin(t\sqrt{H})}{t\sqrt{H}}P_{ac}(H)f\Big \|_{L^2}\lesssim \|f\|_{L^2}.$$
Hence by using  (\ref{main-results-regular1}), (\ref{main-results-regular2}) and Riesz-Thorin interpolation theorem,  one can obtain the following $L^p$-$L^{p'}$ estimates
\begin{equation}\label{Lp-estimate}
\big \|\cos (t \sqrt{H})P_{ac}(H)f \big\|_{L^{p'}}+\Big\|\frac{\sin (t \sqrt{H})}{t\sqrt{H}}P_{ac}(H)f \Big\|_{L^{p'}} \lesssim  |t|^{3(\frac{1}{p}-\frac{1}{2})}\|f\|_{L^p}.
\end{equation}
where $1\leq p\leq 2$ and $\frac{1}{p}+\frac{1}{p'}=1$.
 \vskip0.2cm
\item[(iii)]Recently, Goldberg and Green \cite{GG21b}  have proved that the following wave operators
\begin{align}\label{def-wave}
	W_\pm=W_\pm(H,\Delta^2) :=s-\lim_{t\to\pm\infty}e^{itH}e^{-it\Delta^2}
\end{align}
are bounded on $L^p(\mathbb{R}^3)$ for $1<p<\infty$ if zero is regular point of $H=\Delta^2+V$ with $|V(x)|\lesssim \langle x \rangle ^{-12-}$.
Note that $W_\pm$ satisfy the following intertwining identity:
\begin{align}
	\label{intertwining_1}
	f(H)P_{ac}(H)=W_\pm f(\Delta^2)W_\pm^*,
\end{align}
where $f$ is any Borel measurable function on $\mathbb{R}$. By the $L^p$-boundedness of $W_\pm$ and $W_\pm^*$, one can transfer the $L^p$-$L^{p'}$ estimates of $f(H)P_{ac}(H)$ to the same estimates of $f(\Delta^2)$ by
\begin{align}
	\label{Lp-bound of f(H)}
	\|f(H)P_{ac}(H)\|_{L^p\to L^{p'}}\le \|W_\pm\|_{L^{p'}\to L^{p'}}\ \|f(\Delta^2)\|_{L^p\to L^{p'}}\ \|W_\pm^*\|_{L^{p}\to L^{p}}.
\end{align}
  Hence  we have obtained the same $L^p$-$L^{p'}$ estimates \eqref{Lp-estimate} for any $1<p\le 2$. However, due to the absence of the $L^1$ and $L^\infty$ boundednesses of wave operators $W_\pm$ above  in \cite{GG21b}, also see \cite{MWY23} for the counterexamples of  the endpoint boundedness.  Therefore  the time decay estimates in Theorem \ref{thm-main-results-regular} can not be obtained by wave operator methods.
\end{itemize}
\end{remark}


As is known,  the time decaying rate may change for many different types of dispersive estimates if zero  resonance or eigenvalue arises, see for example \cite{JK, ES1} for three-dimensional and \cite{Erdogan-Green} for two-dimensional Schr\"{o}dinger operators.  For fourth order Schr\"{o}dinger operators, for example, cf.  \cite{EGT19, GT19,LSY21}. For the classical wave equation with potential \eqref{waveequation-twoorder}, it was  shown that the first kind of zero resonance does not destroy the  time decay rate,  but the second kind of zero resonance decreases the decay rate,  e.g. cf.  \cite{Green14, EGG14}.

Now we  will present the decay estimates in the presence of zero resonance or eigenvalue.  We first  define the classical weighted spaces $$L^2_{-\sigma}(\mathbb{R}^3)= \big\{f\in L^2_{loc}(\R^3)~| \ \langle x\rangle^{-\sigma} f(x)\in L^2(\mathbb{R}^3)\big \}$$ where $\sigma\in\mathbb{R}$, and then give the definitions of zero resonances as follows, also see Section \ref{low energy} for their equivalent definitions and characterizations.

We say that zero is {\it the first kind  resonance of  $H$} if there exists a nonzero $\phi\in L^2_{-\sigma}(\mathbb{R}^3)$ for  any $\sigma>\frac{3}{2}$ but no
nonzero $\phi\in L^2_{-\sigma}(\mathbb{R}^3)$ for any $\sigma>\frac{1}{2}$ such that $H\phi=0$ in the distributional sense;
zero is  {\it the second kind resonance of  $H$} if there exists a  nonzero $ \phi\in L^2_{-\sigma}(\mathbb{R}^3)$ for any $\sigma>\frac{1}{2}$  but no  nonzero
$\phi\in L^2(\mathbb{R}^3)$ such that $H\phi=0$;
zero is {\it the third kind resonance of $H$ (i.e. eigenvalue) }  if there exists a nonzero $\phi\in L^2(\mathbb{R}^3)$ such that $H\phi=0$.
We remark that such resonance solutions of  $H\phi=0$ also can be characterized in the form of $L^p$ spaces (cf.  \cite{EGT19}).

\begin{theorem}\label{thm-main-results-resonance}
Assume that $|V(x)|\lesssim (1+|x|)^{-\beta}$ for $x\in\mathbb{R}^3$ with some $\beta>0$,  and  that $H=\Delta^2+V(x)$ has no positive embedded eigenvalues. Let $P_{ac}(H)$ denote the projection onto the absolutely continuous spectrum space of $H$. Then the following statements hold:
\begin{itemize}
\item[(i)]  If zero is the first kind resonance of $H$ and $\beta>11$, then
\begin{equation}\label{main-results-firstkind1}
\big\|\cos (t\sqrt{H})P_{ac}(H)\big\|_{L^{1} \rightarrow L^{\infty}} \lesssim|t|^{-\frac{3}{2}},
\end{equation}
and
\begin{equation}\label{main-results-firstkind2}
\Big\|\frac{\sin(t\sqrt{H})}{\sqrt{H}} P_{a c}(H)\Big\|_{L^{1} \rightarrow L^{\infty}} \lesssim|t|^{-\frac{1}{2}}.
\end{equation}

\item[(ii)] If zero is the second kind resonance of $H$ and $\beta>19$,  or  the third kind resonance of the  $H$ and $\beta>23$, then
\begin{equation}\label{main-results-secondkind1}
\big\|\cos (t\sqrt{H})P_{ac}(H) \big\|_{L^{1} \rightarrow L^{\infty}} \lesssim|t|^{-\frac{1}{2}}.
\end{equation}
Moreover, there exist  two finite rank operators $F_{t}$ and $G_t$ satisfying
$$\|F_{t}\|_{L^1\rightarrow L^\infty} \lesssim |t|^{-\frac{1}{2}}  \ \  \hbox{and}\  \ \|G_{t}\|_{L^1\rightarrow L^\infty} \lesssim |t|^{\frac{1}{2}},$$
 such that
\begin{equation}\label{main-results-secondkind2}
\big\|\cos(t\sqrt{H}) P_{a c}(H)-F_{t}\big\|_{L^1 \rightarrow L^{\infty}} \lesssim|t|^{-\frac{3}{2}},
\end{equation}
and
\begin{equation}\label{main-results-secondkind3}
\Big\|\frac{\sin(t\sqrt{H})}{\sqrt{H}} P_{a c}(H)-G_{t}\Big\|_{L^1 \rightarrow L^{\infty}} \lesssim|t|^{-\frac{1}{2}}.
\end{equation}

\end{itemize}
\end{theorem}

 \begin{remark}\label{remark-thm2} Some further remarks on Theorem \ref{thm-main-results-resonance} are given as follows:
 	\begin{itemize}
 \item[(i)]
 We  notice that the decay assumption of potential $V$ mainly depends on the asymptotic expansion of Theorem  \ref{thm-main-inver-M}  at zero threshold,  and may be not optimal even  for the regular case.  In general since the dimension is less than or equal to the order of $H$, we  guess that decay rate $\beta>4$  is expected at best  in the regular case for  bounded potentials (even local singularity can be allowed, e.g. $\langle x\rangle ^\beta V\in L^p$ for some  $1\le p<\infty$),   cf.  \cite{MY2021} for sharp Hardy potentials. On the other hand, if dimension is larger than the order of $H$,  certain smoothness of potential must be required,  e.g. refer to  \cite{EG21} and \cite{EGG23} for   $L^p$ bounds of wave operators.
 \vskip0.2cm

\item[(ii)]As shown in Theorem \ref{thm-main-results-resonance} above, if  zero is a first-kind resonance, then the decay estimates of the solution operators $\cos(t\sqrt{H})$  and $\frac{\sin(t\sqrt{H})}{\sqrt{H}}$ are the same as in  the regular case, except  that  a faster decay rate for potentials $V$ is required. In addition, we note that the time decay rates of  solution operators $\cos(t\sqrt{H})$ and $\frac{\sin(t\sqrt{H})}{\sqrt{H}}$  are always optimal for high energy part, see Theorem \ref{thm-high-1} below.
 \vskip0.2cm

\item[(iii)]The existence of second and third-kind zero resonances have altered  time decay rates of  the two solution operators.
  In particular,  the $L^1$-$L^\infty $ estimate of
 $\frac{\sin(t\sqrt{H})}{\sqrt{H}}$  becomes worse  and even  show a  positive growth  as $O(|t|^{\frac{1} {2}})$  for  time $t$ goes to infinite if  zero is a second or third-kind resonance.   Moreover,  by  using Remark \ref{remark-Thm1}(ii),  \eqref{main-results-secondkind2} and  \eqref{main-results-secondkind3}, it follows by interpolation that for $1\le p\le 2$,
\begin{equation}\label{Lp-estimateII}
	\big \|\cos (t \sqrt{H})P_{ac}(H)f \big\|_{L^{p'}}+\Big\|\frac{\sin (t \sqrt{H})}{t\sqrt{H}}P_{ac}(H)f \Big\|_{L^{p'}} \lesssim  |t|^{-(\frac{1}{p}-\frac{1}{2})}\|f\|_{L^p},
\end{equation}
Comparing with the estimate \eqref{Lp-estimate}, we see  that the decay bound $O( |t|^{-(\frac{1}{p}-\frac{1}{2})})$ is not  optimal for all $1\le p< 2$ if  zero is a second or third-kind resonance of $H$.  However,  in a recent work \cite{MWY_ArXiv23_2},  the optimal bound  $O( |t|^{-3(\frac{1}{p}-\frac{1}{2})})$ can be obtained for  $3/2<p\le 2$ (but not for $1\le p <3/2$) by the wave operator method as zero is a second or third-kind resonance.  Hence
the presences of strongly zero resonance  indeed bring the essential obstacle to optimal decay for $p=1$.
\end{itemize}
\end{remark}
 \subsection{Further remarks and backgrounds}\label{more-backgrpund}
 Here we make  further comments on  spectral assumptions, and mention some known results on the second order wave equation with real-valued decaying potential.
 \subsubsection{The absence of embedded positive eigenvalues.}
It is well-known by Kato \cite{Kat59} that Schr\"{o}dinger operator $-\Delta +V$ has no positive eigenvalues if a bounded potential $V(x)= o(|x|^{-1})$ as $|x|$ goes to infinity,   also cf. \cite{Sim69, FHHH82, IJ03, KT06} for more related results and references. However, such a criterion
 does not work for fourth-order  Schr\"{o}dinger operator $H=\Delta^2 +V$. Indeed, for any dimension $n\geq 1$,  one can easily construct some $V\in C_0^\infty(\mathbb{R}^n)$ such that $H$ has some embedded positive eigenvalues, see Section 7.1 in \cite{FSWY20}. These results clearly indicate that the absence of positive eigenvalues  for the
 fourth-order Schr\"{o}dinger operator  would be more subtle and unstable than the second order cases with a bounded potential perturbation $V$.

It should be noticed that Feng et al. in \cite{FSWY20} have proved that $H=\Delta^2 +V$ does not contain any positive eigenvalue assume that  potential $V$ is bounded and satisfies the repulsive condition (i.e. $(x\cdot \nabla)V \leq 0$). Moreover, we also remark that for a general self-adjoint operator
  $\mathcal{H}$ on $L^{2}(\mathbf{R}^n)$, even if $\mathcal{H}$ has a simple embedded{\tiny } eigenvalue, Costin and Soffer  \cite{CS01} have proved
  that $\mathcal{H}+\epsilon W$ can kick off the eigenvalue located in a small interval under generic small perturbation of potential.
\subsubsection{The decay estimates of the fourth order Schr\"{o}dinger operators.}\label{Schrodinger-operator-background}
Recently, there exists  a few of  works devoting to the time decay estimates of $e^{-itH}$ generated by the fourth order Schr\"{o}dinger operator $H=\Delta^2 +V$ with a decaying potential $V$.  Feng et al. \cite{FSY} first proved that Kato-Jensen decay estimate of $e^{-itH}$ is bounded by
 $(1+|t|)^{-n/4}$ for $n\geq 5$,  and $L^1-L^\infty$ decay estimate is $O(|t|^{-1/2})$ for $n=3$ in the regular case. Somewhat later,  Erdog\u{a}n et al. \cite{EGT19} for $n=3$ and Green et al. \cite{GT19} for $n=4$,  proved that
the  $L^1-L^\infty$ estimates of $e^{-itH}$ is $O(|t|^{-n/4})$ for $n=3,4$ if zero is a regular point and first-kind resonance,  and the time decay rate will be changed as other kind of zero energy resonance occurs. More recently, Soffer et al. \cite{SWY21} proved the $L^1-L^\infty$ estimates of $e^{-itH}$ is $O(|t|^{-1/4})$ for dimension $n=1$ whatever zero is a regular point or resonance. It should be emphasized that the different types of zero resonances do not change the optimal time decay rate of $e^{-itH}$ in dimension one just at the cost of faster decay rate of potential. In \cite{LSY21} the authors have studied the $L^1-L^\infty$ estimates of $e^{-itH}$  in dimension two.

Furthermore, we mention that there exist  many other interesting works on the $L^p$ bounds of higher order wave operators $H=(-\Delta)^m+V$,  for instance, see  \cite{GG21b, EG21,Erdogan-Green23, EGG23} for $n=3, m=2$ or  $n>2m\ge 4$ in the regular case, and refer to \cite{MWY22} and \cite{Galtbayar_Yajima_ArXiv23} for the $L^p$-boundedness of wave operators of fourth order Schr\"{o}dinger operators with zero resonance cases  in $n=1, 4$, respectively.
\subsubsection{Classical wave equations with potentials.}\label{sub-waveequation}  In the following, we just recall  some results on decay estimates of  solution for  classical wave equation with a potential:
\begin{equation}\label{waveequation-twoorder}
u_{tt}+(-\Delta+V)u=0,\,\ u(0,x)=f(x),\,\ u_t(0,x)=g(x),  \, x\in\mathbb{R}^n.
\end{equation}
In the free case ( i.e.  $V=0$), it was well known that the $W^{k+1,1}- L^\infty$ estimate of
solution operator $\cos(t\sqrt{-\Delta})$ is $ O(|t|^{-\frac{1}{2}})$  and the $W^{k,1}- L^\infty$ estimate of
solution operator $\frac{\sin(t\sqrt{-\Delta})}{\sqrt{-\Delta}}$ is $O(|t|^{-\frac{1}{2}})$
for $k>\frac{1}{2}$ in dimension two,  In general, one need  use Hardy,  Besov  or BMO spaces to obtain the sharp $k=\frac{n-1}{2}$ smoothness  index  in even dimensions, see e.g. \cite{MSW}. One can get such decay bound in Sobolev spaces  in odd dimensions  (cf. \cite{Str05}).

When $V\neq 0$, Beals and Strauss \cite{B-S} first studied the $L^\infty$-decay  estimates of solution operators to equation (\ref{waveequation-twoorder}) in dimension $n\geq3$ (also cf. \cite{Beals, DP, GV03}).
There is not much work on the $W^{k,1}\rightarrow L^\infty$ dispersive estimates or `regularized' $L^1\rightarrow L^\infty$
 type estimates for $n=2$, where negative powers of $-\Delta+V$ are employed.
 In dimension two,  Moulin \cite{Mou} studied the high frequency estimates
 of this type. Kopylova \cite{Kop}  studied  local estimates  and obtained the decay rate of $t^{-1}(\log t)^{-2}$  for large $t$ when zero  is regular point.  Green \cite{Green14} also studied the $L^1-L^\infty$ dispersive estimates of solution operator of wave equation with potential in dimension two, and time decay rate of solution operators were improved in weighted space if zero is a regular point of the spectrum of $H$.

Besides, many  advances have been made in other dimensions, see e.g. Cardosa and  Vodev \cite{GV12} for dimensions $4\leq n\leq 7$. These results all require the assumption that zero is regular. In \cite{EGG14},  Erdo\v{g}an, Goldberg and Green have established a low energy $L^1-L^\infty$  bounds for  wave equation  with potential in four spatial dimensions, and showed that the loss of derivatives on the initial data for the wave equation is a high energy phenomenon.  Also cf.  \cite{ Bec-Gold, GV95, BPSZ03, BPSZ04, GV03, DP} for further decay bounds and  Strichartz estimates for wave equation with a potential particularly in dimensions $n\geq3$.
\subsection{The outline of the proof.}
Here we briefly explain some  ideas of the proofs of the theorems above.
For simplicity, we only consider the regular case.

In order to establish the decay estimates in Theorem \ref{thm-main-results-regular} and Theorem \ref{thm-main-results-resonance},  we will use the following Stone's formulas:
\begin{equation}\label{stoneformula-cos}
   \begin{split}
 \cos(t\sqrt{H})P_{ac}(H)f(x)=\frac{2}{\pi i} \int_0^\infty \lambda^3\cos(t\lambda^2)[ R_V^+(\lambda^4)-R_V^-(\lambda^4)]f(x)d\lambda,
  \end{split}
\end{equation}
\begin{equation}\label{stoneformula-sin}
   \begin{split}
 \frac{\sin(t\sqrt{H})}{\sqrt{H}}P_{ac}(H)g(x)=\frac{2}{\pi i} \int_0^\infty \lambda\sin(t\lambda^2)[ R_V^+(\lambda^4)-R_V^-(\lambda^4)]g(x)d\lambda,
  \end{split}
\end{equation}
where $R^\pm_V(\lambda^4)=(H-\lambda^4  \mp i0)^{-1}$ are the boundary operators of   resolvents  of $H$.
We need to study the expansions of the resolvent operators $R^\pm_V(\lambda^4)$ as $\lambda$ is  near zero by using
perturbations of  the following free resolvent $ R_0(z)$  (see e.g. \cite{FSY}):
\begin{equation}\label{R0zlaplace}
   \begin{split}
 R_0(z):=\big((-\Delta)^2 -z\big)^{-1}=\frac{1}{2z^\frac{1}{2}} \big(R(-\Delta;z^\frac{1}{2})
    -R(-\Delta;-z^\frac{1}{2})\big),\ z\in \mathbb{C} \setminus [0,\infty).
  \end{split}
\end{equation}
Here the resolvent $ R(-\Delta;z^\frac{1}{2}):=(-\Delta-z^\frac{1}{2})^{-1}$ with $\Im z^\frac{1}{2}>0$. For $\lambda \in \mathbb{R}^+$,
we define the limiting resolvent operators by
\begin{equation}\label{R0lambda-pm}
   \begin{split}
 R_0^\pm(\lambda):=R_0^\pm(\lambda \pm i0)= \lim_{\epsilon \rightarrow 0}
 \big( \Delta^2-(\lambda \pm i\epsilon) \big)^{-1},
  \end{split}
\end{equation}
\begin{equation}\label{RVlambda-pm}
   \begin{split}
 R_V^\pm(\lambda):=R_V^\pm(\lambda \pm i0)= \lim_{\epsilon \rightarrow 0}
 \big( H-(\lambda \pm i\epsilon) \big)^{-1}.
  \end{split}
\end{equation}
By using the equality (\ref{R0zlaplace}) for $R_0(z)$ with $z=w^4$ for $w$ in the first quadrant
of the complex plane, and taking limits as $w\rightarrow \lambda$ and $w\rightarrow i\lambda$, we have
\begin{equation}\label{R0lambda-4pm}
   \begin{split}
 R^\pm_0(\lambda^4)=\frac{1}{2\lambda^2}\big( R^\pm(-\Delta; \lambda^2)-R(-\Delta; -\lambda^2) \big),
  \ \lambda>0.
  \end{split}
\end{equation}
It was well-known that by the limiting absorption principle (see e.g. Agmon \cite{Agmon}), $R^\pm(-\Delta;\lambda^2)$ are well-defined  as the bounded operators of $B(L^2_s,L^2_{-s})$ for any $s>1/2$,
 therefore $R^\pm_0(\lambda^4) $ are also well-defined between the weighted spaces by using  \eqref{R0lambda-4pm}. This property is extended to $R^\pm_V(\lambda^4)$ for $\lambda >0$  for certain decay bounded potentials, see \cite{FSY}.
Moreover, we note that the kernel of the free resolvent of Laplacian in dimension three (see e.g. \cite{Goldberg-Schlag04}):
\begin{equation}\label{reslolent-R}
 R^\pm (-\Delta;\lambda^2)(x,y)= \frac{e^{\pm i\lambda|x-y|}}{4\pi |x-y|}, \ x, y\in\mathbb{R}^3,
   \end{equation}
 so by the identity (\ref{R0lambda-4pm})  we  obtain the following  kernel of $R^\pm_0(\lambda^4)$ for each $\lambda>0$:
\begin{equation}\label{resolent-R0lambda4}
R^\pm_0(\lambda^4)(x,y)=\frac{1}{2\lambda^2}\Big( \frac{e^{\pm i\lambda|x-y|}}{4\pi |x-y|}
-\frac{e^{-\lambda|x-y|}}{4\pi |x-y|}  \Big).
\end{equation}

 In order to estimate (\ref{stoneformula-cos}) and (\ref{stoneformula-sin}),  since
 \begin{equation}\label{relation-sin-cos}
\cos(t\sqrt{H})=\frac{e^{it\sqrt{H}}+e^{-it\sqrt{H}}}{2}, \ \,
\frac{\sin(t\sqrt{H})}{\sqrt{H}}=\frac{e^{it\sqrt{H}}-e^{-it\sqrt{H}}}{2i\sqrt{H}}.
\end{equation}
so  it suffices to estimate  $H^{\frac{\alpha}{2}}e^{-it\sqrt{H}}P_{ac}(H)$ for $\alpha=-1, 0$ by using
\begin{equation}\label{stone-H-alpha}
   \begin{split}
H^{\frac{\alpha}{2}}e^{-it\sqrt{H}}P_{ac}(H)f=& \frac{2}{\pi i } \int_0^\infty e^{-it\lambda^2}\lambda^{3+2\alpha}
[R_V^+(\lambda^4) -R_V^-(\lambda^4)]f d\lambda.
\end{split}
\end{equation}
We further decompose the integral \eqref{stone-H-alpha} into the low energy $\{0\leq \lambda \ll 1\}$  and  the high energy $\{\lambda \gg 1\}$ two parts.

For the high energy part, we will use the  following resolvent identity:
\begin{equation}\label{Rv-high}
   \begin{split}
R^\pm_V(\lambda^4)=R^\pm_0(\lambda^4)-R^\pm_0(\lambda^4)VR^\pm_0(\lambda^4)
+R^\pm_0(\lambda^4)VR^\pm_V(\lambda^4)VR^\pm_0(\lambda^4).
  \end{split}
\end{equation}
Hence we will need to study  the contribution of every term in (\ref{Rv-high}) to the integral (\ref{stone-H-alpha}).

For the low energy part,  we need to establish the expansions of the resolvent operators $R^\pm_V(\lambda^4)$ for $\lambda$ near zero.
Set $U(x)=\hbox{sign}\big(V(x)\big)$ and $v(x)=|V(x)|^{1/2}$.
Let $M^{\pm}(\lambda)=U+ vR^\pm_0(\lambda^4)v$. Then we have the following symmetric resolvent identity
 \begin{equation*}
R^\pm_V(\lambda^4) = R_0^\pm(\lambda^4) -R^\pm_0(\lambda^4)v(M^\pm(\lambda))^{-1}vR^\pm_0(\lambda^4).
\end{equation*}
Now we need to establish the expansions for $(M^\pm(\lambda))^{-1}$.
In the regular case,   for example,  the  expansions of $(M^\pm(\lambda))^{-1}$ is of the following form( see Theorem \ref{lem-M} below  )
 \begin{equation*}
\big(M^\pm(\lambda)\big)^{-1}=QA_{0,1}^0Q+\Gamma_1(\lambda), \  \  \lambda<<1.
\end{equation*}
 where $A_{0,1}^0, Q\in B(L^2, L^2)$ satisfying
 $Qv=0$, $\|\Gamma_1(\lambda)\|_{L^2\rightarrow L^2}=O(\lambda^2)$. Thus, we need to study the following integral for low energy
 \begin{equation}\label{oscillation-low-Q}
\int_0^\infty e^{-it\lambda^2}\lambda^{3+2\alpha}\big[R_0^\pm(\lambda^4)v(QA^0_{0,1}Q)vR_0^\pm(\lambda^4)\big](x,y)d\lambda.
\end{equation}
In order study the integral \eqref{oscillation-low-Q}, we will  make use of cancellation condition  $Qv=0$ and Lemma \ref{Taylor-low}, and then use  Lemma \ref{lem-LWP} to estimate  the oscillatory integral  (\ref{oscillation-low-Q}).

 The paper is organized as follows. In Section 2, we establish the dispersive bounds in the free case. In Section 3, we first recall the resolvent expansions  when $\lambda$ is near zero, then by Stone's formula, Littlewood-Paley method and oscillation integral we establish the low energy decay bounds of Theorem \ref{thm-main-results-regular} and Theorem \ref{thm-main-results-resonance}. In Section 4,
 we prove Theorem \ref{thm-main-results-regular} and Theorem \ref{thm-main-results-resonance} in  high energy.
 Finally, for the convenience  of reader, we give the asymptotic expansion of $(M^\pm(\lambda))^{-1}$ when $\lambda$ is  near to zero in Appendix.

\bigskip

\section{The decay estimates for the free case}
In this section, we are devote to get the decay bounds  of free Beam equation by using  Littlewood-Paley decomposition and oscillatory integral theory.

Choosing a fixed even function $ \varphi \in C^\infty_c(\mathbb{R})$ such that
 $\varphi(s)=1$ for
$ |s|\leq \frac{1}{2}$ and $ \varphi(s)=0$ for $ |s| \geq 1$.
Let $\varphi_N(s)=\varphi(2^{-N}s)- \varphi(2^{-N+1}s),\ N\in \mathbb{Z}$. Then
$\varphi_N(s)=\varphi_0(2^{-N}s)$,
$\hbox{supp}\varphi_0 \subset [ \frac{1}{4}, 1]$ and
\begin{equation}\label{jieduan}
\sum_{N=-\infty}^{\infty}\varphi_0(2^{-N}s)=1,\  s\in \mathbb{R}\setminus \{0\}.
\end{equation}

To estimate the integrals in (\ref{stoneformula-cos}) and (\ref{stoneformula-sin}) when $V=0$, by \eqref{relation-sin-cos} it is enough to establish the $L^1-L^\infty$ bounds of $(-\Delta)^{\alpha}e^{it\Delta}$ for $\alpha=-1, 0$.
Using Stone's formula  \eqref{stone-H-alpha} and \eqref{jieduan}, one has
\begin{equation}\label{free formula1}
   \begin{split}
(-\Delta)^{\alpha}e^{it\Delta}f=& \frac{2}{\pi i } \int_0^\infty e^{-it\lambda^2}\lambda^{3+2\alpha}
[R_0^+(\lambda^4) -R_0^-(\lambda^4)]f d\lambda\\
= &\frac{2}{\pi i }\sum_{N=-\infty}^{\infty} \int_0^\infty e^{-it\lambda^2}\lambda^{3+2\alpha}
\varphi_0(2^{-N}\lambda)[R_0^+(\lambda^4) -R_0^-(\lambda^4)]f d\lambda.
\end{split}
\end{equation}
Therefore,
it suffices to estimate the following integral kernel for each $N$:
$$\int_0^\infty e^{-it\lambda^2}\lambda^{3+2\alpha}\varphi_0(2^{-N}\lambda) R_0^\pm(\lambda^4)(x,y) d\lambda.$$

In the following, we first give a lemma which plays an  important role in estimating  integrals appeared  in this paper.
Since its proof is similar to Lemma 3.3 in \cite{LSY21},  we here omit the details.
\begin{lemma}\label{lem-LWP}
Let $A$ be some subset of $\mathbb{Z}$. Suppose that $\Phi(s,z)$ is a function on $\mathbb{R}\times \mathbb{R}^m$ which is smooth
 for the first variable $s$, and satisfies for any $(s,z)\in [1/4, 1] \times \mathbb{R}^m$,
$$|\partial_s^k\Phi(2^Ns,z)|\lesssim 1,\, k=0,1, N\in A\subset\mathbb{Z}.$$
Suppose that $\varphi_0(s)$ be a smoothing function of $\mathbb{R}$ defined in (\ref{jieduan}), $\Psi(z)$ is a nonnegative  function on
$\mathbb{R}^m$. Let $ N_0 =\big[\frac{1}{3}\log_2\frac{\Psi(z)}{|t|}\big]$,
for each $z\in \mathbb{R}^m$,  $l\in\mathbb{R}$, $N\in A$ and $t\neq 0$,  we have
\begin{equation*}
\Big|\int_0^\infty e^{-it2^{2N}s^2}
e^{\pm i2^Ns\Psi(z)} s^l\varphi_0(s)\Phi(2^Ns,z) ds \Big| \lesssim
\begin{cases}
(1+|t| 2^{2N})^{-\frac{1}{2}},  & \hbox{if}\  |N-N_0|\leq2,\\
(1+|t|2^{2N})^{-1}, & \hbox{if}\  |N-N_0|> 2.
\end{cases}
\end{equation*}
\end{lemma}
Throughout this paper, $\Theta_{N_0,N}(t)$ always denotes  the following function:
\begin{equation}\label{ function-theta}
 \begin{split}
\Theta_{N_0,N}(t)&:=
\begin{cases}
(1+|t|2^{2N})^{-\frac{3}{2}},  & \hbox{if}\,  |N-N_0|\leq2,\\
(1+|t| 2^{2N})^{-2},& \hbox{if}\, |N-N_0|> 2.
\end{cases}\\
\end{split}
\end{equation}
where $ N_0=\big[ \frac{1}{3}\log_2\frac{\Psi(z)}{|t|}  \big]$ and  $ \Psi(z)$ is a non-negative real value function on $ \mathbb{R}^m$.

\begin{proposition}\label{prop-free estimates}
Let $\Theta_{N_0,N}(t)$ be the  function defined in (\ref{ function-theta}) with $\Psi(z)=|x-y|$ and $z=(x,y)\in \mathbb{R}^6$. Then for each $x\neq y$ and $-\frac{3}{2}< \alpha \leq 0$,
\begin{equation}\label{pro-k3-1}
\Big|\int_0^\infty e^{-it\lambda^2}\lambda^{3+2\alpha}\varphi_0(2^{-N}\lambda) R_0^\pm(\lambda^4)(x,y) d\lambda \Big|
\lesssim 2^{(3+2\alpha)N}\Theta_{N_0,N}(t).
\end{equation}
Moreover,
\begin{equation}\label{pro-k3}
\sup\limits_{x,y\in \mathbb{R}^3}\Big|\int_0^\infty e^{-it\lambda^2}\lambda^{3+2\alpha}
R_0^\pm(\lambda^4)(x,y) d\lambda \Big|
 \lesssim|t|^{-\frac{3+2\alpha}{2}},
\end{equation}
which gives
 \begin{equation}\label{decayestimate-freecase}
 	\big \|(-\Delta)^{\alpha}e^{it\Delta }f\big \|_{L^\infty(\mathbb{R}^3)}\lesssim
 |t|^{-\frac{3+2\alpha}{2}}\big\| f\big\|_{L^1(\mathbb{R}^3)}.
 \end{equation}
As a consequence, we immediately obtain that
 \begin{equation}\label{decayestimate-freecase-cos}
\big \|\cos (t \Delta)f \big\|_{L^\infty(\mathbb{R}^3)}\lesssim |t|^{-\frac{3}{2}}\ \|f\|_{L^1(\mathbb{R}^3)},
\end{equation}
and
 \begin{equation}\label{decayestimate-freecase-sin}
\big\|\frac{\sin (t \Delta)}{\Delta}g \big\|_{L^\infty(\mathbb{R}^3)} \lesssim |t|^{-\frac{1}{2}}\ \|g\|_{L^1(\mathbb{R}^3)}.
\end{equation}
\end{proposition}
\begin{remark}
When $\alpha=0$ in \eqref{decayestimate-freecase}, it is well-known that
	\begin{equation}\label{gassian-int}
		e^{it\Delta}f(x)=\frac{1}{(4\pi i t)^{\frac{3}{2}}}\int_{\mathbb{R}^3}e^{-\frac{i|x-y|^2}{4t}}f(y)dy,\ \  f\in L^1 \cap L^2.
	\end{equation}
As a consequence,  we  immediately obtain the decay estimate \eqref{decayestimate-freecase-cos} from the Young's inequality and  Gaussian integral \eqref{gassian-int} above.
\end{remark}
\begin{proof}
 For each $N\in \mathbb{Z}$, we write
\begin{equation*}
   \begin{split}
K_{0,N}^\pm(t,x,y):= \int_0^\infty e^{-i t\lambda^2}
                \lambda^{3+2\alpha}\varphi_0(2^{-N}\lambda)R^\pm_0(\lambda^4)(x,y)d\lambda.
\end{split}
\end{equation*}
Let $F^\pm(p)= \frac{e^{\pm ip}-e^{-p}}{p}, \, p\geq 0$, by the identity \eqref{resolent-R0lambda4}, we have
\begin{equation}\label{reso-R0-Fpm}
R^\pm_0(\lambda^4)(x,y)= \frac{1}{8\pi\lambda}F^\pm(\lambda|x-y|).
\end{equation}
Let $\lambda=2^Ns$, then
\begin{equation*}
   \begin{split}
K_{0,N}^\pm(t,x,y)
                =&\frac{ 2^{(3+2\alpha )N}}{8\pi}\int_0^\infty
  e^{-i t 2^{2N}s^2} s^{2+2\alpha}\varphi_0(s)F^\pm(2^Ns|x-y|)ds.
	\end{split}
\end{equation*}
Note that $s\in \hbox{supp} \varphi_0 \subset [1/4, 1]$, by using integration by parts, we have
\begin{equation}\label{esti-K0N}
   \begin{split}
| K_{0,N}^\pm(t,x,y)|
  \lesssim &\frac{ 2^{(3+2\alpha )N}}{1+|t|2^{2N}}\Big| \int_0^\infty
  e^{-i t 2^{2N}s^2} \partial_s\Big( s^{1+2\alpha}\varphi_0(s)F^\pm(2^Ns|x-y|)\Big)ds\Big|\\
  \lesssim &\frac{ 2^{(3+2\alpha )N}}{1+|t|2^{2N}}\bigg( \Big| \int_0^\infty  e^{-i t 2^{2N}s^2} \partial_s\big(s^{1+2\alpha}\varphi_0(s)\big)F^\pm(2^Ns|x-y|)ds\Big|\\
  &+\Big| \int_0^\infty  e^{-i t 2^{2N}s^2}  s^{1+2\alpha} \varphi_0(s)\partial_s \big(F^\pm(2^Ns|x-y|)\big) ds\Big| \bigg)\\
  := &\frac{ 2^{(3+2\alpha )N}}{1+|t|2^{2N}}\Big(\big|\mathcal{E}_{01,N}^\pm(t,x,y)\big|+\big|\mathcal{E}_{02,N}^\pm(t,x,y)\big|\Big).
	\end{split}
\end{equation}

We  first estimate  $\mathcal{E}_{02,N}^\pm(t,x,y)$. Let $r=|x-y|$, since
\begin{equation*}
   \begin{split}
\partial_sF^\pm(2^Nsr)
=  s^{-1} 2^Nsr(F^\pm)'(2^Nsr):= e^{\pm i 2^Nsr}s^{-1}F^\pm_1(2^Nsr),
	\end{split}
\end{equation*}
where
$$ F^\pm_1(p)= p e^{\mp ip }(F^\pm)'(p)= \frac{(\pm ip -1)+(p+1)e^{-p \mp ip}}{p}. $$
Hence one has
\begin{equation*}
   \begin{split}
\mathcal{E}_{02,N}^\pm(t,x,y)
  = \int_0^\infty
  e^{-i t 2^{2N}s^2} e^{\pm i2^Ns|x-y|} s^{2\alpha}\varphi_0(s) F^\pm_1(2^Ns|x-y|)ds.
	\end{split}
\end{equation*}
Observe that  for any $x,y$,
$$|\partial_s^k F^\pm_1(2^Ns|x-y|)| \lesssim 1,\,k=0,1,$$
by Lemma \ref{lem-LWP} with $z=(x,y)$, $\Psi(z)=|x-y|$ and $\Phi(2^Ns,z)=F^\pm_1(2^Ns|x-y|)$,
we obtain that $\mathcal{E}_{02,N}^\pm$ is bounded by $(1+|t|2^{2N})\Theta_{N_0,N}(t)$.

Similarly, we obtain the same bounds for $\mathcal{E}_{01,N}^\pm$.
 Furthermore,  by (\ref{esti-K0N}) we get that $K_{0,N}^\pm$ is bounded
by $2^{(3+2\alpha )N}\Theta_{N_0,N}(t)$.   Thus we obtain that the estimate \eqref{pro-k3-1} holds.

Finally, in order to obtain (\ref{pro-k3}), it's enough to show that for any $x\neq y$ and $-\frac{3}{2}<\alpha\leq 0$,
\begin{equation}\label{sum-appha=0-alphabig0}
   \begin{split}
\sum_{N=-\infty}^{+\infty}| K_{0,N}^\pm(t,x,y)|\lesssim|t|^{-\frac{3+2\alpha}{2}}.
\end{split}
\end{equation}

 In fact, for $t\neq0$, there exists $N_0' \in\mathbb{Z}$ such that $2^{N_0'}\sim |t|^{-\frac{1}{2}}$.
If  $-\frac{3}{2}<\alpha<0$,  then we have for any $x\neq y$,
\begin{equation}\label{sum-alpha-not=0}
   \begin{split}
 \sum_{N=-\infty}^{+\infty}| K_{0,N}^\pm(t,x,y)|
&\lesssim \sum_{N=-\infty}^{+\infty}2^{(3+2\alpha)N} (1+|t|2^{2N})^{-\frac{3}{2}}\\
&\lesssim\sum_{N=-\infty}^{N_0'} 2^{(3+2\alpha)N}+
\sum_{N=N_0'+1}^{+\infty} 2^{(3+2\alpha)N}(|t|2^{2N})^{-\frac{3}{2}}\\
&\lesssim|t|^{-\frac{3+2\alpha}{2}}.
\end{split}
\end{equation}
If $\alpha=0$, then we have for any $x\neq y$,
\begin{equation}\label{sum-alpha=0}
   \begin{split}
  \sum_{N=-\infty}^{+\infty}| K_{0,N}^\pm(t,x,y)|&\lesssim\sum_{|N-N_0|\leq2} 2^{3N} (1+|t|2^{2N})^{-\frac{3}{2}}
  +\sum_{|N-N_0|>2} 2^{3N} (1+|t|2^{2N})^{-2} \\
  &\lesssim  |t|^{-\frac{3}{2}}+\sum_{N=-\infty}^{N_0'} 2^{3N}+
\sum_{N=N_0'+1}^{+\infty} 2^{3N}(|t|2^{2N})^{-2}\\
&\lesssim|t|^{-\frac{3}{2}}.
  \end{split}
\end{equation}
Hence the estimate \eqref{sum-appha=0-alphabig0} is proved, which gives \eqref{pro-k3}.

By \eqref{free formula1} it immediately  follows that for $ -\frac{3}{2}<\alpha\le 0,$
$$ \big \|(-\Delta)^{\alpha}e^{it\Delta }f\big \|_{L^\infty(\mathbb{R}^3)}\lesssim
|t|^{-\frac{3+2\alpha}{2}} \|f\big \|_{L^1(\mathbb{R}^3)}.
$$
Furthermore, recall the identity \eqref{relation-sin-cos},  the desired estimates \eqref{decayestimate-freecase-cos} and \eqref{decayestimate-freecase-sin} are obtained.
\end{proof}

\section{Low energy decay estimates  }
\label{low energy}
In this section, we come to establish the low energy decay estimates of the solution operators to high order wave equation (\ref{cauchyequation1}).
We first need to study the asymptotic expansions of the perturbed resolvent $R^\pm_V(\lambda^4)$ as $\lambda$ is  near zero (see  \cite{EGT19}), then by Stone's formula, Littlewood-Paley method and oscillation integral theory we obtain  the  decay bounds of Theorem \ref{thm-main-results-regular} and Theorem \ref{thm-main-results-resonance} for low energy.
\subsection{Asymptotic expansions of resolvent near zero }
In this subsection, we study  the asymptotic expansions of the perturbed resolvent $R^\pm_V(\lambda^4)$  in a neighborhood of zero threshold.

By using the free resolvent kernel  $R^\pm_0(\lambda^4)(x,y) $  in \eqref{resolent-R0lambda4},
we have the following expression when $\lambda|x-y|<1$:
\begin{equation}\label{resolvent-expansion-R0lambda4}
   \begin{split}
 R^\pm_0(\lambda^4)(x,y)=& \frac{a^\pm}{\lambda}+G_0(x,y) +a_1^\pm \lambda G_1(x,y) +a_3^\pm \lambda^3G_3(x,y)\\
&+\lambda^4G_4(x,y)+ \sum_{k=5}^N a_k^\pm \lambda^k G_k(x,y) +O\big( \lambda^{N+1}|x-y|^{N+2}\big),
		\end{split}
	\end{equation}
where
\begin{equation}\label{def-Gk}
   \begin{split}
&G_0(x,y)=-\frac{|x-y|}{8\pi},\ G_1(x,y)=|x-y|^2, \ G_3(x,y)=|x-y|^4,\\
& G_4(x,y)=-\frac{|x-y|^5}{4\pi\cdot 6!},\
 G_k(x,y)= |x-y|^{k+1}, \ k\geq5,
\end{split}
	\end{equation}
and the coefficients
$$\displaystyle a^\pm= \frac{1\pm i}{8\pi}, \ a_1^\pm=\frac{1\mp i}{8\pi \cdot3!},
\ a_3^\pm=\frac{1\pm i}{8\pi \cdot 5!},
\ a_k^\pm= \frac{ (-1)^{k+1}+ (\pm i)^{k+2}}{8\pi\cdot (k+2)!}(k\geq 5).$$
In fact, the expansion remains valid when $\lambda|x-y| \geq 1$.  In  the sequel,  we also denote by $G_k$  operators with  the integral kernels $G_k(x,y)$ above.
In particular,  $G_0=(\Delta^2)^{-1}$.

Let $U(x)=\hbox{sign}\big(V(x)\big)$ and $v(x)=|V(x)|^{1/2}$, then we have $ V=Uv^2$ and  the following symmetric resolvent identity
 \begin{equation}\label{id-RV}
R^\pm_V(\lambda^4) = R_0^\pm(\lambda^4) -R^\pm_0(\lambda^4)v(M^\pm(\lambda))^{-1}vR^\pm_0(\lambda^4),
\end{equation}
where $M^{\pm}(\lambda)=U+ vR^\pm_0(\lambda^4)v$. Hence, we need to obtain the expansions for
$(M^\pm(\lambda))^{-1}$.

Let $T= U+vG_0v$, and $P= \|V\|^{-1}_{L^1} v\langle v, \cdot \rangle$ denote the orthogonal projection onto the span space by $v$.
By the expansions \eqref{resolvent-expansion-R0lambda4} of  free resolvent  $R_0^\pm(\lambda^4)$, we have the following expansions of $M^\pm(\lambda)$.
\begin{lemma}\label{lem-M} Let $|V(x)|\lesssim (1+|x|)^{-\beta}$ with some $\beta > 0$.  Set $\displaystyle \tilde{a}^\pm= a^\pm \|V \|_{L^1}$ and $M^{\pm}(\lambda)=U+ vR^\pm_0(\lambda^4)v$.  Then the following identities of $M^\pm(\lambda)$ hold on $ \mathbf{B}(L^2, L^2)$ for $\lambda>0$:
\begin{itemize}
\item[(i)] If $\beta > 7$, then
  \begin{equation}\label{Mpm-1}
   \begin{split}
   M^\pm(\lambda)= \frac{\tilde{a}^\pm}{\lambda}P +T+\Gamma_1(\lambda);
     \end{split}
 \end{equation}

  \item[(ii)]If $\beta > 11$, then
  \begin{equation}\label{Mpm-2}
   \begin{split}
  M^\pm(\lambda)= \frac{\tilde{a}^\pm}{\lambda}P +T+a_1^\pm \lambda vG_1v+\Gamma_3(\lambda);
     \end{split}
 \end{equation}

  \item[(iii)] If $\beta > 19$, then
\begin{equation}\label{Mpm-3}
   \begin{split}
  M^\pm(\lambda)= &\frac{\tilde{a}^\pm}{\lambda}P +T+a_1^\pm \lambda vG_1v+a_3^\pm \lambda^3 vG_3v \\
  &+\lambda^4vG_4v + a_5^\pm\lambda^5 vG_5v + a_6^\pm \lambda^6vG_6v + \Gamma_7(\lambda);
     \end{split}
 \end{equation}

  \item[(iv)] If $ \beta > 23$, then
  \begin{equation}\label{Mpm-4}
    \begin{split}
  M^\pm(\lambda)= \frac{\tilde{a}^\pm}{\lambda}P &+T+a_1^\pm \lambda vG_1v+a_3^\pm \lambda^3 vG_3v \\
  &+\lambda^4vG_4v +\sum_{k=5}^8a_k^\pm \lambda^kvG_kv + \Gamma_9(\lambda);
   \end{split}
 \end{equation}
 \end{itemize}
where $\Gamma_k(\lambda)(k=1,3,7,9)$ be $\lambda$-dependent operators satisfying that
\begin{equation*}
\big\|\Gamma_k(\lambda)\big\|_{L^2\rightarrow L^2}+\lambda\big\|\partial_\lambda\Gamma_k(\lambda)\big\|_{L^2\rightarrow L^2}
+\lambda^2\big\|\partial^2_\lambda\Gamma_k(\lambda)\big\|_{L^2\rightarrow L^2}
\lesssim \lambda^k,\  \lambda>0.
\end{equation*}
\end{lemma}

  Now we introduce the type of resonances that may occur at the zero energy as follows:
\begin{definition}\label{definition of resonance}Let $ Q=I-P$ and $T= U+vG_0v$.
\begin{itemize}
\item[(i)] If $QTQ$ is invertible on $QL^2$, then we say that zero is a regular point of  $H$. In this case, we define $D_0= (QTQ)^{-1}$ as an operator on $QL^2$.
\vskip0.2cm
\item[(ii)] Assume that $QTQ$ is  not invertible on $QL^2.$ Let $S_1$ be the Riesz projection onto the kernel of $QTQ$.  Then $QTQ+S_1$ is invertible on $QL^2$.  In this case, we define $D_0=\big(QTQ+S_1\big)^{-1}$ as an operator on $QL^2$, which doesn't conflict with the previous definition since $S_1=0$ when zero is a regular point.
We say that zero is the first kind resonance of $H$ if
\begin{equation}\label{T_1}
T_1:= S_1TPTS_1-\frac{\|V\|_{L^1}}{3\cdot (8\pi)^2}S_1vG_1vS_1
\end{equation}
is invertible on $S_1L^2$. We define $D_1=T_1^{-1}$ as an operator on $S_1L^2$.
\vskip0.2cm
\item[(iii)]Assume that $T_1$ is not invertible on $S_1L^2.$ Let $S_2$ be the Riesz projection onto the kernel of $T_1.$ Then $T_1+S_2$ is invertible on $S_1L^2.$ In this case, we define $D_1=\big(T_1+S_2\big)^{-1}$ as an operator on $S_1L^2$, which doesn't conflict with previous
 definition since $S_2=0$ when zero is the first kind of resonance.
We say that zero is the second kind  resonance of  $H$ if
\begin{equation}\label{T_2}
T_2:= S_2vG_3vS_2+\frac{10}{3\|V\|_{L^1}} S_2(vG_1v)^2S_2-\frac{10}{3\|V\|_{L^1}}S_2vG_1vTD_1TvG_1vS_2
\end{equation}
is invertible on $S_2L^2$. We define $D_2=T_2^{-1}$ as an operator on $S_2L^2$.
\vskip0.2cm
\item[(iv)] Finally if $T_2$ is not invertible on $S_2L^2$,
we say that zero is the third kind  resonance of $H$. In this case, the operator $T_3:=S_3vG_4vS_3$ is always invertible on $S_3L^2$ (see
Lemma \ref{T3-inve} in Appendix) where $S_3$ be the Riesz projection onto the kernel of $T_2,$ let $D_3=T_3^{-1}$ as an operator on $S_3L^2$. We define $D_2=(T_2+S_3)^{-1}$ as an operator on $S_2L^2$.
\end{itemize}
\end{definition}
From the definition above, we have $ S_1L^2 \supseteq S_2L^2\supseteq S_3L^2$, which describe the zero energy resonance types of $H$ as follows:
\begin{itemize}
\item ~Zero is a regular point of $H$ if and only if $S_1L^2=\{0\};$
\vskip0.2cm
\item~ Zero is a first-kind resonance of $H$ if and only if $S_1L^2\neq\{0\}$ and $S_2L^2=\{0\};$
\vskip0.2cm
\item~Zero is a second-kind resonance of $H$ if and only if $S_2L^2\neq\{0\}$ and $S_3L^2=\{0\};$
\vskip0.2cm
\item~Zero is an eigenvalue of $H$ ( i.e. the third-kind resonance ) if and only if $S_3L^2\neq\{0\}.$
\end{itemize}

Noting that Theorem \ref{gongzhengkehua-1} gives  the characterizations of threshold spectral subspaces $S_jL^2(j=1,2,3)$ by the distributional solution of $H\phi=0$  in Appendix, hence  we rewrite the following equivalent statements:
\begin{itemize}
\item
Zero is a firs-kind resonance of $H$ if there exists a nonzero $\phi\in L^2_{-\sigma}(\mathbb{R}^3) $  for all $\sigma>\frac{3}{2}$ but no nonzero $\phi\in L^2_{-\sigma}(\mathbb{R}^3)$ with  $\sigma>\frac{1}{2}$ such that $H\phi=0$ in the distributional sense;
\item
Zero is a second-kind resonance of $H$ if there exists a nonzero  $\phi\in L^2_{-\sigma}(\mathbb{R}^3)$ for   all $\sigma>\frac{1}{2}$ but no nonzero $\phi\in L^2$ such that $H\phi=0$ in the distributional sense;
\item
Zero is a third-kind resonance  (i.e. eigenvalue) of $H$ if there exists  nonzero $\phi\in L^2(\mathbb{R}^3)$ such that $H\phi=0$ in the distributional sense;
\item
Zero is a regular point of $H$ if zero is neither a resonance nor an eigenvalue of $H$.
\end{itemize}

Furthermore, since $vG_0v$ is a Hilbert-Schmidt operator, and $T=U+vG_0v$ is the compact perturbation of $U$ (see e.g. \cite{EGG14, GT19}). Hence $S_1$ is a finite-rank projection by Fredholm alternative theorem. Notice that $S_3\leq S_2\leq S_1$,  then all $S_j(j=1,2,3)$ are finite-rank operators. Moreover, by the definitions of $S_j(j=1,2,3)$, we have these identities
$S_iD_j=D_jS_i=S_i\  (3\ge i\geq j\ge 1)$ and $S_iD_j=D_jS_i=D_j \ (1\le i< j\le3)$.

\begin{definition}
We say an operator $T:\,L^2(\mathbb{R}^3)\rightarrow L^2(\mathbb{R}^3)$ with kernel $T(\cdot, \cdot)$
 is absolutely bounded if the operator with the kernel $|T(\cdot, \cdot)|$ is bounded from $L^2(\mathbb{R}^3)$ into itself.
\end{definition}

We remark that Hilbert-Schmidt and finite-rank operators are absolutely bounded operators. Moreover, we have the following proposition, see Lemma 4.3 in \cite{EGT19}.
\begin{proposition}\label{Pro-absulu-oper}
Let $|V(x)|\leq (1+|x|)^{-7-}$. Then $QD_0Q$ is absolutely bounded.
\end{proposition}

In the following, we will give the specific characterizations of projection spaces
$S_jL^2(j=1,2,3)$ by the orthogonality of these projection operators $S_j(j=1,2,3)$.

\begin{lemma}\label{projiction-spaces-SjL2}
Let $S_j(j=1,2,3)$ be the projection operators given by Definition \ref{definition of resonance}. Then
\begin{itemize}
\item[(i)]$f\in S_1L^2$ if and only if $f\in\hbox{ker}(QTQ)$. Moreover, $QTS_1= S_1TQ=0$.

\item[(ii)]$f\in S_2L^2$ if and only if
\begin{equation*}
\begin{split}
f\in\hbox{ker}(T_1)&= \hbox{ker}(S_1TPTS_1)\cap \hbox{ker}(S_1vG_1vS_1)\\
&=\{ f\in S_1L^2 \big|PTf=0, \langle x_iv, f\rangle =0,\, j=1,2,3 \}.
\end{split}
\end{equation*}
In particular, $TS_2=S_2T=0$, $QvG_1vS_2=S_2vG_1vQ=0$.

\item[(iii)]$f\in S_3L^2$ if and only if
\begin{equation*}
\begin{split}
f\in \hbox{ker}(T_2)=
\{ f\in S_2L^2 \big| \langle x_ix_jv, f\rangle =0, i, j=1,2,3\}.
\end{split}
\end{equation*}
 \end{itemize}
\end{lemma}
Now we will  give asymptotic expansions of $\big(M^\pm(\lambda)\big)^{-1}$ as follows:
\begin{theorem}\label{thm-main-inver-M}
Let $S_j$(j=1,2,3) be the operators defined in Definition \ref{definition of resonance}.  Assume that $|V(x)| \lesssim (1+|x|)^{-\beta}$ with some $\beta>0$. Then we have
the following expansions of $\big(M^\pm(\lambda)\big)^{-1}$ in $L^2(\mathbb{R}^3)$ for $0<\lambda \ll 1$.
\begin{itemize}
\item[(i)] If zero is a regular point of $H$ and $\beta > 7$, then
\begin{equation}\label{thm-regularinver-M0 }
	\begin{split}
\big(M^\pm(\lambda)\big)^{-1}=&QA_{0,1}^0Q+\Gamma_1(\lambda);
\end{split}
\end{equation}
\item[(ii)] If zero is the first kind resonance of $H$ and $\beta > 11$, then
\begin{equation}\label{thm-resoinver-M1 }
	\begin{split}
\big(M^\pm (\lambda)\big)^{-1}
 = \frac{S_1A_{-1,1}^1S_1}{\lambda}+ \Big(S_1A^1_{0,1} +A^1_{0,2}S_1 +QA^1_{0,3}Q \Big)+\Gamma_1(\lambda);
\end{split}
\end{equation}
\item[(iii)] If zero is the second kind resonance of $H$ and $\beta > 19$, then
\begin{equation}\label{thm-resoinver-M2 }
   \begin{split}
\big(M^\pm (\lambda)\big)^{-1}
=&\frac{S_2A^2_{-3,1}S_2}{\lambda^{3}}+\frac{S_2A^2_{-2,1}S_1 + S_1A^2_{-2,2}S_2}{\lambda^2}
+\frac{S_2A^2_{-1,1}+A^2_{-1,2}S_2}{\lambda}\\
&+\frac{S_1A^2_{-1,3}S_1}{\lambda}+\Big( S_1A_{0,1}^2 +A^2_{0,2}S_1 +QA^2_{0,3}Q\Big)
+ \Gamma_1(\lambda);
\end{split}
\end{equation}
\item[(iv)] If zero is the third kind resonance of $H$ and $\beta > 23$, then
\begin{equation}\label{thm-resoinver-M3 }
   \begin{split}
\big(M^\pm(\lambda) \big)^{-1}
=&\frac{S_3 D_3 S_3}{\lambda^4}+\frac{S_2A^3_{-3,1}S_2}{\lambda^3}+\frac{S_2A^3_{-2,1}S_1 + S_1A^3_{-2,2}S_2}{\lambda^2}
 +\frac{S_2A^3_{-1,1}+ A^3_{-1,2}S_2}{\lambda}\\
& +\frac{S_1A^3_{-1,3}S_1}{\lambda}
+\Big( S_1A_{0,1}^3 +A^3_{0,2}S_1 +QA^3_{0,3}Q\Big)+\Gamma_1(\lambda);
\end{split}
\end{equation}
\end{itemize}
where $A_{i,j}^k$ are $\lambda$-independent absolutely  bounded operators in $L^2(\mathbb{R}^3)$;
$\Gamma_1(\lambda)$ be some  $\lambda$-dependent operator which may be different in each  expansion. but  all  satisfy that
\begin{equation*}
\big\|\Gamma_1(\lambda)\big\|_{L^2\rightarrow L^2}+\lambda\big\|\partial_\lambda\Gamma_1(\lambda)\big\|_{L^2\rightarrow L^2}
+\lambda^2\big\|\partial^2_\lambda\Gamma_1(\lambda)\big\|_{L^2\rightarrow L^2}
\lesssim \lambda.
\end{equation*}
\end{theorem}
The asymptotic expansions of $\big(M^\pm(\lambda)\big)^{-1}$ in $L^2(\mathbb{R}^3)$ above  can be seen in \cite{EGT19},  also see \cite{FSY} for the regular case. In Theorem \ref{thm-main-inver-M} we take somehow different notations and use more detailed expansions for our purpose.  For the convenience of readers,  we give all details of  the proof as an appendix below.


\subsection{Low energy decay estimates}
In this subsection, we are devoting to establishing  the low energy decay bounds for Theorem \ref{thm-main-results-regular} and Theorem \ref{thm-main-results-resonance}. By identities (\ref{relation-sin-cos}) it suffices to establish low energy dispersive bounds of $H^{\frac{\alpha}{2}}e^{-it\sqrt{H}}P_{ac}(H)$ for $\alpha=-1, 0$.

Below we use the smooth and  even cut-off $\chi$ given by $\chi=1$ for $|\lambda|<\lambda_0\ll1$ and $\chi=0$ for $|\lambda|>2\lambda_0$, where $\lambda_0 $ is some sufficiently small (but fixed) positive constant depending on low energy expansion of $(M^{\pm}(\lambda))^{-1}$. In analyzing the high energy later, we utilize the complementary cut-off $\widetilde{\chi}(\lambda):=1-\chi(\lambda).$

By using  Stone's formula, one has
\begin{equation}\label{Stone-Paley}
   \begin{split}
H^{\frac{\alpha}{2}}e^{-it\sqrt{H}}P_{ac}(H)f=& \frac{2}{\pi i } \int_0^\infty e^{-it\lambda^2}\lambda^{3+2\alpha}
[R_V^+(\lambda^4) -R_V^-(\lambda^4)]f d\lambda\\
=&\frac{2}{\pi i }\int_0^\infty \chi(\lambda)e^{-it\lambda^2}\lambda^{3+2\alpha}[R_V^+(\lambda^4) -R_V^-(\lambda^4)]f d\lambda\\
&+\frac{2}{\pi i } \int_0^\infty\widetilde{\chi}(\lambda) e^{-it\lambda^2}\lambda^{3+2\alpha}[R_V^+(\lambda^4) -R_V^-(\lambda^4)]f d\lambda,
\end{split}
\end{equation}
where $\chi(\lambda)=\sum\limits_{N=-\infty}^{N'}\varphi_0(2^{-N}\lambda)$  and
$\widetilde{\chi}(\lambda)=\sum\limits_{N=N'+1}^{+\infty}\varphi_0(2^{-N}\lambda)$  for $ N'<0$. We remark that the choice of the constant
$ N'$ depends on a sufficiently small neighborhood of $\lambda=0$ such that  the expansions of all resonance types in
Theorem \ref{thm-main-inver-M} hold.

Hence in order to establish the low energy decay bounds for Theorem \ref{thm-main-results-regular} and Theorem \ref{thm-main-results-resonance},
by  using (\ref{Stone-Paley}), it suffices to prove the following theorem.
\begin{theorem}\label{thm-low} Let $|V(x)|\lesssim (1+|x|)^{-\beta}$ $ (x\in\mathbb{R}^3)$ with some $\beta>0$. Then
\begin{itemize}
\item[(i)] If zero is a regular point of $H$ and $\beta>7 $, then for $-\frac{3}{2} < \alpha \leq 0$,
\begin{equation}\label{thm-low-01}
\big\| H^{\frac{\alpha}{2}}e^{-it\sqrt{H}}P_{ac}(H)\chi(H)\big\|_{L^1\rightarrow L^\infty}\lesssim |t|^{-\frac{3+2\alpha}{2}}.
\end{equation}
\item[(ii)] If zero is a first-kind resonance of $H$ and $\beta>11$, then for $\  -\frac{3}{2} < \alpha \leq 0$,
\begin{equation}\label{thm-low-02}
\big\| H^{\frac{\alpha}{2}}e^{-it\sqrt{H}}P_{ac}(H)\chi(H)\big\|_{L^1\rightarrow L^\infty}\lesssim |t|^{-\frac{3+2\alpha}{2}}.
\end{equation}
\item[(iii)] If zero is a second-kind  resonance of $H$ and $\beta>19$ or third-kind  resonance of $H$  and $\beta>23$, then for $\  -\frac{1}{2} < \alpha \leq 0$,
\begin{equation}\label{thm-low-0311}
	\big\| H^{\frac{\alpha}{2}}e^{-it\sqrt{H}}P_{ac}(H)\chi(H)\big\|_{L^1\rightarrow L^\infty}\lesssim |t|^{-\frac{1+2\alpha}{2}}.
\end{equation}
Moreover, there are two time-dependent operators $F_{t}$ and $G_t$ satisfying
$$\|F_ {t}\|_{L^1 \rightarrow L^\infty}   \lesssim|t|^{-\frac{1}{2}} \ \ \hbox{and}  \ \
\|G_ {t}\|_{L^1 \rightarrow L^\infty}   \lesssim|t|^{\frac{1}{2}},$$
 such that
\begin{equation}\label{thm-low-cos-031}
\big\| \cos(t\sqrt{H})P_{ac}(H)\chi(H)-F_{t}\big\|_{L^1\rightarrow L^\infty}\lesssim |t|^{-\frac{3}{2}},
\end{equation}
and
\begin{equation}\label{thm-low--sin-031}
\Big\| \frac{\sin(t\sqrt{H})}{\sqrt{H}}P_{ac}(H)\chi(H)-G_{t}\Big\|_{L^1\rightarrow L^\infty}\lesssim |t|^{-\frac{1}{2}}.
\end{equation}
\end{itemize}
\end{theorem}

Before proving Theorem \ref{thm-low}, we first give the following lemma, which has a crucial role in making use of cancellations
 of projection operators $Q, S_j(j=1,\cdots,5)$ appearing in the  asymptotic expansions of resolvent $R_V(\lambda^4)$ as $\lambda$  near zero, and will be used frequently to obtain the low energy dispersive estimates for all cases.
\begin{lemma}\label{Taylor-low}
Assume that $x,y \in \mathbb{R}^3 $ and $ \lambda>0$. We define
 $\displaystyle w = w(x)=\frac{x}{|x|}$ for $x\neq 0$ and $w(x)=0$ for $x=0$.
Let $\theta\in[0,1]$ and $\displaystyle |y|\cos\alpha= \langle y, w(x-\theta y)\rangle $ where $\alpha\equiv \alpha(x, y, \theta )$ is  the angle between the vectors $y$ and $x-\theta y $.
\begin{itemize}
\item[(i)]  If $F(p)\in C^{1}(\mathbb{R})$. Then
\begin{equation*}
   \begin{split}
F(\lambda|x-y|)= F(\lambda|x|)-\lambda |y|\int_0^1 F'(\lambda|x-\theta y|)\cos\alpha d\theta.
\end{split}
\end{equation*}

\item[(ii)] If $F(p)\in C^{2}(\mathbb{R})$ and $F'(0)=0$. Then
\begin{align*}
F(\lambda|x-y|)
=&F(\lambda|x|)-\lambda\big\langle y, w(x)\big\rangle F'(\lambda|x|)\\
&+\lambda^2|y|^2\int_0^1(1-\theta)\Big(\sin^2\alpha
\frac{F'(\lambda|x-\theta y|)}{\lambda|x-\theta y|}
+\cos^2\alpha F''(\lambda|x-\theta y|)\Big) d\theta.
\end{align*}

\item[(iii)] If  $F(p)\in C^{3}(\mathbb{R})$ and $ F'(0)=F''(0)=0$. Then
\begin{equation*}
   \begin{split}
 F(\lambda|x-y|)=
 & F(\lambda|x|) -\lambda\big\langle y, w(x)\big\rangle F'(\lambda|x|)
+\frac{\lambda^2}{2} \Big[\big( |y|^2- \big\langle y, w(x)\big\rangle^2\big)\frac{F'(\lambda|x|)}{\lambda|x|}
\\
&+\big\langle y, w(x)\big\rangle^2F''(\lambda|x|)\Big]
+\frac{\lambda^3|y|^3}{2}\int_0^1(1-\theta)^2
\Big[3\cos\alpha \sin^2\alpha \Big( \frac{ F'(\lambda|x-\theta y|)}{\lambda^2|x-\theta y|^2}\\
&- \frac{ F''(\lambda|x-\theta y|)}{\lambda|x-\theta y|} \Big)
-\cos^3\alpha F^{(3)}(\lambda|x-\theta y|)\Big]d\theta.
\end{split}
\end{equation*}
\end{itemize}
\end{lemma}
\begin{proof}
By the same calculations as in the proof of Lemma 3.5 in \cite{LSY21}, we can obtain  this lemma.
Here we omit the details.
\end{proof}

\subsubsection{\textbf{Regular case} }
In order to establish the lower energy estimate (\ref{thm-low-01}), recall
Stone's formula
\begin{equation}\label{stone-formula-low}
   \begin{split}
H^\frac{\alpha}{2}e^{-it\sqrt{H}}P_{ac}(H)\chi(H)f=& \frac{2}{\pi i}\int_0^\infty e^{-it\lambda^2}\chi(\lambda)\lambda^{3+2\alpha}
[R_V^+(\lambda^4)- R_V^-(\lambda^4)]fd\lambda\\
=&\sum_{N=-\infty}^{N'}\sum_{\pm}\frac{ 2}{\pi i}
\int_0^\infty e^{-it\lambda^2}\varphi_0(2^{-N}\lambda)\lambda^{3+2\alpha}R^\pm_V(\lambda^4)fd\lambda.
\end{split}
\end{equation}
If zero is a regular point of $H$, using (\ref{id-RV}) and (\ref{thm-regularinver-M0 }),
we have
\begin{equation}\label{RV-regular}
   \begin{split}
R_V^\pm(\lambda^4)
= R^\pm_0(\lambda^4)- R^\pm_0(\lambda^4)v\Big(QA^0_{0,1}Q \Big) vR^\pm_0(\lambda^4)
-R^\pm_0(\lambda^4)v\Gamma_1(\lambda) vR^\pm_0(\lambda^4).
  \end{split}
\end{equation}
Combining with Proposition \ref{prop-free estimates}, in order to obtain (\ref{thm-low-01}), it suffices to prove the following
Proposition \ref{prop-QA001Q} and Proposition \ref{prop-reg-freeterms}.

\begin{proposition}\label{prop-QA001Q}
Assume that $|V(x)|\lesssim (1+|x|)^{-7-}$. Let $\Theta_{N_0,N}(t)$ be a function defined in (\ref{ function-theta}) and $N <N'$.
 Then for each $x,y$ and $\frac{3}{2}< \alpha \leq 0$, we have
\begin{equation*}
\Big|\int_0^\infty e^{-it\lambda^2}\lambda^{3+2\alpha}\varphi_0(2^{-N}\lambda)\big[R_0^\pm(\lambda^4)v(QA^0_{0,1}Q)
vR_0^\pm(\lambda^4)\big](x,y)d\lambda\Big|
 \lesssim 2^{(3+2\alpha)N}\Theta_{N_0,N}(t),
\end{equation*}
which gives that
\begin{equation}\label{regular-1}
\sup\limits_{x,y\in \mathbb{R}^3}\Big|\int_0^\infty \chi(\lambda)e^{-it\lambda^2}\lambda^{3+2\alpha}
\big[R_0^\pm(\lambda^4)v(QA^0_{0,1}Q)vR_0^\pm(\lambda^4)\big](x,y)d\lambda\Big|\lesssim|t|^{-\frac{3+2\alpha}{2}} .
\end{equation}
As a consequence,  we have
$$
\Big\|\int_0^\infty \chi(\lambda)e^{-it\lambda^2}\lambda^{3+2\alpha}
\big[R_0^\pm(\lambda^4)v(QA^0_{0,1}Q)vR_0^\pm(\lambda^4)\big]fd\lambda\Big\|_{L^\infty}\lesssim|t|^{-\frac{3+2\alpha}{2}} \big\|f\big\|_{L^1}.
$$
\end{proposition}
\begin{proof}
 We write
\begin{equation*}
   \begin{split}
K_{1,N}^{0,\pm}(t;x,y):= \int_0^\infty e^{-it\lambda^2}\lambda^{3+2\alpha}\varphi_0(2^{-N}\lambda)\big[R_0^\pm(\lambda^4)vQA^1_{0,1}QvR_0^\pm(\lambda^4)\big](x,y)
 d\lambda.
\end{split}
\end{equation*}
Let $ F^\pm(p)=\frac{e^{\pm ip}-e^{-p}}{p}, p\geq 0$. Then $ R_0^\pm(\lambda^4)(x,y)=\frac{1}{8\pi \lambda}F^\pm(\lambda|x-y|)$.
Using the orthogonality $Qv(x)=0$ and Lemma \ref{Taylor-low}(i), one has
\begin{equation*}
   \begin{split}
&[R_0^\pm(\lambda^4) vQA^1_{0,1}Q v R_0^\pm(\lambda^4)](x,y)\\
=& \frac{1}{64\pi^2 \lambda^2}\int_{\mathbb{R}^6}F^\pm(\lambda|x-u_2|)
     [vQA^1_{0,1}Qv](u_2,u_1)F^\pm(\lambda |y-u_1|)du_1du_2\\
=&\frac{1}{64\pi^2}\int_{\mathbb{R}^6}\int_0^1\int_0^1\cos\alpha_2\cos\alpha_1(F^\pm)'(\lambda|x-\theta_2u_2|)
(F^\pm)'(\lambda|y-\theta_1u_1|)d\theta_1d\theta_2 \\
&\ \ \ \  \times |u_1||u_2|[vQA^1_{0,1}Qv](u_2,u_1)du_1du_2,
  \end{split}
\end{equation*}
where $\cos\alpha_1= \cos\alpha(y,u_1,\theta_1) $ and $\cos\alpha_2= \cos\alpha(x,u_2,\theta_2)$.

Furthermore, we have
\begin{align}\label{K1N-E1N}
K_{1,N}^{0,\pm}(t;x,y)
=&\frac{1}{64\pi^2}\int_{\mathbb{R}^6}\int_0^1\int_0^1\Big(\int_0^\infty e^{-it\lambda^2}\lambda^{3+2\alpha}
\varphi_0(2^{-N}\lambda)(F^\pm)'(\lambda|x-\theta_2u_2|)\nonumber\\
&(F^\pm)'(\lambda|y-\theta_1u_1|)d\lambda\Big)
\cos\alpha_2\cos\alpha_1d\theta_1d\theta_2|u_1||u_2|[vQA^1_{0,1}Qv](u_2,u_1)du_1du_2\nonumber\\
:= &\frac{1}{64\pi^2}\int_{\mathbb{R}^6}\int_0^1\int_0^1
E^{0,\pm}_{1,N}(t;x,y,\theta_1,\theta_2,u_1,u_2)
\cos\alpha_2\cos\alpha_1d\theta_1d\theta_2\nonumber\\
&\   \  \  \  \ \  \ \ \ \ \ \times |u_1||u_2|[vQA^1_{0,1}Qv](u_2,u_1)du_1du_2.
\end{align}

Now we begin to estimate $E^{0,\pm}_{1,N}(t;x,y,\theta_1,\theta_2,u_1,u_2)$.
In fact, let $s=2^{-N}\lambda$, then
\begin{equation*}
   \begin{split}
&E^{0,\pm}_{1,N}(t;x,y,\theta_1,\theta_2,u_1,u_2)\\
&= 2^{(4+2\alpha)N} \int_0^\infty e^{-it2^{2N}s^2}s^{3+2\alpha}
\varphi_0(s)(F^\pm)'(2^Ns|x-\theta_2u_2|)(F^\pm)'(2^Ns|y-\theta_1u_1|)ds.
  \end{split}
\end{equation*}
Notice that $ s \in \hbox{supp} \varphi_0 \subset[\frac{1}{4}, 1]$, by using integration by parts we have
\begin{align}\label{esti-E0pm1N-123}
&|E^{0,\pm}_{1,N}(t;x,y,\theta_1,\theta_2,u_1,u_2)|\nonumber\\
\lesssim & \frac{2^{(4+2\alpha)N}}{1+|t|2^{2N}}\bigg( \Big|\int_0^\infty e^{-it2^{2N}s^2} \partial_s\big( s^{2+2\alpha}
\varphi_0(s)\big) (F^\pm)'(2^Ns|x-\theta_2u_2|)(F^\pm)'(2^Ns|y-\theta_1u_1|)ds\Big|\nonumber\\
&+\Big|\int_0^\infty e^{-it2^{2N}s^2}  s^{2+2\alpha}
\varphi_0(s)\partial_s\big( (F^\pm)'(2^Ns|x-\theta_2u_2|)(F^\pm)'(2^Ns|y-\theta_1u_1|)\big)ds\Big|
\bigg)\nonumber\\
:=&\frac{2^{(4+2\alpha)N}}{1+|t|2^{2N}}\Big( |\mathcal{E}^{0,\pm}_{1,N}(t;x,y,\theta_1,\theta_2,u_1,u_2)|
+|\mathcal{E}^{0,\pm}_{2,N}(t;x,y,\theta_1,\theta_2,u_1,u_2)|\Big).
\end{align}
We first compute the second term $\mathcal{E}^{0,\pm}_{2,N}$. We have
\begin{equation*}\label{esti-E0Ntuba1}
   \begin{split}
&|\mathcal{E}^{0,\pm}_{2,N}(t;x,y,\theta_1,\theta_2,u_1,u_2)|\\
\lesssim& \Big( \Big|\int_0^\infty e^{-it2^{2N}s^2}s^{1+2\alpha}\varphi_0(s)\cdot
2^Ns|x-\theta_2u_2|(F^\pm)^{(2)}(2^Ns|x-\theta_2u_2|) (F^\pm)'(2^Ns|y-\theta_1u_1|)ds\Big|\\
&+\Big|\int_0^\infty e^{-it2^{2N}s^2}s^{1+2\alpha}\varphi_0(s)\cdot
2^Ns|y-\theta_1u_1| (F^\pm)'(2^Ns|x-\theta_2u_2|)
(F^\pm)^{(2)}(2^Ns|y-\theta_1u_1|)ds\Big| \\
:=&|\mathcal{E}^{0,\pm}_{21,N}(t;x,y,\theta_1,\theta_2,u_1,u_2)|+|\mathcal{E}^{0,\pm}_{22,N}(t;x,y,\theta_1,\theta_2,u_1,u_2)|.
  \end{split}
\end{equation*}
For  the first term $\mathcal{E}^{0,\pm}_{21,N}$. Let
$$ F^\pm_1(p)=e^{\mp ip}(F^\pm)'(p)=\frac{(\pm ip -1)+(p+1)e^{-p\mp ip}}{p^2},$$
$$F^\pm_2(p)= p e^{\mp ip}(F^\pm)^{(2)}(p) = \frac{(2\mp 2ip -p^2)- (2+2p+p^2)e^{-p\mp ip}}{p^2}.$$
Then
\begin{equation*}
   \begin{split}
 &2^Ns|x-\theta_2u_2|(F^\pm)^{(2)}(2^Ns|x-\theta_2u_2|)(F^\pm)'(2^Ns|y-\theta_1u_1|)\\
& \ \ \ \ \ \ \ \ \ = e^{\pm i 2^Ns|x-\theta_2u_2|}e^{\pm i2^Ns|y-\theta_1u_1|}
F^\pm_2(2^Ns|x-\theta_2u_2|)F^\pm_1(2^Ns|y-\theta_1u_1|).
  \end{split}
\end{equation*}
Hence, we have
\begin{equation*}
   \begin{split}
&\mathcal{E}^{0,\pm}_{21,N}(t;x,y,\theta_1,\theta_2,u_1,u_2)\\
=&\int_0^\infty e^{-it2^{2N}s^2}e^{\pm i 2^Ns(|x-\theta_2u_2|+|y-\theta_1u_1|)}
s^{1+2\alpha}\varphi_0(s)F^\pm_2(2^Ns|x-\theta_2u_2|)F^\pm_1(2^Ns|y-\theta_1u_1|)ds.
  \end{split}
\end{equation*}
It is easy to check that
$$\Big|\partial_s^k\Big(F^\pm_2(2^Ns|x-\theta_2u_2|)F^\pm_1(2^Ns|y-\theta_1u_1|)\Big)\Big| \lesssim 1,\,k=0,1.$$
By Lemma \ref{lem-LWP} with $z=(x,y,\theta_1,\theta_2,u_1,u_2)$ and $\Psi(z)=|x-\theta_2u_2|+ |y-\theta_1u_1|$,and
$$\Phi(2^Ns;z)= F^\pm_2(2^Ns|x-\theta_2u_2|)F^\pm_1(2^Ns|y-\theta_1u_1|),$$
we obtain that $\mathcal{E}^{0,\pm}_{21,N} $ is bounded by $(1+|t|2^{2N})\Theta_{N_0,N}(t)$.
Similarly, we obtain that $\mathcal{E}^{0,\pm}_{22,N} $ is controlled by the same bound.
Hence we get that  $\mathcal{E}^{0,\pm}_{2,N}$ is bounded by $(1+|t|2^{2N})\Theta_{N_0,N}(t)$.

Similarly, we obtain that $\mathcal{E}^{0,\pm}_{1,N}$ is controlled by the same bound.
By \eqref{esti-E0pm1N-123} we have
\begin{equation}\label{esti-E1N}
   \begin{split}
 |E^{0,\pm}_{1,N}(t;x,y,\theta_1,\theta_2,u_1,u_2)| &\lesssim
2^{(4+2\alpha)N}\Theta_{N_0,N}(t).
 \end{split}
\end{equation}
Since $|V(x)|\lesssim (1+|x|)^{-7-}$, by using (\ref{K1N-E1N}), (\ref{esti-E1N}) and  H\"{o}lder's inequality, one has
\begin{equation*}
   \begin{split}
|K_{1,N}^{0,\pm}(t;x,y)|
&\lesssim2^{(4+2\alpha)N}\Big(\|u_1v(u_1)\|_{L^2}\|QA_{0,1}^0Q\|_{L^2\rightarrow L^2}\|u_2v(u_2)\|_{L^2}\Big)\Theta_{N_0,N}(t)\\
&\lesssim2^{(3+2\alpha)N}\Theta_{N_0,N}(t).
\end{split}
 \end{equation*}

 Finally, by the same argument with the proof of \eqref{sum-appha=0-alphabig0}, we immediately get the desired conclusions.
\end{proof}

\begin{proposition}\label{prop-reg-freeterms}
Assume that $|V(x)|\lesssim (1+|x|)^{-7-}$.  Let $\Theta_{N_0,N}(t)$ be a function defined in (\ref{ function-theta}) and $N <N'$.
Then for each $x,y$ and $-\frac{3}{2}< \alpha \leq 0$,
\begin{equation}\label{weixiang1}
\Big|\int_0^\infty e^{-it\lambda^2}\lambda^{3+2\alpha}\varphi_0(2^{-N}\lambda)\big[R_0^\pm(\lambda^4)v\Gamma_1(\lambda) vR_0^\pm(\lambda^4)\big](x,y)d\lambda\Big|
 \lesssim 2^{(3+2\alpha)N} \Theta_{N_0, N}(t).
\end{equation}
Moreover,
\begin{equation}\label{weixiang11}
\begin{split}
\sup\limits_{x,y\in \mathbb{R}^3}\Big|\int_0^\infty\chi(\lambda) e^{-it\lambda^2}\lambda^{3+2\alpha}\big[R_0^\pm(\lambda^4)v\Gamma_1(\lambda) vR_0^\pm(\lambda^4)\big](x,y)d\lambda\Big|\lesssim|t|^{-\frac{3+2\alpha}{2}}.
\end{split}
\end{equation}
\end{proposition}
\begin{proof}
To get (\ref{weixiang1}), it's equivalent to show that
\begin{equation*}
   \begin{split}
K^{0,\pm}_{2,N}(t;x,y):=\int_0^\infty e^{-it\lambda^2}\lambda^{3+2\alpha}\varphi_0(2^{-N}\lambda)
\Big\langle [v\Gamma_1(\lambda) v]\big(R_0^\pm(\lambda^4)(*, y)\big)(\cdot), ~(R_0^\pm)^*(\lambda^4)(x,\cdot)   \Big\rangle d\lambda
\end{split}
\end{equation*}
is bounded by  $ 2^{(3+2\alpha)N}\Theta_{N_0, N}(t)$.

Let $ F^\pm(p)= \frac{e^{\pm ip} -e^{-p}}{p} (p\geq0)$, then $ R_0^\pm(\lambda^4)(x,y)= \frac{1}{8\pi\lambda}F^\pm(\lambda|x-y|)$.
Hence one has
\begin{equation*}
   \begin{split}
&\Big\langle [v\Gamma_1(\lambda) v]\big(R_0^\pm(\lambda^4)(*,y)\big)(\cdot),~ R_0^\mp(\lambda^4)(x,\cdot)   \Big\rangle\\
&=\frac{1}{64\pi^2\lambda^2}\Big\langle[v\Gamma_1(\lambda) v]
\big(F^\pm(\lambda|*-y|)\big)(\cdot),~
F^\mp(\lambda|x-\cdot|)  \Big\rangle
:=\frac{1}{64\pi^2\lambda^2}E^{0,\pm}_{2}(\lambda;x,y).
  \end{split}
\end{equation*}
Let $\lambda =2^Ns$, then
\begin{equation*}
   \begin{split}
K^{0,\pm}_{2,N}(t;x,y)= \frac{2^{(2+2\alpha) N}}{64\pi^2}
\int_0^\infty e^{-it2^{2N}s^2}s^{1+2\alpha}\varphi_0(s)E^{0,\pm}_{2}(2^Ns;x,y) ds.
  \end{split}
\end{equation*}
Notice that $ s \in \hbox{supp} \varphi_0 \subset[\frac{1}{4}, 1]$, by using integration by parts we have
\begin{equation}\label{esti-K0pm2N-123}
   \begin{split}
|K^{0,\pm}_{2,N}(t;x,y)|
\lesssim & \frac{2^{(2+2\alpha) N}}{1+|t|2^{2N}} \bigg(
 \Big| \int_0^\infty e^{-it2^{2N}s^2} \partial_s\big(s^{2\alpha}\varphi_0(s)\big)E^{0,\pm}_{2}(2^Ns;x,y) ds \Big| \\
  &+\Big| \int_0^\infty e^{-it2^{2N}s^2}s^{2\alpha}\varphi_0(s)\partial_s\big(E^{0,\pm}_{2}(2^Ns;x,y)\big) ds \Big| \bigg)\\
 :=& \frac{2^{(2+2\alpha) N}}{1+|t|2^{2N}}\Big( |\mathcal{E}^{0,\pm}_{1,N}(t;x,y)|+|\mathcal{E}^{0,\pm}_{2,N}(t;x,y)|   \Big).
  \end{split}
\end{equation}

We first estimate $\mathcal{E}^{0,\pm}_{2,N}$.  Since
\begin{equation*}
   \begin{split}
\partial_s\big(E^{0,\pm}_{2}(2^Ns;x,y)\big)
=& \big\langle[v\partial_s\Gamma_1(2^Ns) v]\big(F^\pm(2^Ns|*-y|)\big)(\cdot),~
F^\mp(2^Ns|x-\cdot|)\big\rangle\\
&+\big\langle[v\Gamma_1(2^Ns) v]\big(\partial_sF^\pm(2^Ns|*-y|)\big)(\cdot),~
F^\mp(2^Ns|x-\cdot|)\big\rangle\\
&+\big\langle[v\Gamma_1(2^Ns) v]\big(F^\pm(2^Ns|*-y|)\big)(\cdot),~
\partial_sF^\mp(2^Ns|x-\cdot|)\big\rangle\\
:=&E^{0,\pm}_{21}(2^Ns;x,y)+E^{0,\pm}_{22}(2^Ns;x,y)
+E^{0,\pm}_{23}(2^Ns;x,y).
  \end{split}
\end{equation*}
Then
\begin{equation*}
   \begin{split}
\mathcal{E}^{0,\pm}_{2,N}(t;x,y)=&
\int_0^\infty e^{-it2^{2N}s^2}s^{2\alpha}\varphi_0(s)
 \big(E^{0,\pm}_{21}
 +E^{0,\pm}_{22}+E^{0,\pm}_{23}\big)(2^Ns;x,y)\Big)ds\\
 :=& \mathcal{E}^{0,\pm}_{21,N}(t;x,y)+\mathcal{E}^{0,\pm}_{22,N}(t;x,y)+\mathcal{E}^{0,\pm}_{23,N}(t;x,y).
  \end{split}
\end{equation*}

For the first term $\mathcal{E}^{0,\pm}_{21,N}$. Since $$\lambda\|\partial_\lambda\Gamma_1(\lambda)\|_{L^2\rightarrow L^2}+\lambda^2\|\partial_\lambda^2\Gamma_1(\lambda)\|_{L^2\rightarrow L^2}\lesssim\lambda,$$ then
$$\big\|\partial_s^k\big(\partial_s\Gamma_1(2^Ns)\big)\big\|_{L^2\rightarrow L^2} \lesssim 2^N,\, k=0,1.$$
Similarly,
$$\big|\partial_s^k \big(E^{0,\pm}_{21}(2^Ns;x,y)\big)\big| \lesssim  2^N,  \,k=0,1.$$
By integration by parts,  we obtain that $ \mathcal{E}^{0,\pm}_{21,N}(t;x,y)$ is bounded by $2^N (1+|t|2^{2N})^{-1}$.

We turn to compute  $\mathcal{E}^{0,\pm}_{22,N}$. Let $ F^\mp_0(p)= \frac{1-e^{-p}e^{\pm i p}}{p} $, then
$F^\mp(p)= e^{\mp ip}F^\mp_0(p)$.
Since
$$ \partial_sF^\pm(2^Ns|*-y|)=2^N|*-y|(F^\pm)'(2^Ns|*-y|)
:= e^{\pm i 2^Ns|*-y|} s^{-1}F^\pm_1(2^Ns|*-y|),$$
where$$F^\pm_1(p)=pe^{\mp ip}(F^\pm)'(p)=\frac{(\pm ip -1)+(p+1)e^{-p}e^{\mp ip}}{p},$$
thus we have,
\begin{equation*}
   \begin{split}
E^{0,\pm}_{22}(2^Ns;x,y)
= &e^{\pm i2^Ns(|x|+|y|)}s^{-1}
\Big\langle [v\Gamma_1(2^Ns)v]\big(e^{\pm i2^Ns(|*-y|-|y|)}F^\pm_1(2^Ns|*-y|)\big)(\cdot),\\
& e^{\mp i2^Ns(|x-\cdot|-|x|)}F^\mp_0(2^Ns|x-\cdot|)\Big\rangle
:= e^{\pm i2^Ns(|x|+|y|)}s^{-1}\widetilde{E}^{0,\pm}_{22}(2^Ns;x,y).
  \end{split}
\end{equation*}
Hence, we have
\begin{equation*}
   \begin{split}
\mathcal{E}^{0,\pm}_{22,N}(t;x,y)
=&\int_0^\infty e^{-it2^{2N}s^2}
e^{\pm i 2^Ns(|x|+|y|)}s^{-1+2\alpha}\varphi_0(s)\widetilde{E}^{0,\pm}_{22}(2^Ns;x,y) ds.
  \end{split}
\end{equation*}
Note that
\begin{equation*}
   \begin{split}
&\big|\partial_s^k\big(e^{\pm i2^Ns(|*-y|-|y|)}F^\pm_1(2^Ns|*-y|)\big)\big|\lesssim2^{kN}\langle*\rangle, \, k=0,1,\\
&\big|\partial_s^k\big(e^{\mp i2^Ns(|x-\cdot|-|x|)}F^\mp_0(2^Ns|x-\cdot|)\big)\big|\lesssim2^{kN}\langle\cdot\rangle, \, k=0,1.
\end{split}
\end{equation*}
Since $|V(x)|\lesssim(1+|x|)^{-7-}$, by H\"{o}lder's inequality we have
\begin{equation*}
   \begin{split}
 \big|\partial_s^k \widetilde{E}^{0,\pm}_{22}(2^Ns;x,y)\big|&\lesssim\sum\limits_{k=0}^1\big\|v(\cdot)\langle \cdot\rangle^{1-k}\big\|^2_{L^2}
\big\|\partial_s^k \Gamma_1(2^Ns)\big\|_{L^2\rightarrow  L^2}\lesssim 2^N.
   \end{split}
\end{equation*}
By Lemma \ref{lem-LWP} again with $ z=(x,y)$, $\Psi(z) =|x|+ |y|$ and  $\Phi(2^Ns;z)= \widetilde{E}^{0,\pm}_{22}(2^Ns;x,y)$,
we obtain that $\mathcal{E}^{0,\pm}_{22,N}$ is bounded by $2^{N } (1+|t|2^{2N})\Theta_{N_0, N}(t)$.
Similar to get that $\mathcal{E}^{0,\pm}_{23,N}$ is controlled  by the same bound.
Hence we obtain that $\mathcal{E}^{0,\pm}_{2,N}$ is bounded by  $2^{N } (1+|t|2^{2N})\Theta_{N_0, N}(t)$.

Similarly, we obtain that $\mathcal{E}^{0,\pm}_{1,N}$ is controlled by
the same bounds. By \eqref{esti-K0pm2N-123}, we immediately obtain that
$K^{0,\pm}_{2,N}$ is bounded by $2^{(3+2\alpha)N}\Theta_{N_0, N}(t)$. Hence we obtain that  (\ref{weixiang1}) holds.

Finally, by the same argument with the proof of \eqref{sum-appha=0-alphabig0},
we immediately get the desired conclusions.
\end{proof}

\subsubsection{\textbf{The first kind of resonance}}
If zero is the first kind of resonance  of $H$, then by using (\ref{id-RV}) and (\ref{thm-resoinver-M1 })
one has
\begin{equation}\label{RV-first resonance}
   \begin{split}
R_V^\pm(\lambda^4)
=& R^\pm_0(\lambda^4)
- R^\pm_0(\lambda^4)v\Big(\lambda^{-1} S_1A^1_{-1,1} S_1\Big)vR^\pm_0(\lambda^4)-R^\pm_0(\lambda^4)v\\
&\times \Big(S_1A^1_{0,1}+A^1_{0,2}S_1+QA^1_{0,3}Q \Big)vR^\pm_0(\lambda^4)-R^\pm_0(\lambda^4)v\Gamma_1(\lambda)vR^\pm_0(\lambda^4).
\end{split}
\end{equation}

In order to obtain the estimates (\ref{thm-low-02}), comparing with the analysis of regular case, it is enough to prove the following proposition.
\begin{proposition}\label{prop-first-S1A-101S1}
Assume that $|V(x|)\lesssim (1+|x|)^{-11-}$. Then for any $-\frac{3}{2}< \alpha \leq 0$,
\begin{equation*}
   \begin{split}
&\sup\limits_{x,y\in \mathbb{R}^3}\Big|\int_0^\infty \chi(\lambda)e^{-it\lambda^2}\lambda^{3+2\alpha}\varphi_0(2^{-N}\lambda) \lambda^{-1}\big[R^\pm_0(\lambda^4)v S_1A^1_{-1,1}S_1 vR^\pm_0(\lambda^4)\big](x,y)
d\lambda\Big|\lesssim|t|^{-\frac{3+2\alpha}{2}},\\
&\sup\limits_{x,y\in \mathbb{R}^3}\Big|\int_0^\infty \chi(\lambda)e^{-it\lambda^2}\lambda^{3+2\alpha}\varphi_0(2^{-N}\lambda) \big[R^\pm_0(\lambda^4)v S_1A^1_{0,1} vR^\pm_0(\lambda^4)\big](x,y)
d\lambda\Big|\lesssim|t|^{-\frac{3+2\alpha}{2}}.
\end{split}
\end{equation*}
\end{proposition}
\begin{proof}
Here we use the orthogonality $S_1v=0$, by the same argument with the proof of
Proposition \ref{prop-QA001Q}, we obtain the desired conclusions.
\end{proof}

\subsubsection{\textbf{The second kind of resonance}}
If zero is the second kind of resonance  of $H$, then using (\ref{id-RV}) and
(\ref{thm-resoinver-M2 }) one has
\begin{align*}\label{Rv-secondreso}
R^\pm_V(\lambda^4)&=R^\pm_0(\lambda^4)-R^\pm_0(\lambda^4)v\Big(\lambda^{-3}S_2A^2_{-3,1}S_2\Big)vR^\pm_0(\lambda^4)
-R^\pm_0(\lambda^4)v\Big(  \lambda^{-2}S_2A^2_{-2,1}S_1\nonumber\\
&+ \lambda^{-2}S_1A^2_{-2,2}S_2\Big)vR^\pm_0(\lambda^4)
-R^\pm_0(\lambda^4)v\Big(  \lambda^{-1}S_2A^2_{-1,1}+\lambda^{-1}A^2_{-1,2}S_2
+\lambda^{-1}S_1A^2_{-1,3}S_1\Big) \nonumber\\
&\times vR^\pm_0(\lambda^4)-R^\pm_0(\lambda^4)v\Big( S_1A^2_{0,1} + A^2_{0,2}S_1
+Q A^2_{0,3}Q \Big)vR^\pm_0(\lambda^4)
-R^\pm_0(\lambda^4)v\Gamma_1(\lambda)vR^\pm_0(\lambda^4).
\end{align*}
In order to prove Theorem \ref{thm-low}(iii), comparing with the proof of regular case and the first kind resonance,
it is enough to estimate integrals with the following three terms:
 \begin{equation}\label{term-Omega-second123}
   \begin{split}
 \Omega_{2,1}(\lambda):=& R^\pm_0(\lambda^4)v\big(\lambda^{-3}S_2A^2_{-3,1}S_2\big)vR^\pm_0(\lambda^4),\\
\Omega_{2,2}(\lambda):=& R^\pm_0(\lambda^4)v\big(\lambda^{-2}S_2A^2_{-2,1}S_1\big)vR^\pm_0(\lambda^4),\\
\Omega_{2,3}(\lambda):=& R^\pm_0(\lambda^4)v\big(\lambda^{-1}S_2A^2_{-1,1}\big)vR^\pm_0(\lambda^4).
\end{split}
\end{equation}

  Recall $ F^\pm(p)= \frac{e^{\pm ip} -e^{-p}}{p} (p\geq0)$ and $R^\pm_0(\lambda^4)(x,y)=\frac{1}{8\pi \lambda}F^\pm(\lambda|x-y|)$. Since $(F^\pm)'(0) \neq 0$ so it doesn't satisfy the condition of Lemma \ref{Taylor-low}(ii), hence we can't make full use of the orthogonality of $S_2$ ( i.e. $S_2(x_iv)=0, \ i=1,2,3$ ). In order to use orthogonality $S_2(x_jv)=0$  to improve the time decay of solution operator $\cos (t\sqrt{H})$ and $\frac{\sin(t\sqrt{H})}{\sqrt{H}}$, we need to subtract a specific operator to satisfy the conditions of Lemma \ref{Taylor-low}(ii). Then we can make full use of the orthogonality of $S_2$.

Recall that  $\displaystyle G_0=\frac{|x-y|}{8\pi}$, let $\widetilde{F}^\pm(p)= \frac{e^{\pm ip}-e^{-p}}{p} + p,\,\, p\in \mathbb{R}$, then
\begin{equation}\label{Fpmtuba}
   \begin{split}
R^\pm_0(\lambda^4)(x,y)-G_0= \frac{1}{8 \pi \lambda }\widetilde{F}^\pm(\lambda|x-y|),
\end{split}
\end{equation}
 $\widetilde{F}^\pm(p)\in C^2(\mathbb{R})$ and $(\widetilde{F}^\pm)'(0)= 0$. Hence $\widetilde{F}^\pm(p)$ satisfies the condition of Lemma \ref{Taylor-low}(ii). We now begin to estimate the terms $\Omega_{2,i}(\lambda)(i=1,2,3)$ in \eqref{term-Omega-second123}.

\vskip0.1cm
Firstly, we deal with the first term $\Omega_{2,1}(\lambda)$ in \eqref{term-Omega-second123}. We have
 \begin{align}\label{secres-imp-termS2S2}
\Omega_{2,1}(\lambda)=&\big(R^\pm_0(\lambda^4)-G_0\big) v(\lambda^{-3}S_2A^2_{-3,1}S_2)v\big(R^\pm_0(\lambda^4)-G_0\big)
+\big(R^\pm_0(\lambda^4)-G_0\big) v(\lambda^{-3}S_2A^2_{-3,1}S_2)\nonumber\\
&\times vG_0+G_0v(\lambda^{-3}S_2A^2_{-3,1}S_2)v\big(R^\pm_0(\lambda^4)-G_0\big)
+G_0v(\lambda^{-3}S_2A^2_{-3,1}S_2)vG_0\nonumber\\
:=&\Gamma^2_{-3,1}(\lambda)+\Gamma^2_{-3,2}(\lambda)
+\Gamma^2_{-3,3}(\lambda)+\Gamma^2_{-3,4}(\lambda).
\end{align}
\vskip0.1cm
\begin{proposition}\label{prop-sencond-S2S2-G0}\label{prop-third-041}
Assume $|V(x)|\lesssim (1+|x|)^{-19-}$. Let $\Gamma^2_{-3,j}(\lambda)(j=1,2,3,4)$ be  operators defined in (\ref{secres-imp-termS2S2}).
Then
\begin{equation}\label{pro-second-bad1}
   \begin{split}
\sup\limits_{x,y\in \mathbb{R}^3}\Big|\int_0^\infty \chi(\lambda)e^{-it\lambda^2}\lambda^{3+2\alpha}\Gamma^2_{-3,1}(\lambda)(x,y)d\lambda\Big|
\lesssim |t|^{-\frac{3+2\alpha}{2}},\, -\frac{3}{2}<\alpha \leq0,
\end{split}
\end{equation}
\begin{equation}\label{pro-second-bad2}
   \begin{split}
\sup\limits_{x,y\in \mathbb{R}^3}\Big|\int_0^\infty\chi(\lambda) e^{-it\lambda^2}\lambda^{3+2\alpha}\Gamma^2_{-3,j}(\lambda)(x,y)d\lambda\Big|
\lesssim|t|^{-\frac{2+2\alpha}{2}}, \,-1<\alpha \leq0,\,  j=2,3,
\end{split}
\end{equation}
\begin{equation}\label{pro-second-bad3}
   \begin{split}
\sup\limits_{x,y\in \mathbb{R}^3}\Big|\int_0^\infty\chi(\lambda) e^{-it\lambda^2}\lambda^{3+2\alpha}\Gamma^2_{-3,4}(\lambda)(x,y)d\lambda\Big|
\lesssim|t|^{-\frac{1+2\alpha}{2}}, \,-\frac{1}{2}<\alpha \leq0.
\end{split}
\end{equation}
\end{proposition}
\begin{proof}
 We first estimate the first term $\Gamma^2_{-3,1}(\lambda)$. For each $N$, we write
\begin{equation*}
   \begin{split}
&\widetilde{K}^{2,\pm}_{1,N}(t;x,y)\\
&=\int_0^\infty e^{-it\lambda^2}\lambda^{3+2\alpha}\varphi_0(2^{-N}\lambda)
\big[\big(R^\pm_0(\lambda^4)-G_0\big)v(\lambda^{-3}S_2A^2_{-3,1}S_2)
v\big(R^\pm_0(\lambda^4)-G_0\big)\big](x,y)d\lambda.
\end{split}
\end{equation*}
Notice that
 $$R^\pm_0(\lambda^4)(x,y)-G_0:= \frac{1}{8 \pi \lambda }\widetilde{F}^\pm(\lambda|x-y|),$$
by Lemma \ref{Taylor-low}(ii) and the orthogonality $S_2x_jv(x)=S_2v(x)=0(j=1,2,3)$, then
\begin{equation*}
   \begin{split}
&\big[\big(R^\pm_0(\lambda^4)-G_0\big)v(\lambda^{-3}S_2A^2_{-3,1}S_2)v\big(R^\pm_0(\lambda^4)-G_0\big)\big](x,y)\\
=& \frac{1}{64\pi^2\lambda^5}\int_{\mathbb{R}^6}\widetilde{F}^\pm(\lambda|x-u_2|)v(u_2)
(S_2A^2_{-3,1}S_2)(u_2,u_1)v(u_1)\widetilde{F}^\pm(\lambda|y-u_1|) du_1du_2\\
=& \frac{1}{64\pi^2\lambda}\int_{\mathbb{R}^6}\int_0^1\int_0^1(1-\theta_1)(1-\theta_2)
\Big(\frac{(\widetilde{F}^\pm)'(\lambda|x-\theta_2u_2|)}{\lambda|x-\theta_2u_2|}\sin^2\alpha_2
+(\widetilde{F}^\pm)^{(2)}(\lambda|x-\theta_2u_2|)\\
&\times \cos^2\alpha_2\Big)
\Big(\frac{(\widetilde{F}^\pm)'(\lambda|y-\theta_1u_1|)}{\lambda|y-\theta_1u_1|}\sin^2\alpha_1
+(\widetilde{F}^\pm)^{(2)}(\lambda|y-\theta_1u_1|)\cos^2\alpha_1\Big)d\theta_1 d\theta_2\\
&\ \ \ \ \ \ \ \ \ \  \times |u_1|^2|u_2|^2v(u_2)v(u_1) (S_2A^2_{-3,1}S_2)(u_2,u_1) du_1du_2.
\end{split}
\end{equation*}
Furthermore, we have
\begin{equation*}
   \begin{split}
&\widetilde{K}^{2,\pm}_{1,N}(t;x,y)\\
=&\frac{1}{64\pi^2}\int_{\mathbb{R}^6}\int_0^1\int_0^1 \Big[ \int_0^\infty e^{-it\lambda^2}\lambda^{2+2\alpha}\varphi_0(2^{-N}\lambda)
\Big(\frac{(\widetilde{F}^\pm)'(\lambda|x-\theta_2u_2|)}{\lambda|x-\theta_2u_2|}\sin^2\alpha_2\\
&+(\widetilde{F}^\pm)^{(2)}(\lambda|x-\theta_2u_2|)\cos^2\alpha_2\Big)
\Big(\frac{(\widetilde{F}^\pm)'(\lambda|y-\theta_1u_1|)}{\lambda|y-\theta_1u_1|}\sin^2\alpha_1
+(\widetilde{F}^\pm)^{(2)}(\lambda|y-\theta_1u_1|)\\
&\ \ \ \ \times\cos^2\alpha_1\Big)d\lambda\Big](1-\theta_1)(1-\theta_2)|u_1|^2|u_2|^2v(u_2)v(u_1)(S_2A^2_{-3,1}S_2)(u_2,u_1)
d\theta_1d\theta_2 du_1du_2.
\end{split}
\end{equation*}
Let
\begin{equation*}
   \begin{split}
&\widetilde{E}^{2,\pm}_{1,N}(t;x,y,\theta_1,\theta_2,u_1,u_2)\\
=&\int_0^\infty e^{-it\lambda^2}\lambda^{2+2\alpha}\varphi_0(2^{-N}\lambda)
\Big(\frac{(\widetilde{F}^\pm)'(\lambda|x-\theta_2u_2|)}{\lambda|x-\theta_2u_2|}\sin^2\alpha_2
+(\widetilde{F}^\pm)^{(2)}(\lambda|x-\theta_2u_2|)\cos^2\alpha_2\Big)\\
&\ \ \ \ \ \ \ \ \ \ \ \ \ \ \ \times\Big(\frac{(\widetilde{F}^\pm)'(\lambda|y-\theta_1u_1|)}{\lambda|y-\theta_1u_1|}\sin^2\alpha_1
+(\widetilde{F}^\pm)^{(2)}(\lambda|y-\theta_1u_1|)\cos^2\alpha_1\Big)d\lambda,
\end{split}
\end{equation*}
then we have
\begin{equation}\label{esti-K2pmtuba-E2pmtua}
   \begin{split}
\big|\widetilde{K}^{2,\pm}_{1,N}(t;x,y)\big|
\lesssim& \int_{\mathbb{R}^6}\int_0^1\int_0^1
\big| \widetilde{E}^{2,\pm}_{1,N}(t;x,y,\theta_1,\theta_2,u_1,u_2)\big|
|u_1|^2|u_2|^2|v(u_2)v(u_1)|\\
&\ \ \ \ \ \ \ \ \ \  \ \ \times|(S_2A^2_{-3,1}S_2)(u_2,u_1)|d\theta_1d\theta_2 du_1du_2.
\end{split}
\end{equation}
Let $s=2^{-N}\lambda$, $r_1=2^N|y-\theta_1u_1|$ and $r_2=2^N|x-\theta_2u_2|$, then
\begin{equation*}
   \begin{split}
&\widetilde{E}^{2,\pm}_{1,N}(t;x,y,\theta_1,\theta_2,u_1,u_2)\\
=&2^{(3+2\alpha)N}\int_0^\infty e^{-it2^{2N}s^2}s^{2+2\alpha}\varphi_0(s)
\Big(\frac{(\widetilde{F}^\pm)'(2^Ns|x-\theta_2u_2|)}{2^Ns|x-\theta_2u_2|}\sin^2\alpha_2
+(\widetilde{F}^\pm)^{(2)}(2^Ns|x-\theta_2u_2|)\\
&\ \ \ \ \ \ \,\times \cos^2\alpha_2\Big)
\Big(\frac{(\widetilde{F}^\pm)'(2^Ns|y-\theta_1u_1|)}{2^Ns|y-\theta_1u_1|}\sin^2\alpha_1
+(\widetilde{F}^\pm)^{(2)}(2^Ns|y-\theta_1u_1|)\cos^2\alpha_1\Big)ds\\
=&2^{(3+2\alpha)N}\int_0^\infty e^{-it2^{2N}s^2}s^{2+2\alpha}\varphi_0(s)
\prod_{j=1}^2\Big(\frac{(\widetilde{F}^\pm)'(r_js)}{r_js}\sin^2\alpha_j
+(\widetilde{F}^\pm)^{(2)}(r_js)\cos^2\alpha_j\Big)ds.
\end{split}
\end{equation*}
 By integration by parts, we obtain that
\begin{align}\label{esti-E2pm11N-123}
&|\widetilde{E}^{2,\pm}_{1,N}(t;x,y,\theta_1,\theta_2,u_1,u_2)|\nonumber\\
\lesssim & \frac{2^{(3+2\alpha)N}}{1+|t|2^{2N}}\bigg( \Big| \int_0^\infty e^{-it2^{2N}s^2}\partial_s\Big(s^{1+2\alpha}\varphi_0(s)\Big)
\prod_{j=1}^2\Big(\frac{(\widetilde{F}^\pm)'(r_js)}{r_js}\sin^2\alpha_j
+(\widetilde{F}^\pm)^{(2)}(r_js)\cos^2\alpha_j\Big)ds\Big| \nonumber\\
&+ \Big| \int_0^\infty e^{-it2^{2N}s^2}s^{1+2\alpha}\varphi_0(s)
\partial_s\prod_{j=1}^2\Big(\frac{(\widetilde{F}^\pm)'(r_js)}{r_js}\sin^2\alpha_j
+(\widetilde{F}^\pm)^{(2)}(r_js)\cos^2\alpha_j\Big)ds\Big| \bigg)\nonumber\\
:= &\frac{2^{(3+2\alpha)N}}{1+|t|2^{2N}} \Big( |\mathcal{E}^{2,\pm}_{1,N}(t;x,y,\theta_1,\theta_2,u_1,u_2)|+|\mathcal{E}^{2,\pm}_{2,N}(t;x,y,\theta_1,\theta_2,u_1,u_2)|\Big).
\end{align}

We first estimate the term $\mathcal{E}^{2,\pm}_{2,N}$. Let
$$\partial_s \prod_{j=1}^2\Big(\frac{(\widetilde{F}^\pm)'(r_js)}{r_js}\sin^2\alpha_j
+(\widetilde{F}^\pm)^{(2)}(r_js)\cos^2\alpha_j\Big)
:= e^{\pm ir_1s}e^{\pm ir_2s}s^{-1}\widetilde{F}^\pm _{\alpha_1,\alpha_2}(r_1s,r_2s),$$
then
\begin{equation*}
   \begin{split}
\big|\mathcal{E}^{2,\pm}_{2,N}\big|
\lesssim& \Big| \int_0^\infty e^{-it2^{2N}s^2}
e^{\pm i 2^Ns (|x-\theta_2u_2|+|y-\theta_1u_1|)}s^{2\alpha}\varphi_0(s)
 \times\widetilde{F}^\pm _{\alpha_1,\alpha_2}(2^Ns|x-\theta_2u_2| ,2^Ns|y-\theta_1u_1|)\Big|.
\end{split}
\end{equation*}
Note that
$$\big|\partial_s^k\widetilde{F}^\pm _{\alpha_1,\alpha_2}(2^Ns|x-\theta_2u_2| ,2^Ns|y-\theta_1u_1|)
 \lesssim1, \,k=0,1,$$
by Lemma \ref{lem-LWP} with $z=(x,y,\theta_1,\theta_2,u_1,u_2)$, $\Psi(z)= |x-\theta_2u_2|+|y-\theta_1u_1|$, and
$$
\Phi(2^Ns,z)= \widetilde{F}^\pm _{\alpha_1,\alpha_2}(2^Ns|x-\theta_2u_2| ,2^Ns|y-\theta_1u_1|),$$
we obtain that $\mathcal{E}^{2,\pm}_{2,N}$ is bounded by $(1+|t|2^{2N})\Theta_{N_0, N}(t)$.
Similarly,  we get that $\mathcal{E}^{2,\pm}_{1,N}$  is controlled by the same bound.
Hence $\widetilde{E}^{2,\pm}_{1,N}$ is bounded by $(1+|t|2^{2N})\Theta_{N_0, N}(t)$.
By \eqref{esti-E2pm11N-123} and H\"{o}lder's inequality we obtain that
$ \widetilde{K}^{2,\pm}_{1,N}$ is bounded by $ 2^{(3+2\alpha) N}\Theta_{N_0, N}(t)$.
By the same summing way with the proof of \eqref{sum-appha=0-alphabig0}, we immediately obtain \eqref{pro-second-bad1}.

For the term $\Gamma^2_{-3,2}(\lambda)$. We use the orthogonality $S_2x_iv=0, \ i=1,2,3$ for the left hand of $\Gamma^2_{-3,2}(\lambda)$
and $S_2v=0$ for the right hand of  $\Gamma^2_{-3,2}(\lambda)$,  by the same argument with the proof of the term $\Gamma^2_{-3,1}(\lambda)$,
we immediately obtain the desired conclusion. Similarly, we also get the desired integral estimate for the term $\Gamma^2_{-3,3}(\lambda)$.

For the term $\Gamma^2_{-3,4}(\lambda)$. We use the orthogonality $S_2v=0$,  by the same argument with the proof of
Proposition \ref{prop-QA001Q}, we immediately obtain \eqref{pro-second-bad3}.
\end{proof}

Notice that the integral estimates \eqref{pro-second-bad2} and \eqref{pro-second-bad3} doesn't hold for $\alpha=-1$,
we can not reduce the solution operator $\frac{\sin(t\sqrt{H})}{\sqrt{H}}$ into the form $H^{-\frac{1}{2}}e^{-it\sqrt{H}}$, which is worse in
computing the integral. In order to obtain the
decay estimate of solution operator $\frac{\sin(t\sqrt{H})}{\sqrt{H}}$, we need directly estimate the following integral,
\begin{equation}\label{esti--total-K1234}
K_{t,i}(x,y):=\int_0^\infty\chi(\lambda) \sin(t\lambda^2)\lambda\Gamma^2_{-3,i}(\lambda)(x,y)d\lambda,\, i=2,3,4.
\end{equation}

\begin{proposition}\label{prop-sencond-S2S2-G1}
Assume that $|V(x)|\lesssim (1+|x|)^{-19-}$. Let $K_{t,i}(x,y)(i=2,3,4)$ be the integrals defined in (\ref{esti--total-K1234}). Then
\begin{equation}\label{pro-esti-sin1}
\sup\limits_{x,y\in \mathbb{R}^3}\big |K_{t,i}(x,y)\big|\lesssim 1,\, i=2,3, \,\ \ \hbox{and}\,\ \
\sup\limits_{x,y\in \mathbb{R}^3}\big |K_{t,4}(x,y)\big|\lesssim |t|^\frac{1}{2}.
\end{equation}
\end{proposition}
\begin{proof}
We only  estimate the bound of $ K_{t,4}(x,y)$, similar to get the bounds of $ K_{t,i}(x,y), i=2,3$.
Since $\Gamma^2_{-3,4}(\lambda)=G_0v(\lambda^{-3}S_2A^2_{-3,1}S_2)vG_0$ and $G_0=-\frac{|x-y|}{8\pi}$, then
\begin{equation*}
\begin{split}
K_{t,4}(x,y)=&\int_0^\infty\chi(\lambda) \sin(t\lambda^2)\lambda^{-2}[G_0vS_2A^2_{-3,1}S_2vG_0](x,y)d\lambda\\
=&\frac{1}{64\pi^2}\Big(\int_0^\infty\chi(\lambda) \sin(t\lambda^2)\lambda^{-2}d\lambda\Big)\int_{\mathbb{R}^6} |x-u_1|v(u_1)(S_2A^2_{-3,1}S_2)(u_1,u_2)\\
&\times v(u_2)|y-u_2|du_1du_2.
\end{split}
\end{equation*}
Notice that
\begin{equation*}
\int_0^\infty \chi(\lambda)\sin(t\lambda^2)\lambda^{-2}d\lambda
=  \sqrt{|t|}\int_0^\infty \chi\Big(\sqrt{\frac{u}{t}}\Big)\frac{\sin u}{u}  \frac{1}{2\sqrt{u}} du,
\end{equation*}
hence
\begin{equation}\label{growth part}
\Big|\int_0^\infty\chi(\lambda)\sin(t\lambda^2)\lambda^{-2}d\lambda\Big|
\lesssim  |t|^\frac{1}{2}.
\end{equation}
 For the kernel $[G_0vS_2A^2_{-3,1}S_2vG_0](x,y)$, by the orthogonality $S_2v=0$ and by H\"{o}lder's inequality, we obtain that
\begin{equation*}
\begin{split}
\big|[G_0vS_2A^2_{-3,1}&S_2vG_0](x,y)\big|=\frac{1}{64\pi^2}\Big|\int_{\mathbb{R}^6} (|x-u_1|-|x|)v(u_1)(S_2A^2_{-3,1}S_2)(u_1,u_2) v(u_2)\\
 \ &\times(|y-u_2|-|y|)du_1du_2\Big|\lesssim\big\||u_1|v(u_1)\|_{L^2} \big\|S_2A^2_{-3,1}S_2\|_{L^2\rightarrow L^2} \big\||u_2|v(u_2)\big\|_{L^2}\lesssim1.
\end{split}
\end{equation*}
Hence,
$$\sup\limits_{x,y\in \mathbb{R}^3}\big |K_{t,4}(x,y)\big|\lesssim  |t|^\frac{1}{2},$$
which gives
$$\|K_{t,4}\|_{L^1\rightarrow L^\infty}\lesssim  |t|^\frac{1}{2}.$$

The proof of this proposition is completed.
\end{proof}

Secondly, we estimate the second term $\Omega_{2,2}(\lambda)$ in \eqref{term-Omega-second123}.
We have
\begin{equation}\label{secres-imp-termS2S1}
   \begin{split}
\Omega_{2,2}(\lambda)
=&(R^\pm_0(\lambda^4)-G_0)v(\lambda^{-2}S_2A^2_{-2,1}S_1)vR^\pm_0(\lambda^4)+G_0v(\lambda^{-2}S_2A^2_{-2,1}S_1)vR^\pm_0(\lambda^4)\\
:=&\Gamma^2_{-2,1}(\lambda)+\Gamma^2_{-2,2}(\lambda).
\end{split}
\end{equation}

\begin{proposition}\label{prop-sencond-S2S1-S1S2G0}
 Assume that $|V(x)|\lesssim (1+|x|)^{-19-}$. Let $\Gamma^2_{-2,j}(\lambda)(j=1,2)$ be operators defined in
(\ref{secres-imp-termS2S1}). Then
\begin{equation}\label{pro-secondkind-second-term1}
\sup\limits_{x,y\in \mathbb{R}^3}\Big|\int_0^\infty\chi(\lambda) e^{-it\lambda^2}\lambda^{3+2\alpha}\Gamma^2_{-2,1}(\lambda)(x,y)d\lambda\Big|
\lesssim |t|^{-\frac{3+2\alpha}{2}}, \, -\frac{3}{2} < \alpha \leq 0,
\end{equation}
and
\begin{equation}\label{pro-secondkind-second-term2}
\sup\limits_{x,y\in \mathbb{R}^3}\Big|\int_0^\infty \chi(\lambda)e^{-it\lambda^2}\lambda^{3+2\alpha}\Gamma^2_{-2,2}(\lambda)(x,y)d\lambda\Big|
\lesssim |t|^{-\frac{2+2\alpha}{2}}, \, -1 < \alpha \leq 0.
\end{equation}
\end{proposition}
\begin{proof}
By Lemma \ref{Taylor-low}(i)(ii), using the same method with the proofs of Proposition \ref{prop-QA001Q} and Proposition \ref{prop-sencond-S2S2-G0},
we obtain \eqref{pro-secondkind-second-term1}. By using Lemma \ref{Taylor-low}(i), by the same argument with the proof of  Proposition \ref{prop-QA001Q}, we obtain \eqref{pro-secondkind-second-term2}.
\end{proof}

Finally, we deal with the term $\Omega_{2,3}(\lambda)$ in \eqref{term-Omega-second123}. We have
  \begin{equation}\label{secres-imp-termS2-le}
   \begin{split}
\Omega_{2,3}(\lambda)
=&(R^\pm_0(\lambda^4)-G_0)v(\lambda^{-1}S_2A^2_{-1,1})vR^\pm_0(\lambda^4)
+G_0v(\lambda^{-1}S_2A^2_{-1,1})vR^\pm_0(\lambda^4)\\
:=&\Gamma^2_{-1,1}(\lambda)+\Gamma^2_{-1,2}(\lambda).
\end{split}
\end{equation}

Similar to the proof of Proposition \ref{prop-sencond-S2S1-S1S2G0}, we obtain the following proposition.
\begin{proposition}\label{prop-sencondkind-S2-leri}
Assume that $|V(x)|\lesssim (1+|x|)^{-19-}$. Let $\Gamma^2_{-1,j}(\lambda)(j=1,2)$ be  operators defined
in (\ref{secres-imp-termS2-le}). Then
\begin{equation}\label{pro-secondkind-third-term1}
\sup\limits_{x,y\in \mathbb{R}^3}\Big|\int_0^\infty \chi(\lambda)e^{-it\lambda^2}\lambda^{3+2\alpha}\Gamma^2_{-1,1}(\lambda)(x,y)d\lambda\Big|
\lesssim|t|^{-\frac{3+2\alpha}{2}},\, -\frac{3}{2}< \alpha\leq0,
\end{equation}
and
\begin{equation}\label{pro-second-third-term2}
\sup\limits_{x,y\in \mathbb{R}^3}\Big|\int_0^\infty \chi(\lambda)e^{-it\lambda^2}\lambda^{3+2\alpha}\Gamma^2_{-1,2}(\lambda)(x,y)d\lambda\Big|
\lesssim |t|^{-\frac{2+2\alpha}{2}},\, -1<\alpha \leq0.
\end{equation}
\end{proposition}

Notice that the integral estimates \eqref{pro-secondkind-second-term2} and \eqref{pro-second-third-term2} doesn't hold for $\alpha=-1$.
Hence in order to obtain the decay estimate of solution operator $\frac{\sin(t\sqrt{H})}{\sqrt{H}}$, we need compute the following integral,
\begin{equation}\label{esti-total-secondthird-K5}
K_{t,5}(x,y):=\int_0^\infty\chi(\lambda) \sin(t\lambda^2)\lambda\Gamma^2_{-2,2}(\lambda)(x,y)d\lambda,
\end{equation}
\begin{equation}\label{esti-total-secondthird-K6}
K_{t,6}(x,y):=\int_0^\infty\chi(\lambda) \sin(t\lambda^2)\lambda\Gamma^2_{-1,2}(\lambda)(x,y)d\lambda.
\end{equation}
Similar to the proof  of Proposition \ref{prop-sencond-S2S2-G1}, we have the following proposition.
\begin{proposition}\label{prop-sencond-S2S1-S2}
Assume that $|V(x)|\lesssim (1+|x|)^{-19-}$. Let $K_{t,5}(x,y)$ and $K_{t,6}(x,y)$ be integrals defined in \eqref{esti-total-secondthird-K5}
and \eqref{esti-total-secondthird-K6}, respectively. Then
\begin{equation}\label{pro-esti-sin1}
\sup\limits_{x,y\in \mathbb{R}^3}\big |K_{t,i}(x,y)\big|\lesssim 1,\, i=5,6.
\end{equation}
\end{proposition}

\textbf{The proof of Theorem \ref{thm-low} in the second kind resonance case.}
Combining with Proposition \ref{prop-sencond-S2S2-G0}-Proposition \ref{prop-sencond-S2S1-S2} and the proof of the first kind resonance,
we immediately obtain for $-\frac{1}{2}\leq \alpha \leq 0$,
\begin{equation*}
	\big\| H^{\frac{\alpha}{2}}e^{-it\sqrt{H}}P_{ac}(H)\chi(H)\big\|_{L^1\rightarrow L^\infty}\lesssim |t|^{-\frac{1+2\alpha}{2}}.
\end{equation*}
Let $F_t$ be an operator with the integral kernel
$$F_{t}(x,y):= \int_0^\infty \chi(\lambda) e^{-it\lambda^2} \lambda \big[\Gamma^2_{-3,2}(\lambda)+ \Gamma^2_{-3,3}(\lambda)+\Gamma^2_{-3,4}(\lambda)
+\Gamma^2_{-2,2}(\lambda)+\Gamma^2_{-1,2}(\lambda)\big](x,y)d\lambda,$$
where $\Gamma^2_{-3,i}(\lambda)(i=2,3,4), \Gamma^2_{-2,2}(\lambda)$ and  $\Gamma^2_{-1,2}(\lambda)$ are operators defined in \eqref{secres-imp-termS2S2}, \eqref{secres-imp-termS2S1} and \eqref{secres-imp-termS2-le}.
Combining with the integral estimates \eqref{pro-second-bad2}, \eqref{pro-second-bad3}, \eqref{pro-secondkind-second-term2}
and \eqref{pro-second-third-term2}, we have
\begin{equation}\label{esti-Ft1}
\|F_{t} \|_{L^1\rightarrow L^\infty} \lesssim |t|^{-\frac{1}{2}}.
\end{equation}
Moreover, by using Proposition \ref{prop-sencond-S2S2-G0}-Proposition \ref{prop-sencond-S2S1-S2} again, we have
\begin{equation*}
\big\| \cos(t\sqrt{H})P_{ac}(H)\chi(H)-F_{t}\big\|_{L^1\rightarrow L^\infty}\lesssim |t|^{-\frac{3}{2}}.
\end{equation*}

Let $G_t$ be an operator with the integral kernel
$$G_{t}(x,y):= \sum_{i=1}^6 K_{t,i}(x,y),$$
 where $K_{t,i}(x,y)(i=1,\cdots, 6)$ are integrals defined in \eqref{esti--total-K1234},  \eqref{esti-total-secondthird-K5} and
\eqref{esti-total-secondthird-K6}.  By using Proposition \ref{prop-sencond-S2S2-G1} and Proposition \ref{prop-sencond-S2S1-S2} again, we obtain
that
\begin{equation}\label{esti-Gt1}
\|G_{t} \|_{L^1\rightarrow L^\infty} \lesssim |t|^{\frac{1}{2}}.
\end{equation}
Combining with Proposition \ref{prop-sencond-S2S2-G0}-Proposition \ref{prop-sencond-S2S1-S2} again, we immediately obtain
\begin{equation*}
\Big\| \frac{\sin(t\sqrt{H})}{\sqrt{H}}P_{ac}(H)\chi(H)-G_{t}\Big\|_{L^1\rightarrow L^\infty}\lesssim |t|^{-\frac{1}{2}}.
\end{equation*}
Thus the proof of Theorem \ref{thm-low} is completed  in the second kind resonance case.
\cqd

\subsubsection{\textbf{The third kind of resonance}}
If zero is the  third kind of resonance of $H$, then using (\ref{id-RV}) and
(\ref{thm-resoinver-M3 }) one has
\begin{align}\label{Rv-thirdreso}
R^\pm_V(\lambda^4)
=&R^\pm_0(\lambda^4)
-R^\pm_0(\lambda^4)v\Big(\lambda^{-4}S_3D_3S_3\Big)vR^\pm_0(\lambda^4)
-R^\pm_0(\lambda^4)v\Big(\lambda^{-3}S_2A^3_{-3,1}S_2\Big)vR^\pm_0(\lambda^4)\nonumber\\
&-R^\pm_0(\lambda^4)v\Big(  \lambda^{-2}S_2A^3_{-2,1}S_1
+ \lambda^{-2}S_1A^3_{-2,2}S_2\Big)vR^\pm_0(\lambda^4)
-R^\pm_0(\lambda^4)v\Big(  \lambda^{-1}S_2A^3_{-1,1}\nonumber \\
&+ \lambda^{-1}A^3_{-1,2}S_2
+\lambda^{-1}S_1A^3_{-1,3}S_1\Big)vR^\pm_0(\lambda^4)
-R^\pm_0(\lambda^4)v\Big( S_1A^3_{0,1} + A^3_{0,2}S_1\nonumber\\
&+Q A^3_{0,3}Q \Big)vR^\pm_0(\lambda^4)
-R^\pm_0(\lambda^4)v\Gamma_1(\lambda)vR^\pm_0(\lambda^4).
\end{align}
In order to get the estimates in Theorem \ref{thm-low}(iii) in the second kind resonnace, we need to analyse what influence the term $\lambda^{-4}S_3D_3S_3$  has on Stone's formula (\ref{Stone-Paley}). By a simple calculation, we obtain that
 \begin{equation*}
   \begin{split}
&R^+_0(\lambda^4)v\big( \lambda^{-4}S_3D_3S_3 \big)vR^+_0(\lambda^4)
-R^-_0(\lambda^4)v\big( \lambda^{-4}S_3D_3S_3 \big)vR^-_0(\lambda^4)\\
=&\big(R^+_0(\lambda^4)-R_0^-(\lambda^4)\big) v\big( \lambda^{-4}S_3D_3S_3 \big)vR^+_0(\lambda^4)
+R^-_0(\lambda^4)v\big( \lambda^{-4}S_3D_3S_3 \big)v\big(R^+_0(\lambda^4)-R_0^-(\lambda^4)\big).
\end{split}
\end{equation*}

\begin{proposition}\label{prop-third-041}
Let $|V(x)|\lesssim (1+|x|)^{-23-}$.  Then for any $x, y\in \mathbb{R}^3$ and  $-1< \alpha \leq 0$,
\begin{equation*}
 \begin{split}
&\Big|\int_0^\infty \chi(\lambda)e^{-it\lambda^2}\lambda^{3+2\alpha}
\big[\big(R^+_0(\lambda^4)-R_0^-(\lambda^4)\big) v\big( \lambda^{-4}S_3D_3S_3 \big)vR^+_0(\lambda^4)\big](x,y)d\lambda\Big|\lesssim|t|^{-\frac{2+2\alpha}{2}},\\
&\Big|\int_0^\infty \chi(\lambda)e^{-it\lambda^2}\lambda^{3+2\alpha}
\big[R^-_0(\lambda^4)v\big( \lambda^{-4}S_3D_3S_3 \big)v\big(R^+_0(\lambda^4)-R_0^-(\lambda^4)\big)\big](x,y)d\lambda\Big|\lesssim|t|^{-\frac{2+2\alpha}{2}}.
\end{split}
\end{equation*}
\end{proposition}

\begin{proof}
 We only estimate the first integral, similar to get the second integral estimate. Let
\begin{equation}\label{integral disanlei}
   \begin{split}
&K^{3,\pm}_{1,N}(t;x,y)\\
:=&\int_0^\infty e^{-it\lambda^2}\lambda^{3+2\alpha}\varphi_0(2^{-N}\lambda)
\times\big[\big(R^+_0(\lambda^4)-R_0^-(\lambda^4)\big)v\big( \lambda^{-4}S_3D_3S_3 \big)vR^+_0(\lambda^4)\big](x,y)d\lambda.
 \end{split}
\end{equation}
Let
$ \bar{R}_0(p)=\frac{e^{ip}-e^{-ip}}{p}$ and $ F^\pm(p)=\frac{e^{\pm ip}-e^{-p}}{p}.$ Then $R^\pm_0(\lambda^4)= \frac{1}{8\pi \lambda}F^\pm(\lambda|x-y|)$, and
$$R^+_0(\lambda^4)-R^-_0(\lambda^4)= \frac{1}{8\pi \lambda}\bar{R}_0(\lambda|x-y|).$$
Note that $\bar{R}_0\in C^5(\mathbb{R})$ and  $(\bar{R}_0)'(p)=0$, but $(\bar{R}_0)^{(2)}(p)= -\frac{2i}{3}\neq 0$. Let
$$\bar{F}(p)= \bar{R}_0(p)+\frac{2i}{3}p^2, \,\,\bar{F}(p)=\frac{e^{ip}-e^{-ip}}{p}+\frac{ip^2}{3}.$$
Hence, we have $\bar{F}\in C^5(\mathbb{R})$ and $(\bar{F})^{(k)}(p)=0, \,k=1,2$.
By the orthogonality $S_3v=S_3x_iv=S_3x_ix_jv=0(i, j=1,2,3)$, one has
\begin{equation*}
   \begin{split}
&\big[\big(R^+_0(\lambda^4)-R_0^-(\lambda^4)\big) v\big( \lambda^{-4}S_3D_3S_3 \big)vR^+_0(\lambda^4)\big](x,y)\\
=&\frac{1}{64\pi^2 \lambda^6}\int_{\mathbb{R}^6} \bar{R}_0(\lambda |x-u_2|)v(u_2)(S_3D_3S_3)(u_2,u_1)v(u_1)
F^+(\lambda|y-u_1|)du_1du_2\\
=&\frac{1}{64\pi^2 \lambda^6}\int_{\mathbb{R}^6} \big(\bar{R}_0(\lambda |x-u_2|)+\frac{2i}{3}\lambda^2|x-u_2|^2\big)v(u_2)(S_3D_3S_3)(u_2,u_1)v(u_1)\\
&\ \ \ \ \ \ \ \ \, \times F^+(\lambda|y-u_1|)du_1du_2\\
=&\frac{1}{64\pi^2 \lambda^6}\int_{\mathbb{R}^6} \bar{F}(\lambda |x-u_2|)v(u_2)(S_3D_3S_3)(u_2,u_1)v(u_1)
F^+(\lambda|y-u_1|)du_1du_2.
\end{split}
\end{equation*}
By Lemma \ref{Taylor-low}(iii), one has
\begin{equation*}
   \begin{split}
&\big[\big(R^+_0(\lambda^4)-R_0^-(\lambda^4)\big) v\big( \lambda^{-4}S_3D_3S_3 \big)vR^+_0(\lambda^4)\big](x,y)\\
=&-\frac{1}{128\pi^2 \lambda^2}\int_{\mathbb{R}^6}\int_0^1\int_0^1 (1-\theta_2)^2
\Big[\Big(\frac{\bar{F}'( \lambda|x-\theta_2u_2|)}{\lambda^2|x-\theta_2u_2|^2}
-\frac{\bar{F}^{(2)}( \lambda|x-\theta_2u_2|)}{\lambda|x-\theta_2u_2|}\Big)3\cos\alpha_2\sin^2\alpha_2\\
&\ \ \ \ \  \ \ \ \ \ \ \ \ \ \ \ \ \ \ -\bar{F}^{(3)}( \lambda|x-\theta_2u_2|)\cos^3\alpha_2\Big]
(F^+)'(\lambda|y-\theta_1u_1|)\cos\alpha_1 d\theta_1\theta_2\\
&\ \ \ \ \ \ \ \ \ \ \  \ \ \ \ \ \ \ \ \ \ \ \ \times|u_2|^3v(u_2)v(u_1)|u_1|(S_3D_3S_3)(u_2,u_1)du_1du_2.
\end{split}
\end{equation*}
Furthermore, we have
\begin{align}\label{the third1}
&K^{3,\pm}_{1,N}(t;x,y)\nonumber\\
&=-\frac{1}{128\pi^2}\int_{\mathbb{R}^6}\int_0^1\int_0^1
\bigg[\int_0^\infty e^{-it\lambda^2}\lambda^{1+2\alpha}\varphi_0(2^{-N}\lambda)
\Big[\Big(\frac{\bar{F}'( \lambda|x-\theta_2u_2|)}{\lambda^2|x-\theta_2u_2|^2}
-\frac{\bar{F}^{(2)}( \lambda|x-\theta_2u_2|)}{\lambda|x-\theta_2u_2|}\Big)\nonumber\\
&\times 3\cos\alpha_2\sin^2\alpha_2
-\bar{F}^{(3)}( \lambda|x-\theta_2u_2|)\cos^3\alpha_2\Big]
(F^+)'(\lambda|y-\theta_1u_1|)\cos\alpha_1 d\lambda \bigg]
(1-\theta_2)^2 d\theta_1\theta_2\nonumber\\
&\ \ \ \ \ \ \ \ \ \ \ \ \ \ \ \ \ \ \ \ \ \ \ \ \ \ \times|u_2|^3|u_1|v(u_2)v(u_1)(S_3D_3S_3)(u_2,u_1)du_1du_2\nonumber\\
&:= -\frac{1}{128\pi^2}\int_{\mathbb{R}^6}\int_0^1\int_0^1
E^{3,+}_{1,N}(t;x,y,\theta_1,\theta_2,u_1,u_2)
(1-\theta_2)^2 d\theta_1\theta_2
|u_2|^3|u_1|v(u_2)v(u_1)\nonumber\\
&\ \ \ \ \ \ \ \ \ \ \ \ \ \ \ \ \ \ \ \ \ \ \ \ \ \ \ \ \ \ \ \times (S_3D_3S_3)(u_2,u_1)du_1du_2.
\end{align}
Let $s=2^{-N}\lambda$, then
\begin{equation*}
   \begin{split}
&E^{3,+}_{1,N}(t;x,y,\theta_1,\theta_2,u_1,u_2)\\
=&2^{(2+2\alpha)N}\int_0^\infty e^{-it2^{2N}s^2}s^{1+2\alpha}\varphi_0(s)
\Big[\Big(\frac{\bar{F}'(2^Ns|x-\theta_2u_2|)}{(2^Ns)^2|x-\theta_2u_2|^2}
-\frac{\bar{F}^{(2)}( 2^Ns|x-\theta_2u_2|)}{2^Ns|x-\theta_2u_2|}\Big)\\
&\times 3\cos\alpha_2\sin^2\alpha_2 -\bar{F}^{(3)}( 2^Ns|x-\theta_2u_2|)
\cos^3\alpha_2\Big](F^+)'(2^Ns|y-\theta_1u_1|)\cos\alpha_1 ds.
\end{split}
\end{equation*}
Let $r_1=2^N|y-\theta_1u_1|$ and $r_2=2^N|x-\theta_2u_2|$, by using integration by parts we have
\begin{equation*}
   \begin{split}
&\big|E^{3,+}_{1,N}(t;x,y,\theta_1,\theta_2,u_1,u_2)\big|\\
\lesssim &\frac{2^{(2+2\alpha)}N}{1+|t|2^{2N}}\bigg(
\Big|\int_0^\infty e^{-it2^{2N}s^2} \partial_s\big(s^{2\alpha}\varphi_0(s)\big)
\Big(\big(\frac{\bar{F}'(r_2s)}{(r_2s)^2}-\frac{\bar{F}^{(2)}(r_2s)}{r_2s}\big)3\cos\alpha_2\sin^2\alpha_2\\
& -\bar{F}^{(3)}(r_2s)\cos^3\alpha_2\Big)(F^+)'(r_1s)\cos\alpha_1ds\Big|
+\Big|\int_0^\infty e^{-it2^{2N}s^2} s^{2\alpha}\varphi_0(s)\\
&\times \partial_s\Big[\Big(\big(\frac{\bar{F}'(r_2s)}{(r_2s)^2}
-\frac{\bar{F}^{(2)}(r_2s)}{r_2s}\big)
3\cos\alpha_2\sin^2\alpha_2  -\bar{F}^{(3)}(r_2s)\cos^3\alpha_2\Big)(F^+)'(r_1s)\cos\alpha_1 \Big]ds\Big|\bigg)\\
:=&\frac{2^{(2+2\alpha)}N}{1+|t|2^{2N}}\Big(\big|\mathcal{E}^{3,+}_{1,N}(t;x,y,\theta_1,\theta_2,u_1,u_2)\big|
+\big|\mathcal{E}^{3,+}_{2,N}(t;x,y,\theta_1,\theta_2,u_1,u_2)\big|  \Big).
\end{split}
\end{equation*}

For $\big|\mathcal{E}^{3,+}_{2,N}(t;x,y,\theta_1,\theta_2,u_1,u_2)\big|$. Note that
 \begin{equation*}
   \begin{split}
&\partial_s\Big[\Big(\big(\frac{\bar{F}'(r_2s)}{(r_2s)^2}
-\frac{\bar{F}^{(2)}(r_2s)}{r_2s}\big)
3\cos\alpha_2\sin^2\alpha_2-\bar{F}^{(3)}(r_2s)\cos^3\alpha_2\Big)
(F^+)'(r_1s)\cos\alpha_1 \Big]\\
&:=e^{ ir_1s}e^{ir_2s}s^{-1}\bar{F} _{\alpha_1,\alpha_2}(r_1s,r_2s),
\end{split}
\end{equation*}
then we have
\begin{equation*}
   \begin{split}
\big|\mathcal{E}^{3,+}_{2,N}&(t;x,y,\theta_1,\theta_2,u_1,u_2)\big|
\lesssim \Big|\int_0^\infty e^{-it2^{2N}s^2}e^{ i(r_1+r_2)s}s^{-1+2\alpha}\varphi_0(s)
\bar{F}_{\alpha_1,\alpha_2}(r_1s,r_2s)ds\Big|.
\end{split}
\end{equation*}
Since
$$\big|\partial_s^k\bar{F} _{\alpha_1,\alpha_2}(2^Ns|x-\theta_2u_2| ,2^Ns|y-\theta_1u_1|)\big|
 \lesssim1, \,k=0,1, $$
by Lemma \ref{lem-LWP} with $z=(x,y,\theta_1,\theta_2,u_1,u_2)$ and
$$\Psi(z)=r_1+r_2= |x-\theta_2u_2|+|y-\theta_1u_1|,\ \
\Phi(2^Ns,z)= \bar{F} _{\alpha_1,\alpha_2}(2^Ns|x-\theta_2u_2| ,2^Ns|y-\theta_1u_1|),$$
we obtain  that $\mathcal{E}^{3,+}_{2,N}$ is bounded by $ (1+|t|2^{2N})\Theta_{N_0, N}(t)$.
Similar to obtain that $\mathcal{E}^{3,+}_{1,N}$ is controlled by the same bound.
Hence $E^{3,+}_{1,N}$ is bounded by $ 2^{(2+2\alpha)N}\Theta_{N_0, N}(t)$.
 By (\ref{the third1}) and  H\"{o}lder's inequality we obtain that
$ K^{3,\pm}_{1,N}$ is bounded by $2^{(2+2\alpha) N}\Theta_{N_0, N}(t)$.
\end{proof}

In the proof of Proposition \ref{prop-third-041}, for the projection operator $S_3$ in the right hand of the integral (\ref{integral disanlei}), we only use the orthogonality $S_3v=0$. As a result, we don't get a  decay estimate as well as the regular case and the first kind of resonance case. Analogous to the way to deal with the second kind of resonance case, we aim to subtract some specific operators to get the same decay estimate rate as the regular case and the first kind of resonance case. Note that
\begin{align}\label{secres-imp-termS3S3}
& R^+_0(\lambda^4)v(\lambda^{-4}S_3D_3S_3)vR^+_0(\lambda^4)-R^-_0(\lambda^4)v(\lambda^{-4}S_3D_3S_3)vR^-_0(\lambda^4)\nonumber \\
=&\big(R^+_0(\lambda^4)-R^-_0(\lambda^4)\big) v(\lambda^{-4}S_3D_3S_3)vR^+_0(\lambda^4)+
R^-_0(\lambda^4)v(\lambda^{-4}S_3D_3S_3)v\big(R^+_0(\lambda^4)-R^-_0(\lambda^4)\big)\nonumber \\
=&\big(R^+_0(\lambda^4)-R^-_0(\lambda^4)\big) v(\lambda^{-4}S_3D_3S_3)v\big(R^+_0(\lambda^4)-G_0\big)+
(R^+_0(\lambda^4)-R^-_0(\lambda^4)) v(\lambda^{-4}S_3D_3S_3)vG_0\nonumber\\
&+\big(R^-_0(\lambda^4)-G_0\big)v(\lambda^{-4}S_3D_3S_3)v\big(R^+_0(\lambda^4)-R^-_0(\lambda^4)\big)+
G_0v(\lambda^{-4}S_3D_3S_3)v\big(R^+_0(\lambda^4) \nonumber\\
&-R^-_0(\lambda^4)\big):=\Gamma^3_{-4,1}(\lambda)+\Gamma^3_{-4,2}(\lambda)
+\Gamma^3_{-4,3}(\lambda)+\Gamma^3_{-4,4}(\lambda).
\end{align}

Hence, in order to complete the proof of Theorem \ref{thm-low}(iii) in the third kind resonance, combining with the analysis of the second kind of resonance case,
it suffices to show the following proposition.
\begin{proposition}\label{prop-sencond-S3S3-G0}
 Assume that $|V(x)|\lesssim (1+|x|)^{-23-}$. Let $\Gamma^3_{-3,j}(\lambda)(j=1,\cdots,4)$ be operators defined in (\ref{secres-imp-termS3S3}).
 Then
\begin{equation*}
   \begin{split}
&\sup\limits_{x,y\in \mathbb{R}^3}\Big|\int_0^\infty \chi(\lambda)e^{-it\lambda^2}\lambda^{3+2\alpha}\Gamma^3_{-4,j}(\lambda)(x,y)d\lambda\Big|
\lesssim |t|^{-\frac{3+2\alpha}{2}}, -\frac{3}{2}<\alpha\leq0,\, j=1,3,\\
&\sup\limits_{x,y\in \mathbb{R}^3}\Big|\int_0^\infty\chi(\lambda)e^{-it\lambda^2}\lambda^{3+2\alpha}\Gamma^3_{-4,j}(\lambda)(x,y)d\lambda\Big|
\lesssim|t|^{-\frac{2+2\alpha}{2}}, -1<\alpha \leq 0,\,   j=2,4.
\end{split}
\end{equation*}
\end{proposition}
\begin{proof}
By the same argument with the proof of Proposition \ref{prop-sencond-S2S2-G0} and Proposition \ref{prop-third-041}, we  obtain that this proposition holds. Here we omit these details.
\end{proof}
In order to obtain that the estimate of solution operator $\frac{\sin(t\sqrt{H})}{\sqrt{H}}$ in the third resonance, we need compute the following
integrals,
\begin{equation}\label{esti-total-secondthird-K7}
K_{t,7}(x,y):=\int_0^\infty\chi(\lambda) \sin(t\lambda^2)\lambda\big[\Gamma^3_{-4,2}(\lambda)+\Gamma^3_{-4,4}(\lambda)\big](x,y)d\lambda.
\end{equation}

Similar to the proof  of Proposition \ref{prop-sencond-S2S2-G1}, we have the following proposition.
\begin{proposition}\label{prop-sencond-S3S3-sin}
Assume that $|V(x)|\lesssim (1+|x|)^{-23-}$. Let $K_{t,7}(x,y)$ be integral defined in \eqref{esti-total-secondthird-K7}. Then
\begin{equation}\label{pro-esti-sin1}
\sup\limits_{x,y\in \mathbb{R}^3}\big |K_{t,7}(x,y)\big|\lesssim 1.
\end{equation}
\end{proposition}

\textbf{The proof of Theorem \ref{thm-low} in the third kind resonance case.} Let $\widetilde{F}_t$ be an operator with the
following integral kernel
$$\widetilde{F}_{t}(x,y):= F_{t}(x,y)+ \int_0^\infty \chi(\lambda) e^{-it\lambda^2} \lambda \big[\Gamma^3_{-3,2}(\lambda)+ \Gamma^3_{-3,4}(\lambda)\big](x,y)d\lambda.$$
By using the estimate \eqref{esti-Ft1} and Proposition \ref{prop-sencond-S3S3-G0}, we obtain that
$$\|\widetilde{F}_{t} \|_{L^1\rightarrow L^\infty} \lesssim |t|^{-\frac{1}{2}}.$$
Let $\widetilde{G}_{t}(x,y)= \sum_{i=1}^8 K_{t,i}(x,y).$
By using the estimate \eqref{esti-Gt1} and Proposition \ref{prop-sencond-S3S3-sin}, we obtain that
that
$$\|\widetilde{G}_{t} \|_{L^1\rightarrow L^\infty} \lesssim |t|^{\frac{1}{2}}.$$
Combining with Proposition \ref{prop-third-041}-Proposition \ref{prop-sencond-S3S3-sin} and the proof of the second kind resonance again,
we immediately obtain that
\begin{equation*}
\big\| \cos(t\sqrt{H})P_{ac}(H)\chi(H)-\widetilde{F}_{t}\big\|_{L^1\rightarrow L^\infty}\lesssim |t|^{-\frac{3}{2}},
\end{equation*}
\begin{equation*}
\Big\| \frac{\sin(t\sqrt{H})}{\sqrt{H}}P_{ac}(H)\chi(H)-\widetilde{G}_{t}\Big\|_{L^1\rightarrow L^\infty}\lesssim |t|^{-\frac{1}{2}}.
\end{equation*}
Thus the proof of Theorem \ref{thm-low}(iii) is completed.
\cqd

\section{High energy dispersive estimates }

In this subsection, we are devote to establishing  the  decay bounds of Theorem \ref{thm-main-results-regular} and Theorem \ref{thm-main-results-resonance} for high energy. By the identities (\ref{relation-sin-cos}) it suffices to establish high energy dispersive bounds of $H^{\frac{\alpha}{2}}e^{-it\sqrt{H}}P_{ac}(H)$ for $\alpha=-1, 0$. Furthermore,  it suffices to prove the following theorem.
\begin{theorem}\label{thm-high}
Let $|V(x)|\lesssim \langle x \rangle^{-4-}$.  Then
\begin{equation}\label{thm-high-1}
\|H^\frac{\alpha}{2}e^{-it\sqrt{H}}P_{ac}(H)\widetilde{\chi}(H)\|_{L^1\rightarrow L^\infty}\lesssim |t|^{-\frac{3+2\alpha}{2}},\,
\, \alpha \leq 0.
\end{equation}
\end{theorem}

To complete the proof of Theorem \ref{thm-high} , we need to use the following Stone's formula,
\begin{equation}\label{highenergy-formula}
   \begin{split}
H^{\frac{\alpha}{2}}e^{-it\sqrt{H}}P_{ac}(H)\widetilde{\chi}(H)f=\sum_{N=N'+1}^{+\infty}\sum_{\pm}\frac{2}{\pi i}
\int_0^\infty e^{-it\lambda^2}\varphi_0(2^{-N}\lambda)\lambda^{3+2\alpha}
R^\pm_V(\lambda^4)fd\lambda,
  \end{split}
\end{equation}
and the resolvent identity,
\begin{equation}\label{highenergy-reso}
   \begin{split}
R^\pm_V(\lambda^4)=R^\pm_0(\lambda^4)-R^\pm_0(\lambda^4)VR^\pm_0(\lambda^4)
+R^\pm_0(\lambda^4)VR^\pm_V(\lambda^4)VR^\pm_0(\lambda^4).
  \end{split}
\end{equation}
Combining with Proposition \ref{prop-free estimates}, it is enough to prove
the following Proposition \ref{largeenergy-firstterm} and Proposition \ref{largeenergy-secondterm}.
\begin{proposition}\label{largeenergy-firstterm}
Let $|V(x)|\lesssim (1+|x|)^{-3-}$. Then for $\alpha \leq 0$,
\begin{equation*}
\begin{split}
\sup\limits_{x,y\in \mathbb{R}^3}\Big|\int_0^\infty \widetilde{\chi}(\lambda)e^{-it\lambda^2}\lambda^{3+2\alpha}
\big[ R^\pm_0(\lambda^4)VR^\pm_0(\lambda^4) \big](x,y)d\lambda \Big|
 \lesssim |t|^{-\frac{3+2\alpha}{2}}.
\end{split}
\end{equation*}
\end{proposition}
\begin{proof}
We write
\begin{equation*}
   \begin{split}
L^{\pm}_{1,N}(t;x,y):=\int_0^\infty e^{-it\lambda^2}\lambda^{3+2\alpha}\varphi_0(2^{-N}\lambda)
\big[ R^\pm_0(\lambda^4)VR^\pm_0(\lambda^4) \big](x,y)d\lambda.
\end{split}
\end{equation*}
Then
$$\int_0^\infty \widetilde{\chi}(\lambda)e^{-it\lambda^2}\lambda^{3+2\alpha}
\big[ R^\pm_0(\lambda^4)VR^\pm_0(\lambda^4) \big](x,y)d\lambda = \sum_{N=N'+1}^\infty L^{\pm}_{1,N}(t;x,y).$$
Let $ F^\pm(p)=\frac{e^{\pm ip}-e^{-p}}{p}$, then $R_0^\pm(\lambda^4)(x,y)= \frac{1}{8\pi \lambda}F^\pm(\lambda|x-y|).$
Set $s=2^{-N}\lambda$, one has
\begin{equation*}
   \begin{split}
L^{\pm}_{1,N}&(t;x,y)
=\int_{\mathbb{R}^3} \int_0^\infty e^{-it\lambda^2}\lambda^{3+2\alpha}\varphi_0(2^{-N}\lambda)
 R^\pm_0(\lambda^4)(x,u_1)V(u_1)R^\pm_0(\lambda^4)(u_1,y)d\lambda du_1\\
=& \frac{2^{(2+2\alpha) N}}{64\pi^2}\int_{\mathbb{R}^3} \int_0^\infty e^{-it2^{2N}s^2} s^{1+2\alpha} \varphi_0(s)F^\pm(2^Ns|x-u_1|)F^\pm(2^Ns|y-u_1|)dsV(u_1)du_1\\
:=&\frac{1}{64\pi^2}\int_{\mathbb{R}^3} E^{\pm}_{1,N}(t;x,y,u_1) V(u_1)du_1.
\end{split}
\end{equation*}
By using integration by parts, we have
\begin{align}\label{high-Epm1N-12}
\big|E^{\pm}_{1,N}(t;x,y,u_1)\big|
\lesssim & \frac{2^{(2+2\alpha) N}}{1+|t|2^{2N}}\bigg(
\Big|\int_0^\infty e^{-it2^{2N}s^2} \partial_s\big(s^{2\alpha} \varphi_0(s)\big) F^\pm(2^Ns|x-u_1|)F^\pm(2^Ns|y-u_1|)ds\Big|\nonumber\\
&+\Big|\int_0^\infty e^{-it2^{2N}s^2} s^{2\alpha} \varphi_0(s)\partial_s\Big( F^\pm(2^Ns|x-u_1|)F^\pm(2^Ns|y-u_1|)\Big)ds\Big|\bigg)\nonumber\\
:=& \frac{2^{(2+2\alpha) N}}{1+|t|2^{2N}}\Big(|\mathcal{E}^{\pm}_{1,N}(t;x,y,u_1)|+|\mathcal{E}^{\pm}_{2,N}(t;x,y,u_1)|   \Big).
\end{align}
For $\mathcal{E}^{\pm}_{2,N}$. Since
\begin{equation*}
   \begin{split}
&\partial_s\Big( F^\pm(2^Ns|x-u_1|)F^\pm(2^Ns|y-u_1|)\Big)
:=s^{-1}e^{\pm i 2^Ns(|x-u_1|+|y-u_1|)}\\
&\ \ \ \times\Big(F^\pm_{1}(2^Ns|x-u_1|)F^\pm_{0}(2^Ns|y-u_1|)
+F^\pm_{0}(2^Ns|x-u_1|)F^\pm_{1}(2^Ns|y-u_1|)\Big),
\end{split}
\end{equation*}
where
$$F^\pm_{1}(p)=pe^{\mp ip}(F^\pm)'(p)=
\frac{(\pm ip -1)+(p+1)e^{-p\mp ip}}{p}, \ \,F^\pm_{0}(p)=\frac{1-e^{-p\mp ip}}{p}.$$
Then we have
\begin{equation*}
   \begin{split}
|\mathcal{E}^{\pm}_{2,N}(t;x,y,u_1)|
\lesssim &
\Big|\int_0^\infty e^{-it2^{2N}s^2}e^{\pm i 2^Ns(|x-u_1|+|y-u_1|)} s^{-1+2\alpha} \varphi_0(s)\Big(F^\pm_{1}(2^Ns|x-u_1|)\\
&\times F^\pm_{0}(2^Ns|y-u_1|)+F^\pm_{0}(2^Ns|x-u_1|)F^\pm_{1}(2^Ns|y-u_1|)\Big)ds \Big|.
\end{split}
\end{equation*}
Since $N>N_0'$, then
\begin{equation*}
   \begin{split}
|\mathcal{E}^{\pm}_{2,N}(t;x,y)|
\lesssim &
\Big|\int_0^\infty e^{-it2^{2N}s^2}e^{\pm i 2^Ns(|x-u_1|+|y-u_1|)} s^{-1+2\alpha} \varphi_0(s)
\Big(F^\pm_{1}(2^Ns|x-u_1|)\\
&\times F^\pm_{0}(2^Ns|y-u_1|)+F^\pm_{0}(2^Ns|x-u_1|)F^\pm_{1}(2^Ns|y-u_1|)\Big)ds \Big|.
\end{split}
\end{equation*}
Noting that for $k=0,1,$
\begin{equation*}
   \begin{split}
&\big|\partial_s^k\big(F^\pm_{1}(2^Ns|x-u_1|)F^\pm_{0}(2^Ns|y-u_1|)\big)\big|\lesssim 1,\\
&\big|\partial_s^k\big(F^\pm_{0}(2^Ns|x-u_1|)F^\pm_{1}(2^Ns|y-u_1|)\big)\big|\lesssim 1,
\end{split}
\end{equation*}
then by Lemma \ref{lem-LWP} with $z=(x,y,u_1)$, $\Psi(z)=|x-u_1|+|y-u_1|,$  and
$$ \Phi(2^Ns;z)=\Big(F^\pm_{1}(2^Ns|x-u_1|)F^\pm_{0}(2^Ns|y-u_1|)
+F^\pm_{0}(2^Ns|x-u_1|)F^\pm_{1}(2^Ns|y-u_1|)\Big),$$
we obtain that  $\mathcal{E}^{\pm}_{2,N}$ is bounded by $(1+|t|2^{2N})\Theta_{N_0, N}(t)$.
Similarly,  $\mathcal{E}^{\pm}_{1,N}$ is controlled by the same bound. By \eqref{high-Epm1N-12},
we obtain that $E^{\pm}_{1,N}$ is bounded by  $2^{(3+2\alpha)N}\Theta_{N_0, N}(t)$.  Hence we have
\begin{equation*}
   \begin{split}
|L^{\pm}_{1,N}(t;x,y)|
\lesssim2^{(3+2\alpha) N}\Theta_{N_0,N}(t)\int_{\mathbb{R}^3} | V(u_1)|du_1\lesssim2^{(3+2\alpha) N}\Theta_{N_0, N}(t).
\end{split}
\end{equation*}

Finally, by the same summing way  with the proof of \eqref{sum-appha=0-alphabig0}, we immediately obtain the desired conclusion.
\end{proof}

In order to  deal with the term $ R^\pm_0(\lambda^4)VR^\pm_V(\lambda^4)VR^\pm_0(\lambda^4)$, we need to give a lemma
as follows, see \cite{FSY}.
\begin{lemma}\label{lem-largeenergy}
Let $k\geq 0$ and $|V(x)| \lesssim (1+|x|)^{-k-1-}$ such that
$H=\Delta^2 +V$ has no embedded positive eigenvalues. Then for any $\sigma > k+\frac{1}{2}$,
 $R^\pm_V(\lambda)\in \mathcal{B}\big( L^2_\sigma(\mathbb{R}^d), L^2_{-\sigma}(\mathbb{R}^d)\big)$
are $C^k$-continuous for all $\lambda>0$. Furthermore,
$$ \big\|\partial_\lambda R^\pm_V(\lambda)\big\|_{L^2_\sigma(\mathbb{R}^d)\rightarrow  L^2_{-\sigma}(\mathbb{R}^d)} = O\big(|\lambda|^{\frac{-3(k+1)}{4}}\big), k=0,1,\  \hbox{as}\ \lambda \rightarrow +\infty. $$
\end{lemma}
\begin{proposition}\label{largeenergy-secondterm}
Let $|V(x)|\lesssim (1+|x|)^{-4-}$. Then for $\alpha \leq 0$,
\begin{equation}\label{weixiang22}
\begin{split}
\sup\limits_{x,y\in \mathbb{R}^3}\Big|\int_0^\infty \widetilde{\chi}(\lambda)e^{-it\lambda^2}\lambda^{3+2\alpha}
\big[R^\pm_0(\lambda^4)VR^\pm_V(\lambda^4)VR^\pm_0(\lambda^4) \big](x,y)d\lambda \Big|
 \lesssim |t|^{-\frac{3+2\alpha}{2}}.
\end{split}
\end{equation}
\end{proposition}
\begin{proof}
In order to get (\ref{weixiang22}), it's equivalent to prove that the integral
\begin{equation*}
   \begin{split}
L^\pm_{2,N}(t;x,y)
:=\int_0^\infty e^{-it\lambda^2}\lambda^{3+2\alpha}\varphi_0(2^{-N}\lambda)
\Big\langle V R^\pm_V(\lambda^4)V\big(R_0^\pm(\lambda^4)(*,y)\big)(\cdot),
 \big(R_0^\pm(\lambda^4)\big)^*(x,\cdot) \Big\rangle d\lambda
  \end{split}
\end{equation*}
is bounded by  $2^{(3+2\alpha)N}\Theta_{N_0, N}(t)$ for $N>N'$.

In fact, let $F^\pm(p)=\frac{e^{\pm ip}-e^{-p}}{p} $, then $R_0^\pm(\lambda^4)(x,y)= \frac{ 1}{8\pi \lambda} F^\pm(\lambda|x-y|)$.
Hence
\begin{equation*}
   \begin{split}
\big\langle  V & R^\pm_V(\lambda^4)V(R_0^\pm(\lambda^4)(*,y))(\cdot),~ R_0^\mp(\lambda)(x,\cdot)   \big\rangle\\
=&\frac{1}{64\pi^2\lambda^2}\big\langle V R^\pm_V(\lambda^4)V \big(F^\pm(\lambda|*-y|)\big)(\cdot),~
F^\mp(\lambda|x-\cdot|)\big\rangle
:=\frac{1}{64\pi^2\lambda^2}E^{L,\pm}_2(\lambda;x,y).
  \end{split}
\end{equation*}
Let $s=2^{-N}\lambda$, then
\begin{equation*}
   \begin{split}
L^{\pm}_{2,N}(t;x,y):= \frac{2^{(2+2\alpha )N}}{{64\pi^2}}\int_0^\infty e^{-it2^{2N}s^2}s^{1+2\alpha}\varphi_0(s)E^{L,\pm}_2(2^Ns;x,y) ds.
  \end{split}
\end{equation*}
By using integration by parts, we have
\begin{equation}\label{high-esti-Lpm2N-12}
   \begin{split}
|L^{\pm}_{2,N}(t;x,y)|
\lesssim& \frac{2^{(2+2\alpha )N}}{1+|t|2^{2N}}
\bigg( \Big| \int_0^\infty e^{-it2^{2N}s^2} \partial_s\big(s^{2\alpha}\varphi_0(s)\big)
\partial_sE^{L,\pm}_2(2^Ns;x,y) ds\Big|\\
&+ \Big| \int_0^\infty e^{-it2^{2N}s^2}s^{2\alpha}\varphi_0(s)
\partial_s\big(E^{L,\pm}_2(2^Ns;x,y)\big) ds\Big|\bigg)\\
:=& \frac{2^{(2+2\alpha )N}}{1+|t|2^{2N}}\Big( | \mathcal{E}^{L,\pm}_{1,N}(t;x,y)|+|\mathcal{E}^{L,\pm}_{2,N}(t;x,y)| \Big).
  \end{split}
\end{equation}
For $\mathcal{E}^{L,\pm}_{2,N}$. Since
 \begin{equation*}
   \begin{split}
\partial_s\big(E^{L,\pm}_2(2^Ns;x,y)\big)
=& \Big\langle V\partial_s\big(R_V(2^Ns^4)\big) V\big(F^\pm(2^Ns|*-y|)\big)(\cdot),~
F^\mp(2^Ns|x-\cdot|)\Big\rangle\\
&+\Big\langle VR_V(2^Ns^4) V\big(\partial_sF^\pm(2^Ns|*-y|)\big)(\cdot),~
F^\mp(2^Ns|x-\cdot|)\Big\rangle\\
&+\Big\langle VR_V(2^Ns^4) V\big(F^\pm(2^Ns|*-y|)\big)(\cdot),~
\partial_sF^\mp(2^Ns|x-\cdot|)\Big\rangle\\
:=&E^{L,\pm}_{21}(2^Ns;x,y)+E^{L,\pm}_{22}(2^Ns;x,y)
+E^{L,\pm}_{23}(2^Ns;x,y),
  \end{split}
\end{equation*}
then we have
\begin{equation}\label{high-Elpm-123}
   \begin{split}
\mathcal{E}^{L,\pm}_{2,N}(t;x,y)=&
\int_0^\infty e^{-it2^{2N}s^2}s^{2\alpha}\varphi_0(s)
 \big(E^{L,\pm}_{21}+E^{L,\pm}_{22}
 +E^{L,\pm}_{23}\big)(2^Ns;x,y)ds\\
 :=& \mathcal{E}^{\pm,L}_{21,N}(t;x,y) +\mathcal{E}^{\pm,L}_{22,N}(t;x,y)+\mathcal{E}^{\pm,L}_{23,N}(t;x,y).
  \end{split}
\end{equation}

 We first deal with the first term $\mathcal{E}^{\pm,L}_{21,N}(t;x,y)$. Let $\sigma >k+1+\frac{1}{2}$, then
$$   \big|\partial_\lambda E^{L,\pm}_{21}(2^Ns;x,y)\big| \lesssim
\sum\limits_{k=0}^1\big\|V(\cdot)\langle \cdot\rangle^{\sigma}\big\|^2_{L^2}
\big\|\partial_s^{k+1} R^\pm_V(2^{4N}s^4)\big\|_{L^2_\sigma\rightarrow  L^2_{-\sigma}} \lesssim 2^{-N}\lesssim1, \, k=0,1.
$$
Note that $s\in \hbox{supp}\varphi_0 \subset [\frac{1}{4}, 1]$, by using  integration by parts again, we obtain that
\begin{equation*}\label{L1}
   \begin{split}
\big|\mathcal{E}^{\pm,L}_{21,N}(t;x,y)\big|\lesssim &
\frac{1}{1+|t|2^{2N}}\Big|\int_0^\infty e^{-it2^{2N}s^2} \partial_s \Big(s^{-1+2\alpha}\varphi_0(s)E^{L,\pm}_{21}(2^Ns;x,y)\Big)ds\Big|\\
\lesssim &\frac{1}{1+|t| 2^{2N}}.
  \end{split}
\end{equation*}

Next, we deal with the second term $\mathcal{E}^{\pm,L}_{22,N}$. Let
$$F^\mp(p):= e^{\mp ip}F^\mp_0(p), \,\, F^\mp_0(p)= \frac{1-e^{-p}e^{\pm i p}}{p},$$
then
$$ \partial_sF^\pm(2^Ns|*-y|)=2^N|*-y|(F^\pm)'(2^Ns|*-y|)
:= e^{\pm i 2^Ns|*-y|} s^{-1}F^\pm_1(2^Ns|*-y|),$$
where
$$F^\pm_1(p)=pe^{\mp ip}(F^\pm)'(p)=\frac{(\pm ip -1)+(p+1)e^{-p}e^{\mp ip}}{p}.$$
Thus
\begin{equation*}
   \begin{split}
E^{L,\pm}_{22}(2^Ns;x,y)
= &e^{\pm i2^Ns(|x|+|y|)}s^{-1}
\Big\langle VR_V(2^Ns)V\big(e^{\pm i2^Ns(|*-y|-|y|)}F^\pm_1(2^Ns|*-y|)\big)(\cdot),\\
& e^{\mp i2^Ns(|x-\cdot|-|x|)}F^\mp_0(2^Ns|x-\cdot|)\Big\rangle
:= e^{\pm i2^Ns(|x|+|y|)}s^{-1}\widetilde{E}^{L,\pm}_{22}(2^Ns;x,y).
  \end{split}
\end{equation*}
Furthermore,
\begin{equation*}
   \begin{split}
\mathcal{E}^{\pm,L}_{22,N}(t;x,y)
=& \frac{2^{(2+2\alpha) N}}{1+|t|2^{2N}}\int_0^\infty e^{-it2^{2N}s^2}
e^{\pm i 2^Ns(|x|+|y|)}s^{-1+2\alpha}\varphi_0(s)\widetilde{E}^{L,\pm}_{22}(2^Ns;x,y) ds.
  \end{split}
\end{equation*}
Note that
\begin{equation*}
   \begin{split}
&\big|\partial_s^k\big(e^{\pm i2^Ns(|*-y|-|y|)}F^\pm_1(2^Ns|*-y|)\big)\big|\lesssim2^{kN}\langle*\rangle, \, k=0,1,\\
&\big|\partial_s^k\big(e^{\mp i2^Ns(|x-\cdot|-|x|)}F^\mp_0(2^Ns|x-\cdot|)\big)\big|\lesssim2^{kN}\langle\cdot\rangle, \, k=0,1.
\end{split}
\end{equation*}
Since  $|V(x)|\lesssim(1+|x|)^{-4-}$, then by  H\"{o}lder's inequality, we have
 \begin{equation*}
   \begin{split}
 \big|\partial_s \widetilde{E}^{L,\pm}_{22}(2^Ns;x,y)\big|&\lesssim\sum\limits_{k=0}^12^{(1-k)N}\big\|V(\cdot)\langle \cdot\rangle^{\sigma+1-k}\big\|^2_{L^2}
\big\|\partial_s^k R^\pm_V(2^{4N}s^4)\big\|_{L^2_\sigma\rightarrow  L^2_{-\sigma}}
\lesssim 2^{-2N}\lesssim1.
   \end{split}
\end{equation*}
By Lemma \ref{lem-LWP} with $ z=(x,y)$,
$\Psi(z) =|x|+ |y|$ and $\Phi(2^Ns;z)=\widetilde{E}^{L,\pm}_{22,N}(2^Ns;x,y)$,
we get that $\mathcal{E}^{\pm,L}_{22,N}(t;x,y)$ is bounded by  $(1+|t|2^{2N})\Theta_{N_0, N}(t)$. By the same argument we obtain that
$\mathcal{E}^{\pm,L}_{23,N}(t;x,y)$ is bounded by  $(1+|t|2^{2N})\Theta_{N_0, N}(t)$. By \eqref{high-Elpm-123}, we have
$$ |\mathcal{E}^{\pm,L}_{2,N}(t;x,y)| \lesssim \frac{1}{1+|t| 2^{2N}}+ (1+|t|2^{2N})\Theta_{N_0, N}(t)\lesssim
 (1+|t|2^{2N})\Theta_{N_0, N}(t).  $$

Similarly, we obtain that $\mathcal{E}^{\pm,L}_{1,N} $ is bounded by $(1+|t|2^{2N})\Theta_{N_0, N}(t)$. By \eqref{high-esti-Lpm2N-12},
we obtain that $L^{\pm}_{2,N}$ is bounded by $2^{(2+2\alpha)N}\Theta_{N_0, N}(t)$. Thus $L^\pm_{2,N}$ is bounded by  $2^{(3+2\alpha) N}\Theta_{N_0, N}(t)$.

Finally, by the same summing way with the proof of \eqref{sum-appha=0-alphabig0}, we immediately obtain the desired conclusion.
\end{proof}

\section{Appendix  }
In this appendix, we first prove Theorem \ref{thm-main-inver-M}, then we give the characterization of zero resonance subspaces $S_iL^2(\mathbb{R}^3)(i=1,2,3)$ according to the distributional solutions to $H\phi=0$.

\subsection{ The resolvent expansions of $\big(M^\pm(\lambda)\big)^{-1}$ for $\lambda$ near zero}
In this  subsection, we are  devote to showing Theorem \ref{thm-main-inver-M}, i.e. computing the expansions
of $\big(M^\pm(\lambda)\big)^{-1}$  for $\lambda$ near zero case by case, see \cite{EGT19} and also  \cite{FSY} for regular case.

For the convenience, we shall use  notations $A^k_{i,j}$ denote general  $\lambda$-independent absolutely bounded operators on $L^2$, and $O(\lambda^k)$ denote  some $\lambda$-dependent operators $\Gamma(\lambda)$ which satisfy that
\begin{equation*}
\big\|\Gamma(\lambda)\big\|_{L^2\rightarrow L^2}+\lambda\big\|\partial_\lambda \Gamma(\lambda)\big\|_{L^2\rightarrow L^2}
+\lambda^2\big\|\partial^2_\lambda \Gamma(\lambda)\big\|_{L^2\rightarrow L^2}
\lesssim \lambda^k,\  \lambda>0.
\end{equation*}
We  emphasize that these notations of operators $A^k_{i,j}$ and $O(\lambda^k)$ may vary from line to line.

Before computing the expansions of $\big(M^\pm(\lambda)\big)^{-1}$ as $\lambda\rightarrow 0$,
we first state the following lemma used frequently, see  e.g. \cite{JN}.
\begin{lemma}\label{lemma-JN}
Let $A$ be a closed operator and $S$ be a projection. Suppose $A+S$ has a bounded inverse. Then $A$ has
a bounded inverse if and only if
\begin{equation*}		
a:= S-S(A+S)^{-1}S
\end{equation*}
has a bounded inverse in $SH$, and in this case
\begin{equation*}		
A^{-1}= (A+S)^{-1} + (A+S)^{-1}S a^{-1} S(A+S)^{-1}.
\end{equation*}	
\end{lemma}

Now we begin to compute the asymptotic expansions of $\big(M^\pm(\lambda)\big)^{-1}$ when $\lambda$ near zero.

Let $ \displaystyle M^\pm(\lambda)= \frac{\tilde{a}^\pm}{\lambda} \widetilde{M}^\pm(\lambda)$.
By (\ref{Mpm-4}) one has
\begin{equation}\label{id-M-tuta}
   \begin{split}
\widetilde{M}^\pm(\lambda)
 = &P + \frac{\lambda}{\tilde{a}^\pm}T+ \frac{a_1^\pm}{\tilde{a}^\pm}\lambda^2vG_1v
 + \frac{a_3^\pm}{\tilde{a}^\pm}\lambda^4vG_3v \\
 &+ \frac{1}{\tilde{a}^\pm}\lambda^5vG_4v
 +\sum_{k=5}^8 \frac{a_k^\pm}{\tilde{a}^\pm}\lambda^{k+1}vG_kv +O(\lambda^{10}).
   \end{split}
\end{equation}
Thus, it suffices to compute the asymptotic expansions of  $\big(\widetilde{M}^\pm(\lambda)\big)^{-1}$
when $\lambda$ near zero.
Let $Q=I-P$. Since $vG_iv(i=0,1,\cdots,8)$ are bounded operators on $L^2$, then by using (\ref{id-M-tuta}) and  Neumann series expansion one has
\begin{equation}\label{id-MQ}
   \begin{split}
\big(\widetilde{M}^\pm(\lambda)+Q\big)^{-1}
  = I-\sum_{k=1}^9 \lambda^kB_k^\pm+ O(\lambda^{10}),
   \end{split}
\end{equation}
where $B_k^\pm(1\leq k\leq 9)$  are bounded operators in $L^2$ as follows:\\
\begin{equation*}
   \begin{split}
 B_1^\pm=& \frac{1}{\tilde{a}^\pm} T, \,
  B_2^\pm = \frac{a_1^\pm}{\tilde{a}^\pm} vG_1v -\frac{1}{(\tilde{a}^\pm )^2}T^2, \,
  B_3^\pm= -\frac{a_1^\pm}{(\tilde{a}^\pm)^2} (TvG_1v +vG_1vT )
   + \frac{1}{(\tilde{a}^\pm)^3}T^3,\\
B_4^\pm=& \frac{a_3^\pm}{\tilde{a}^\pm}vG_3v
         - \frac{(a_1^\pm)^2}{(\tilde{a}^\pm )^2}(vG_1v)^2
           + \frac{a_1^\pm}{(\tilde{a}^\pm )^3} ( T^2vG_1v+TvG_1vT+vG_1vT^2)
             -\frac{1}{(\tilde{a}^\pm )^4}T^4,\\
 B_5^\pm=& \frac{1}{\tilde{a}^\pm}vG_4v -\frac{a_3^\pm}{(\tilde{a}^\pm)^2}(TvG_3v+vG_3vT )
        + \frac{(a_1^\pm)^2}{(\tilde{a}^\pm )^3}\big(T(vG_1v)^2+(vG_1v)^2T+vG_1vTvG_1v \big)\\
        &\ \ \ \ \  -\frac{a_1^\pm}{(\tilde{a}^\pm )^4}\big( T^3vG_1v +vG_1vT^3 +T^2vG_1vT +TvG_1vT^2 \big)
         +\frac{1}{(\tilde{a}^\pm )^5}T^5.
   \end{split}
\end{equation*}

For the convenience, we list the orthogonality relations of various operators and projections which are used frequently later.
\begin{equation}\label{orthog-relation-1}
   \begin{split}
QD_0&=D_0Q=D_0,\\
 S_iD_j& = D_jS_i =S_i, \, i>j,  \,  S_iD_j = D_jS_i =D_j, \, i \leq j, \,\\
   D_iD_j& = D_jD_i =D_i,\,  i >  j,
   \end{split}
\end{equation}
\begin{equation}\label{orthog-relation-2}
   \begin{split}
& S_2TP=PTS_2=S_2T=TS_2=QvG_1vS_2=S_2vG_1vQ=0,\\
&QTS_1=S_1TQ= 0,\,vG_1vS_3=S_3vG_1v=0, \, S_2vG_3vS_3=S_3vG_3vS_2=0,
\end{split}
\end{equation}
\begin{equation}\label{orthog-relation-3}
   \begin{split}
&B_1^\pm S_2=S_2B_1^\pm=0,\, QB_2^\pm S_2=S_2B_2^\pm Q=B_2^\pm S_3=S_3B_2^\pm =0,\\
&S_2B_3^\pm S_2=B_3^\pm S_3=S_3B_3^\pm=0, \, S_2B_4^\pm S_3=S_3B_4^\pm S_2=0.
\end{split}
\end{equation}
Now we turn to the proof of Theorem \ref{thm-main-inver-M} case by case.
Here we prove only the assertion  for the case $+$ sign, since
the case $-$ sign  proceeds identically. By Lemma \ref{lemma-JN}, we know that $\widetilde{M}^+(\lambda)$ is invertible on $L^2$ if and only if
\begin{equation}\label{id-M1M1tabu}
   \begin{split}
M_1^+(\lambda)
&:=Q-Q(\widetilde{M}^+(\lambda)+Q)^{-1}Q=\frac{\lambda}{\tilde{a}^+}QTQ+\sum_{i=2}^9\lambda^kQB_k^+Q+ O(\lambda^{10})
   \end{split}
\end{equation}
is invertible on $QL^2$.

\textbf{The proof of Theorem \ref{thm-main-inver-M}(i)-(iii).}

 (i) If zero is a regular point of the spectrum of $H$, then $QTQ$ is invertible on $QL^2$. Let
 $D_0=(QTQ)^{-1}$ be an operator on $QL^2$. If $|V(x)|\lesssim (1+|x|)^{-7-}$, by using (\ref{id-M1M1tabu}) one has
$$M_1^+(\lambda)= \frac{\lambda}{\tilde{a}^+}QTQ+O(\lambda^{2})
=\frac{\lambda}{\tilde{a}^+}\Big( QTQ+O(\lambda)\Big)
:= \frac{\lambda}{\tilde{a}^+}\widetilde{M}_1^+(\lambda).$$
By Neumann  series, one has
$$\big(\widetilde{M}_1^+(\lambda)\big)^{-1} = QA^0_{0,1}Q+O(\lambda),$$
thus,
$$\big(M_1^+(\lambda)\big)^{-1} = \lambda^{-1}QA^0_{-1,1}Q+O(1).$$
By using Lemma \ref{lemma-JN}, we have
\begin{equation*}
 \begin{split}\big(\widetilde{M}^+(\lambda)\big)^{-1}=&\big( \widetilde{M}^+(\lambda) +Q\big)^{-1}+ \big( \widetilde{M}^+(\lambda)+Q\big)^{-1}Q\big(M_1^+(\lambda)\big)^{-1}Q \big( \widetilde{M}^+(\lambda) +Q\big)^{-1}\\=&\lambda^{-1}QA^0_{-1,1}Q+ O(1).   \end{split}
\end{equation*}
Since $ \displaystyle \big( M^+(\lambda)\big)^{-1}
= \frac{\lambda}{\tilde{a}^+}\big( \widetilde{M}^+(\lambda)\big)^{-1} $, one has
$$ \big(M^+(\lambda)\big)^{-1}= QA^0_{0,1}Q +O(\lambda). $$

(ii) If zero is the first kind resonance, then by (\ref{id-M1M1tabu}) one has
\begin{equation*}
   \begin{split}
M_1^+(\lambda)=&\frac{\lambda}{\tilde{a}^+}QTQ + \lambda^2QB_2^+Q + \lambda^3 QB_3^+ Q+O(\lambda^{4})\\
=&\frac{\lambda}{\tilde{a}^+}\Big(   QTQ +\widetilde{a}^+\lambda QB_2^+Q + \widetilde{a}^+\lambda^2 QB_3^+Q
+ O(\lambda^{3}) \Big)
:=\frac{\lambda}{\tilde{a}^+} \widetilde{M}_1^+(\lambda).
   \end{split}
\end{equation*}
By the definition of the first kind of  resonance, then $QTQ$ is not invertible on $QL^2$.
Let $S_1$ be the Riesz projection onto the kernel of $QTQ$. Then it is easy to check that $QTQ+ S_1$ is invertible on $QL^2$. In this case, we define $D_0= (QTQ+S_1)^{-1}$ as a bounded operator on $QL^2$.
By Neumann series, one has
\begin{equation}\label{id-M1tuba+S1}
   \begin{split}
\big(\widetilde{M}_1^+(\lambda)+S_1\big)^{-1}
  = D_0- \lambda B_1^0 -\lambda^2B_2^0+ O(\lambda^{3}),
   \end{split}
\end{equation}
where
$B_1^0= \tilde{a}^+D_0B_2^+D_0$ and
$ B_2^0= \widetilde{a}^+D_0B_3^+D_0- (\tilde{a}^+)^2D_0(B_2^+D_0)^2 $.

According to Lemma \ref{lemma-JN}, $ \widetilde{M}^+_1(\lambda)$ has bounded inverse on $QL^2$
if and only if
\begin{equation}\label{H-lambda}
\begin{split}
M_2^+(\lambda):=& S_1- S_1\big(\widetilde{M}^+_1(\lambda) +S_1\big)^{-1}S_1
=\lambda S_1B_1^0S_1+\lambda^2S_1B_2^0S_1 +O(\lambda^{3})
\end{split}
 \end{equation}
has bounded inverse on $S_1L^2$. Note that $S_1T^2S_1=S_1T(P+Q)TS_1=S_1TPTS_1$, then
\begin{equation}\label{T-1}
   \begin{split}
S_1B_1^0S_1=-\frac{1}{\tilde{a}^+}\Big(S_1TPTS_1 -\frac{\|V\|_{L^1}}{3\cdot (8\pi)^2}S_1vG_1vS_1 \Big)
:=-\frac{1}{\tilde{a}^+}T_1.
\end{split}
\end{equation}
Then
\begin{equation}\label{M2-M2tuba}
   \begin{split}
M_2^+(\lambda)
=-\frac{\lambda}{\tilde{a}^+}\Big(  T_1- \tilde{a}^+\lambda S_1B_2^0S_1 + O(\lambda^{2})\Big)
:=-\frac{\lambda}{\tilde{a}^+}\widetilde{M}_2^+(\lambda).
   \end{split}
\end{equation}
By the definition of the first kind resonance,  then $T_1$ is invertible on $S_1L^2$.
Let $D_1=T_1^{-1}$ be an operator on $S_1L^2$, then $S_1D_1=D_1S_1=D_1$.
By using Neumann series, one has
\begin{equation*}
   \begin{split}
\big( \widetilde{M}_2^+(\lambda)\big)^{-1}
=S_1A^1_{0,1}S_1 +\lambda S_1A^1_{1,1}S_1+O(\lambda^{2}).
   \end{split}
\end{equation*}
By (\ref{M2-M2tuba}), then we have
\begin{equation*}
   \begin{split}
\big( M_2^+(\lambda)\big)^{-1}=\lambda^{-1}S_1A^1_{-1,1}S_1 + S_1A^1_{0,1}S_1+O(\lambda).
   \end{split}
\end{equation*}
By using Lemma \ref{lemma-JN}, one has
\begin{equation*}
   \begin{split}
\big(\widetilde{ M}_1^+(\lambda)\big)^{-1}
= \lambda^{-1}S_1A^1_{-1,1}S_1+QA^1_{0,1}Q +O(\lambda).
   \end{split}
\end{equation*}
By the same argument with the proof of the regular case, we obtain that
\begin{equation*}
   \begin{split}
\big(M^+(\lambda)\big)^{-1}
= \lambda^{-1}S_1A^1_{-1,1}S_1 + \big(S_1A^1_{0,1}+A^1_{0,2}S_1+QA^1_{0,3}Q \big)+O(\lambda).
   \end{split}
\end{equation*}


(iii)  If there is a resonance of second kind at zero, then  by (\ref{id-M1M1tabu})  one has
\begin{equation*}
   \begin{split}
M_1^+(\lambda)=&\frac{\lambda}{\tilde{a}^+}QTQ + \sum_{k=2}^7\lambda^k QB_k^+Q +O(\lambda^{8})\\
=&\frac{\lambda}{\tilde{a}^+}\Big(QTQ +\sum_{k=2}^7\widetilde{a}^+\lambda^{k-1}Q B_k^+Q
+O(\lambda^{7}) \Big)
:=\frac{\lambda}{\tilde{a}^+} \widetilde{M}_1^+(\lambda).
   \end{split}
\end{equation*}
By Neumann series, then
\begin{equation*}
   \begin{split}
\big(\widetilde{M}_1^+(\lambda)+S_1\big)^{-1}
  = D_0- \sum_{k=1}^6\lambda^kB_k^0 + O(\lambda^{7}),
   \end{split}
\end{equation*}
where  $B_k^0(k=1,\cdots,6)$ are bounded operators in $QL^2$ as follows:
\begin{equation*}
   \begin{split}
B_1^0=& \tilde{a}^+D_0B_2^+D_0, \,
 B_2^0= \widetilde{a}^+D_0B_3^+D_0- (\tilde{a}^+)^2D_0(B_2^+D_0)^2,\\
 B_3^0=& \widetilde{a}^+D_0B_4^+D_0 -(\widetilde{a}^+)^2( D_0B_2^+D_0B_3^+D_0 + D_0B_3^+D_0B_2^+D_0)
 +(\widetilde{a}^+)^3D_0(B_2^+D_0)^3,\\
 B_4^0=& \widetilde{a}^+D_0B_5^+D_0 -(\widetilde{a}^+)^2\big( D_0B_2^+D_0B_4^+D_0 + D_0B_4^+D_0B_2^+D_0
 +D_0(B_3^+D_0)^2 \big)\\
 &+ (\widetilde{a}^+)^3\big( D_0(B_2^+D_0)^2B_3^+D_0+ D_0B_2^+D_0B_3^+D_0B_2^+D_0
 + D_0B_3^+D_0(B_2^+D_0 )^2 \big)\\
 & -(\widetilde{a}^+)^4D_0(B_2^+D_0)^4.
   \end{split}
\end{equation*}
Furthermore, we obtain the more detail expansion of $M_2^+(\lambda)$ as follows:
\begin{equation*}
   \begin{split}
M_2^+(\lambda)=& -\frac{\lambda}{\tilde{a}^+}T_1+\sum_{k=2}^6\lambda^k S_1B_k^0S_1 +O(\lambda^{7})\\
=&-\frac{\lambda}{\tilde{a}^+}\Big(  T_1- \tilde{a}^+\sum_{k=2}^6\lambda^{k-1} S_1B_k^0S_1
 + O(\lambda^{6})\Big)
:=-\frac{\lambda}{\tilde{a}^+}\widetilde{M}_2^+(\lambda).
   \end{split}
\end{equation*}
By the definition of the second kind resonance of $H$,  then $T_1$ is not invertible on
on $S_1L^2$. Let $S_2$ is the Riesz projection onto the kernel of $T_1$, then
$T_1+S_2$ is invertible on $S_1L^2$. In this case, let $D_1= (T_1+S_2)^{-1}$ be an operator on $S_1L^2$.
By Neumann series, one has
\begin{equation}\label{id-M2tuba+S2}
   \begin{split}
\big(\widetilde{M}_2^+(\lambda)+S_2\big)^{-1}
  = D_1- \sum_{k=1}^5\lambda^kB_k^1 +O(\lambda^{6}),
   \end{split}
\end{equation}
where  $B_k^1(k=1,\cdots,5)$ are bounded operators in $S_1L^2$ as follows:
\begin{equation*}
   \begin{split}
B_1^1=& -\tilde{a}^+D_1B_2^0D_1, \,
B_2^1= -\widetilde{a}^+D_1B_3^0D_1- (\tilde{a}^+)^2D_1(B_2^0D_1)^2,\\
 B_3^1=& -\widetilde{a}^+D_1B_4^0D_1 -(\widetilde{a}^+)^2( D_1B_2^0D_1B_3^0D_1 + D_1B_3^0D_1B_2^0D_1)
 -(\widetilde{a}^+)^3D_1(B_2^0D_1)^3.\\
   \end{split}
\end{equation*}
According to Lemma \ref{lemma-JN}, $ \widetilde{M}^+_2(\lambda)$ has bounded inverse on $S_1L^2$
if and only if
\begin{equation}\label{M3-lambda}
 \begin{split}
M_3^+(\lambda)&:= S_2- S_2\big(\widetilde{M}^+_2(\lambda) +S_2\big)^{-1}S_2
=\sum_{k=1}^5 \lambda^kS_2B_k^1S_2+O(\lambda^{6})
\end{split}
 \end{equation}
has bounded inverse on $S_2L^2$.
Using the orthogonality (\ref{orthog-relation-1})-(\ref{orthog-relation-3}),
we obtain that $S_2B_1^1S_2=0$ and
\begin{equation}\label{T3}
   \begin{split}
 S_2B_2^1S_2
 =& -\tilde{a}^+a_3^+\Big( S_2vG_3vS_2 +\frac{10}{3\|V\|_{L^1}}S_2(vG_1v)^2S_2\\
 &-\frac{10}{3\|V\|_{L^1}} S_2vG_1vTD_1TvG_1vS_2\Big)
 :=-\tilde{a}^+a_3^+T_2.
   \end{split}
\end{equation}
Furthermore, one has
\begin{equation}\label{M3-M3tuba}
   \begin{split}
M_3^+(\lambda)=& -\tilde{a}^+a_3^+ \lambda^2 T_2 + \sum_{k=3}^5 \lambda^k S_2B_k^1S_2 +O(\lambda^{6})\\
=&-\tilde{a}^+a_3^+\lambda^2 \Big(  T_2-
\frac{1}{\tilde{a}^+a_3^+ }\sum_{k=3}^5 \lambda^{k-2} S_2B_k^1S_2+ O(\lambda^{4})\Big)
:=-\tilde{a}^+a_3^+ \lambda^2 \widetilde{M}_3^+(\lambda).
   \end{split}
\end{equation}
By the definition of the second kind resonance of the spectrum of $H$, then $T_2$ is invertible on $S_2L^2$. In this case, we define $D_2 = T_2^{-1}$ as an operator on $S_2L^2$, then $S_2D_2=D_2S_2=D_2$.
 Using Neumann series, one has
\begin{equation*}
   \begin{split}
\big(\widetilde{M}_3^+(\lambda)\big)^{-1}
=&S_2A^2_{0,1}S_2+\lambda S_2A^2_{1,1}S_2 +\lambda^2 S_2A^2_{2,1}S_2 +\lambda^3 S_2A^2_{3,1}S_2 +O(\lambda^{4}).
   \end{split}
\end{equation*}
Moreover,
\begin{equation*}
   \begin{split}
\big(M_3^+(\lambda)\big)^{-1}
=&\lambda^{-2}S_2A^2_{-2,1}S_2+\lambda^{-1}S_2A^2_{-1,1}S_2 + S_2A^2_{0,1}S_2
 +\lambda S_2A^2_{1,1}S_2 +O(\lambda^{2}).
   \end{split}
\end{equation*}
By using Lemma \ref{lemma-JN}, one has
\begin{equation*}
   \begin{split}
\big(\widetilde{ M}_2^+(\lambda)\big)^{-1}
= \lambda^{-2}S_2A^2_{-2,1}S_2+ \lambda^{-1}\big(S_2A^2_{-1,1}S_1 + S_1A^2_{-1,2}S_2\big)
+S_1A^2_{0,1}S_1+\lambda S_1A^2_{1,1}S_1 +O(\lambda^{2}).
   \end{split}
\end{equation*}
By the same argument with the proof of the first kind resonance, we obtain that
\begin{equation*}
   \begin{split}
\big(M^+(\lambda)\big)^{-1}
= &\lambda^{-3}S_2A^2_{-3,1}S_2+ \lambda^{-2}\big(S_2A^2_{-2,1}S_1 + S_1A^2_{-2,2}S_2\big)
+\lambda^{-1}\big(S_2A^2_{-1,1} +A^2_{-1,2}S_2 \\
&+S_1A^2_{-1,3}S_1\big)+\big( S_1A^2_{0,1}+A^2_{0,1}S_1 +QA^2_{0,3}Q\big) +O(\lambda).
   \end{split}
\end{equation*}
\cqd


Before proving  Theorem \ref{thm-main-inver-M}(iv), we give a lemma as follows, see \cite{EGT19}.
\begin{lemma}\label{T3-inve}
 If $V(x)\lesssim(1+|x|)^{-23-}$, then $\hbox{ker} (S_3vG_4vS_3)= \{ 0 \}$. As a result, $ T_3=S_3vG_4vS_3$ is invertible on $S_3L^2$.
\end{lemma}
\textbf{The proof of Theorem \ref{thm-main-inver-M}(iv)}
If there is a resonance of  the third kind at zero, then
by the same argument, we obtain the more detail expansion of
$M_3^+(\lambda)$ as follows:
\begin{equation*}
   \begin{split}
M_3^+(\lambda)=& -\tilde{a}^+a_3^+ \lambda^2 T_2 + \sum_{k=3}^7 \lambda^k S_2B_k^1S_2 +O(\lambda^{8})\\
=&-\tilde{a}^+a_3^+\lambda^2 \Big(  T_2-
\frac{1}{\tilde{a}^+a_3^+ }\sum_{k=3}^7 \lambda^{k-2} S_2B_k^1S_2+ O(\lambda^{6})\Big)
:=-\tilde{a}^+a_3^+ \lambda^2 \widetilde{M}_3^+(\lambda).
   \end{split}
\end{equation*}
By the definition of the third kind resonance of the spectrum of $H$, $T_2$ is not invertible on
on $S_2L^2$. Let $S_3$ is the Riesz projection onto the kernel of $T_2$, then
$T_2+S_3$ is invertible on $S_2L^2$. In this case, let $D_2= (T_2+S_3)^{-1}$ be an operator on $S_2L^2$.
By Neumann series, one has
\begin{equation}\label{id-M3tuba+S3}
   \begin{split}
\big(\widetilde{M}_3^+(\lambda)+S_3\big)^{-1}
  = D_2- \sum_{k=1}^5\lambda^kB_k^2 + O(\lambda^{6}),
   \end{split}
\end{equation}
where  $B_k^2(k=1,\cdots,5)$ are bounded operators in $S_2L^2$, and
$B_1^2= -\frac{1}{\tilde{a}^+a_3^+}D_2B_3^1D_2$.

According to Lemma \ref{lemma-JN}, $ \widetilde{M}^+_3(\lambda)$ has bounded inverse on $S_2L^2$
if and only if
\begin{equation}\label{M4-lambda}
\begin{split}
M_4^+(\lambda)&:= S_3- S_3\big(\widetilde{M}^+_3(\lambda) +S_3\big)^{-1}S_3
 =\sum_{k=1}^5 \lambda^kS_3B_k^2S_3+O(\lambda^{6})
  \end{split}
 \end{equation}
has bounded inverse on $S_3L^2$.
Using orthogonality  (\ref{orthog-relation-1})-(\ref{orthog-relation-3}), we have
\begin{equation}\label{T3}
S_3B_1^2S_3= \frac{1}{a_3^+}S_3vG_4vS_3:=\frac{1}{a_3^+}T_3 .
 \end{equation}
Moreover, we have
\begin{equation}\label{M4-M4tuba}
   \begin{split}
M_4^+(\lambda)=& \frac{\lambda}{a_3^+}T_3+ \sum_{k=2}^5\lambda^k S_3B_k^2S_3+ O(\lambda^{6})\\
   =& \frac{\lambda}{a_3^+}\Big(  T_3+a_3^+\sum_{k=2}^5\lambda^{k-1} S_3B_k^2S_3+ O(\lambda^{5})\Big)
   :=\frac{\lambda}{a_3^+}\widetilde{ M}_4^+(\lambda).
   \end{split}
\end{equation}
By Lemma \ref{T3-inve}, then  $T_3$ is always invertible on $S_3L^2$, let $D_3=T_3^{-1}$ be an operator on
$ S_3L^2$, then $S_3D_3=D_3S_3=D_3$.
By Neumann series, one has
\begin{equation*}
   \begin{split}
\big(\widetilde{M}_4^+(\lambda)\big)^{-1}= S_3D_3S_3+\sum_{k=1}^4\lambda^k S_3A^3_{k,1}S_3+O(\lambda^{5}).
   \end{split}
\end{equation*}
Moreover, we have
\begin{equation*}
   \begin{split}
\big(M_4^+(\lambda)\big)^{-1}= \lambda^{-1}S_3A^3_{-1,1}S_3+ S_3A^3_{0,1}S_3
+\lambda S_3A^3_{1,1}S_3+\lambda^2S_3A^3_{2,1}S_3+\lambda^3S_3A^3_{3,1}S_3+O(\lambda^{4}).
   \end{split}
\end{equation*}
By using Lemma \ref{lemma-JN}, one has
\begin{equation*}
   \begin{split}
\big(\widetilde{ M}_3^+(\lambda)\big)^{-1}
= \lambda^{-1}S_3A^3_{-1,1}S_3+ S_2A^3_{0,1}S_2
+\lambda S_2A^3_{1,1}S_2+\lambda^2S_2A^3_{2,1}S_2+\lambda^3S_2A^3_{3,1}S_2+O(\lambda^{4})
   \end{split}
\end{equation*}
By the same argument with the proof of the second kind resonance, we obtain that
\begin{equation*}
   \begin{split}
\big(M^+(\lambda)\big)^{-1}
= &\lambda^{-4}S_3A^3_{-4,1}S_3+ \lambda^{-3}S_2A^3_{-3,1}S_2
+\lambda^{-2}\big(S_2A^3_{-2,1}S_1+ S_1A^3_{-2,2}S_2\big)\\
&+\lambda^{-1}\big(S_2A^3_{-1,1}+ A^3_{-1,2}S_2 +S_1A^3_{-1,3}S_1\big)
+\big( S_1A^3_{0,1}+A^3_{0,1}S_1 +QA^3_{0,3}Q\big) +O(\lambda).
   \end{split}
\end{equation*}
Here, it is easy check that $A^3_{-4,1}=(S_3vG_4vS_3)^{-1}=D_3$ which doesn't depend on $\pm \hbox{sign}$.

Thus, the proof of Theorem \ref{thm-main-inver-M} is completed.
\cqd

\subsection{The classification of threshold spectral subspaces }
In this subsection, we give the characterizations of the zero resonance subspaces $S_iL^2(\mathbb{R}^3)(i=1,2,3)$
according to the distributional solutions to $H\phi=0$ in the weighted  $L^2 $ space. Such characterizations of resonances  first were obtained in terms of the $L^p$ spaces in  \cite{EGT19}.  Since its proof is similar, we  just list these results.

For $s\in\mathbb{R}$,  we define
 $W_\sigma(\mathbb{R}^3):=\bigcap_{s>\sigma}L^2_{-s}(\mathbb{R}^3),$
which is increasing in $\sigma$ and satisfies $L^2_{-\sigma}(\mathbb{R}^3)\subset W_\sigma(\mathbb{R}^3)$.  In particular, $f\in W_{\frac{3}{2}}(\mathbb{R}^3)$ if $f\in L^\infty(\mathbb{R}^3)$.

\begin{theorem}\label{gongzhengkehua-1}
Assume that $|V(x)|\lesssim(1+|x|)^{-\beta}$ with $\beta>0$.
\begin{itemize}
\item[(i)] If $\beta> 11$, then $f\in S_1L^2\setminus \{0\}$ if and only if $f=Uv\phi$ with $0\neq \phi\in W_\frac{3}{2}$ ( or $\phi\in L^\infty$) such that $H\phi=0$ in the distributional sense, and $\phi=-G_0v f+c_0, $
 where $c_0=\|V\|_{L^1}^{-1}\langle v, Tf\rangle$.

\item[(ii)] If $\beta> 19$, then $f\in S_2L^2\setminus \{0\}$ if and only if $f=Uv\phi$ with $0\neq\phi\in W_\frac{1}{2}$ (or $\phi\in L^p$ for any  $6\le p<\infty$) such that $H\phi=0$ in the distributional sense, and $\phi=-G_0v f. $

\item[(iii)] If $\beta> 23$, then $f\in S_3L^2\setminus \{0\}$ if and only if $f=Uv\phi$ with $0\neq\phi\in L^2$ such that $H\phi=0$ in the distributional sense, and $\phi=-G_0v f $.
\end{itemize}
\end{theorem}


{\bf Acknowledgments:}
The authors would like to thank the anonymous referees for careful reading
 the manuscript and providing valuable suggestions, which substantially improve the quality
 of this paper. Ping Li is partially supported by NSFC (No.12371136) and the Open Research Fund of Key Laboratory of Nonlinear Analysis \& Applications (Central China Normal University), Ministry of Education, P.R. China. Avy Soffer is partially supported by NSF-DMS (No. 2205931) and the Simon's Foundation (No. 395767). Xiaohua Yao are partially supported by NSFC (No.  12171182). We would like to thank Dr. Zijun Wan for her useful discussions.

\end{document}